\documentclass[12pt]{amsart}

\renewcommand{\Re}{\operatorname{Re}}
\renewcommand{\Im}{\operatorname{Im}}

\newcommand{\defeq}{\stackrel{\rm{def}}{=}}

\usepackage{amssymb}
\usepackage{amsthm}
\usepackage{amsxtra}
\usepackage{graphicx}
\usepackage{color}
\usepackage{amsmath}
\usepackage{listings} 
\usepackage{float}
\usepackage{cases} 
\usepackage{threeparttable}
\usepackage{dcolumn}
\usepackage{multirow}
\usepackage{booktabs}
\usepackage{color}
\usepackage{fancyhdr}
\usepackage{graphics}            
\usepackage{graphicx,epsf}
\usepackage{epsfig}

\newtheorem{theorem}{Theorem}[section]

\newtheorem{proposition}{Proposition}[section]
\newtheorem{lemma}[proposition]{Lemma}
\newtheorem{corollary}[proposition]{Corollary}

\theoremstyle{remark}
\newtheorem{remark}[proposition]{Remark}

\numberwithin{equation}{section}

\newcommand{\ds}{\displaystyle}

\title[Blow-up in supercritical NLS]
{Blow-up dynamics in the mass \\ super-critical NLS equations}

\author[Kai Yang]{Kai Yang}
\address{Department of Mathematics  \& Statistics\\Florida International University,  Miami, FL, USA}
\curraddr{}
\email{yangk@fiu.edu}
\thanks{} 

\author[Svetlana Roudenko]{Svetlana Roudenko}
\address{Department of Mathematics \& Statistics\\Florida International University,  Miami, FL, USA}
\curraddr{}
\email{sroudenko@fiu.edu}
\thanks{}

\author[Yanxiang Zhao]{Yanxiang Zhao}
\address{Department of Mathematics\\George Washington University,  Washington, DC, USA}
\curraddr{}
\email{yxzhao@gwu.edu}
\thanks{}



\begin{document}

\begin{abstract}
We study stable blow-up dynamics in the $L^2$-supercritical nonlinear Schr\"{o}dinger equation with radial symmetry in various dimensions. We first investigate the profile equation and extend the result of Wang \cite{Wang1990} and Budd et al. \cite{BCR1999} on the existence and local uniqueness of solutions of the cubic profile equation to other $L^2$-supercritical nonlinearities and dimensions $d \geq 2$. We then numerically observe the multi-bump structure of such solutions, and in particular, exhibit the $Q_{1,0}$ solution, a candidate for the stable blow-up profile.  
Next, using the dynamic rescaling method, we investigate stable blow-up solutions in the $L^2$-supercritical NLS and confirm the square root rate of the blow-up as well as the convergence of blow-up profiles to the $Q_{1,0}$ profile.   
\end{abstract}

\maketitle


\section{Introduction}

Consider the Cauchy problem of the nonlinear Schr\"odinger (NLS) equation
\begin{align}\label{NLS}
\begin{cases}
iu_t+\Delta u + |u|^{p-1}u=0, \qquad (t,x) \in [0,T)\times \mathbb{R}^d \\
u(x,0)=u_0 \in H^s(\mathbb{R}^d), s \geq 1.   
\end{cases}
\end{align}
During their lifespan, the solutions $u(x,t)$ of the Cauchy problem \eqref{NLS} conserve mass and energy (or Hamiltonian):
\begin{align} \label{conservation}
&~  M[u(t)] \defeq \int |u(x,t)|^2 \, dx = M[u_0], \\
&~  E[u(t)] \defeq \frac{1}{2}\int|\nabla u(x,t)|^2 \, dx -\frac{1}{p+1}\int |u(x,t)|^{p+1} \, dx = E[u_0]. 
\end{align}
(The momentum is also conserved, however, we omit it due to the radial setting.)

We are interested in the $L^2$-supercritical case of the equation \eqref{NLS}, for that we recall the scaling index and scaling invariance. If $u(x,t)$ solves \eqref{NLS}, then so does $\tilde{u}(x,t)=\lambda^{\frac{2}{p-1}}u(\lambda x, \lambda^2 t)$. A direct calculation shows that the following $\dot{H}^{s_c}$ norm is invariant under the above scaling, i.e., 
$\|u\|_{\dot{H}^{s_c}}=\| \tilde{u} \|_{\dot{H}^{s_c}}$,
with the critical index defined by
\begin{equation}\label{criticality 0}
s_c=\frac{d}{2}-\frac{2}{p-1}.
\end{equation}
The equation \eqref{NLS} is classified as 
\begin{itemize}
\item $L^2$-subcritical (or mass-subcritical) if $s_c<0$;
\item $L^2$-critical (or mass-critical) if $s_c=0$;
\item intercritical (or mass-supercritical and energy-subcritical) if $0<s_c<1$;
\item $\dot{H}^1$-critical (or energy-critical) if $s_c=1$;
\item $\dot{H}^1$-supercritical (or energy-supercritical) if $s_c>1$.
\end{itemize}

The well-posedness of solutions to the equation \eqref{NLS} has been long investigated starting with works by Ginibre and Velo in \cite{GV1979aa}, see also \cite{Ca2003} and references therein. We discuss separately cases $s_c \leq 1$ and $s_c >1$. When $s_c \leq 1$, the local well-posedness is available in $H^1(\mathbb{R}^d)$, implying that 
for $u_0 \in H^1(\mathbb{R}^d)$ there exists $0<T\leq \infty$ such that there is a unique solution $u(t) \in \mathbb{C}([0,T), H^1(\mathbb{R}^d))$. We say the solution exists globally in time if $T=\infty$, and the solution blows up in finite time if $T < \infty$ and $\limsup\limits_{t\,\nearrow\, T} \|\nabla u(t)\|_{L^2}=\infty$.
When $s_c>1$, taking the Schwartz class initial data $u_0 \in \mathcal{S}(\mathbb{R}^d)$, local well-posedness is obtained in $H^s$ with $s>s_c$. As  discussed in \cite{HR2007}, solutions of the equation \eqref{NLS} may $H^s$-blow up at time $T^*$, that is $\lim_{t \rightarrow T^*} \|u(t)\|_{\dot{H}^s} = \infty$, and persistence of regularity yields that $T^*$ is unique with respect to the norms $H^{\tilde s}$, $\tilde s > s_c$. Hence, while the local $H^1$ theory is absent in this case, there is still a clear distinction between global solutions ($T=\infty$) and finite-time blow-up solutions ($T<\infty$) for $s_c>1$.  

Solutions to the equation \eqref{NLS} when $s_c \geq 0$ may blow-up in finite time. A typical argument to show the existence of such blow-up solutions is the convexity argument for the negative energy initial data ($E[u_0]<0$) with finite variance ($V[u_0]\defeq \int |xu_0|^2 <\infty$). 
When $s_c<0$, solutions to \eqref{NLS} exist globally in time, for example, see \cite{We1989}. The first attempt 
to study stable blow-up solutions was in the mass-critical case ($s_c=0$), where there has been good progress in both analytical and numerical descriptions, see \cite{SS1999}, \cite{Pe2001}, \cite{FP1998}, \cite{FP2000},\cite{FGW2005},\cite{FGW2007}, \cite{F2015}, \cite{RYZ2017}, and references therein. In this case, the dynamics of finite-time  stable blow-up is described as follows: solutions have the ``log-log" blow-up rate
\begin{align*}
\| \nabla u(t) \|_{2} \sim  \left(\dfrac{\ln\ln(\frac{1}{T-t})}{T-t}\right)^{\frac{1}{2}} \quad \mathrm{as} \,\, t \rightarrow T, 
\end{align*}
for example, see \cite{MPSS1986}, \cite{LPSS1988}, \cite{ADKM2003}, \cite{FMR2006}, \cite{MR2005}, \cite{Pe2001b}, \cite{Me1993}, \cite{Ra2005}, \cite{MR2003}, \cite{RYZ2017}; and the blow-up profiles are given by the (unique positive) ground state solution of the nonlinear elliptic equation
\begin{equation}
\label{Q-critical}
-Q+\Delta Q + |Q|^{p-1} Q = 0, \qquad Q \in H^1(\mathbb{R}^d).
\end{equation}

In this paper, we investigate stable blow-up solutions in the mass-supercritical case $s_c>0$.  
To be specific, we consider radial solutions $u(r,t)$ to the equation \eqref{NLS}, $r=|x|, x \in \mathbb{R}^d$. For consistency with previous work (e.g., \cite{MPSS1986}, \cite{LPSS1988}, \cite{SS1999}), we set $\sigma=\frac{p-1}{2}$ and write the nonlinear term as $|u|^{2\sigma} u$; note that in this notation  $s_c = \frac{d}{2}-\frac{1}{\sigma}$. 
Due to the scaling invariance we can rescale solutions to the NLS equation in the radial setting with a scaling function $L(t)$ (e.g., $L(t)=1/\|u(t)\|_{\infty}^{\sigma}$, for other options see \cite{LePSS1987}, \cite{LPSS1988})
\begin{align}\label{rescaling initial}
u(r,t)=\frac{1}{L(t)^{\frac{1}{\sigma}}}\,v(\xi, \tau), \quad \mbox{where} \quad \xi=\frac{r}{L(t)}, \quad \tau=\int_0^t\frac{ds}{L(s)^2}.
\end{align}
Then the equation \eqref{NLS} becomes
\begin{align}\label{DRNLS}
iv_{\tau}+ia(\tau)\left(\xi v_{\xi}+\frac{v}{\sigma}\right)+\Delta v + |v|^{2\sigma}v=0,
\end{align}
where
\begin{align}\label{E:a}
a(\tau)=-L\frac{dL}{dt}=-\frac{d \ln L}{d\tau}.
\end{align}
The second term in \eqref{DRNLS}, containing $a(\tau)$, makes a fundamental difference in the blowup behavior between the mass-critical and supercritical cases, this is due to the limiting behavior of $a(\tau)$. 
As discussed, for example, in \cite{SS1999} or \cite{Za1984}, the stable blow-up dynamics will correspond to $a(\tau) \rightarrow a$, a constant, as the rescaled time $\tau \to \infty$. 
If $a(\tau)$ converges to zero\footnote{For example, in \cite{LPSS1988}, \cite{SS1999}, \cite{RYZ2017} the convergence of $a(\tau)$ is at the rate $1/(\ln\tau+3\ln \ln \tau)$. 
}, then stable singular solutions to the equation \eqref{NLS} blow-up at the square root rate with the ``log-log" correction (and that happens exactly in the $L^2$-critical case, see, \cite{LePSS1987}, \cite{LePSS1988}, \cite{LPSS1988}, \cite{SS1999}, \cite{F2015}, \cite{RYZ2017}). If $a(\tau)$ converges to a non-zero constant, then solutions to \eqref{NLS} blow up at the square root rate without any correction in the leading term. This characterizes the $L^2$-supercritical stable blow-up, see \cite{LPSS1988}, \cite{SS1999}, \cite{F2015}). 
Therefore, understanding the behavior of $a(\tau)$ sheds light onto the {\it rate} of the stable blow-up solutions in the equation \eqref{NLS}.

A more challenging task is to understand and describe the blow-up profiles. 
Following the setting of Zakharov \cite{Za1984} (see also details in \cite{ZK1986}) 
$v(\xi,\tau) = e^{i \tau} Q(\xi)$ and 
considering $a(\tau) \rightarrow a$ as $\tau \rightarrow \infty$, we obtain
\begin{equation}\label{Q-super} 
\Delta_{\xi} Q -Q+ia\left( \dfrac{Q}{\sigma}+\xi Q_{\xi} \right) +|Q|^{2\sigma}Q=0.
\end{equation}
We refer to this equation as the {\it profile} equation. 
The desired solutions of \eqref{Q-super} satisfy 
\begin{equation}\label{Q-BC}
Q_{\xi}(0)=0,\qquad Q(0)=\mathrm{real}, \quad \mbox{and} \quad Q(\infty)=0.
\end{equation}
Accordingly, we investigate the profile equation \eqref{Q-super} with conditions \eqref{Q-BC}. 

Note that in the critical setting ($s_c=0$), the value $a=0$, and thus, the equation \eqref{Q-super} is reduced to \eqref{Q-critical}, the ground state equation. 

As predicted by Zakharov in \cite{Za1984} (see also \cite{LePSS1988}), 
stable blow-up solutions are expected to be of the self-similar form  
\begin{align}\label{u blow-up}
u(x,t) = \dfrac{1}{L(t)^{\frac{1}{\sigma}}} Q\left(\frac{x}{L(t)}\right) \exp \left({i \theta + \frac{i}{2a}\log \frac{T}{T-t}} \right),
\end{align}
where 
\begin{align}\label{L T}
L(t)=(2a(T-t))^{1/2}.
\end{align}


The existence theory for such solutions $Q$ was first shown by X.-P. Wang in his thesis \cite{Wang1990} for the 3d cubic case, he also showed that in the cubic case $Q(x)$ decays as $1/|x|$ when $|x| \to \infty$. 
Later in \cite{BCR1999}, Budd, Chen and Russell, using the Volterra integral equation theory, extended the existence results for the cubic NLS cases in the inter-critical regime, i.e., when $0 <s_c<1$ or $2<d<4$. They also showed the $1/|x|$ decay of such solutions $Q$ in the cubic case. 

The first numerical evidence of monotone decreasing solutions to the profile equation \eqref{Q-super} for the 3d cubic and 2d quintic NLS cases was given in \cite{LePSS1988a}. There it was also shown that the asymptotic behavior of solutions to \eqref{Q-super}--\eqref{Q-BC} (for any $\sigma >0$), in the form of \eqref{u blow-up}, includes for $Q(\xi)$ two linearly independent solutions for large $\xi$:
$$
Q_1 \approx |\xi|^{-\frac{i}{a}-\frac{1}{\sigma}} \quad \mbox{and} \quad Q_2 \approx e^{-\frac{ia\xi^2}{2}}|\xi|^{\frac{i}{a}-d+\frac{1}{\sigma}},
$$
and thus, $Q$ can be written as $Q=\alpha Q_1+\beta Q_2$ when $\xi \rightarrow \infty$ (with $\alpha, \beta \in \mathbb{C}$). 
Substituting the ansatz \eqref{u blow-up} into the formulas for conserved quantities, we obtain the following expressions for the mass and energy (here, $\omega_{d}$ - the surface area of the $d-1$-dimensional unit sphere):
\begin{align}
& M[u(t)]=L(t)^{2s_c} \, \omega_{d} \,\int_0^\infty |Q(\xi)|^2 \xi^{d-1}\,d\xi, \label{E:mass-u}\\
& E[u(t)]=L(t)^{2(s_c-1)} \, \frac{\omega_{d}}2 \, \int_0^{\infty} \left[ |\nabla Q(\xi)|^2 - \frac{1}{\sigma+1} |Q(\xi)|^{2(\sigma+1)} \right] \xi^{d-1} \,d \xi . 
\label{E:energy-u}
\end{align}
{For the profile solutions $Q$ and to simplify the notation (getting rid of $\omega_d$ and the factor $\frac12$ in \eqref{E:energy-u}), we denote the Hamiltonian of $Q$ as $H[Q]$,} or the integral in \eqref{E:energy-u} 
$$
H[Q] \defeq \, \int_0^{\infty} \bigg[ |\nabla Q(\xi)|^2 - \frac{1}{\sigma+1} |Q(\xi)|^{2(\sigma+1)} \bigg] \xi^{d-1} \,d \xi. 
$$
When $0<s_c<1$, the energy of the self-similar blow-up in \eqref{E:energy-u} must be finite, and hence, the Hamiltonian of $Q$ must be zero: $H[Q]=0$ (since $L(t) \to 0$ as $t \to T$). Similarly, the mass of $u $ in \eqref{E:mass-u} remains constant, and thus, the $L^2$ norm of $Q$ is unbounded. In summary, $Q =\alpha Q_1+\beta Q_2$ is in $\dot{H}^1$ but not in $L^2$. Though both $Q_1$ and $Q_2$ vanish at infinity, the fast oscillating span $Q_2$ should be excluded from consideration (or one has to choose only those parameters $Q(0)$ and $a$, which generate the solutions to \eqref{Q-super} with $\beta \equiv 0$). Thus, the next natural step is  to understand the $Q_1$-type solutions, or solution to \eqref{Q-super} with boundary conditions 
\begin{equation}\label{Q-BC2}
Q_{\xi}(0)=0,\qquad Q(0)=\mathrm{real}, \qquad Q(\infty)=0, \qquad H[Q]=0.
\end{equation}

As far as the uniqueness, the {\it local} uniqueness was shown in \cite{BCR1999}: for any given $Q(0) \in \mathbb{R}$ and constant $a>0$, the equation \eqref{Q-super}--\eqref{Q-BC2} has a unique solution. The obvious question then is,
which $Q(0)$ and $a$ produce solutions that would match the profiles of the stable blow-up solutions of \eqref{NLS} from generic initial data.

It should be mentioned that prior to \cite{BCR1999}, Kopell and Landman in \cite{KL1995} constructed such a unique profile $Q$ in the cubic case when the dimension $d$ is exponentially asymptotically close to $2$. (Using this result, Merle, Rapha\"el, and Szeftel in \cite{MRS2010} constructed stable blow-up solutions in the cubic case when $d \gtrapprox 2.$) In \cite{KR-2002} Rottsh\"{a}fer and Kaper improved the construction to include settings where the dimension $d$ is algebraically close to 2, i.e., $d(a)=2+O(a^l)$ for $l>0$. The physical case of the 3d cubic NLS, however, is in no way close to the dimension 2, consequently, the above perturbative approaches (the only ones currently available for $Q$ construction) leave the question of profile(s) $Q$ open.

LeMesurier, Papanicolaou, Sulem and Sulem in \cite{LePSS1988a}, while considering several generic initial conditions and solving the equation numerically via the dynamic rescaling method that they introduced, found that $a(\tau)$ tends to some specific constant, for example, in the 3d cubic NLS $a = 0.917...$; $Q(0)$ has also a specific value, $Q(0) = 1.885...$. They observed that the values of $a$ and $Q(0)$ are very sensitive to perturbations, even $4\%$ deviation would generate a nontrivial perturbation and would not generate a profile with non-oscillating tail. Despite such a sensitivity, it is remarkable that the authors in \cite{LePSS1988a} were able to identify the above parameters with such precision; furthermore, they investigated the case of the 2d quintic NLS and obtained $a = 1.533...$ and $Q(0)= 1.287...$ in that setting.

Budd, Chen and Russell  in \cite{BCR1999} numerically studied the cubic NLS in $2<d<4$, and found that the span of $Q_1$ solutions to the equation \eqref{Q-super}--\eqref{Q-BC2} is large: an infinite (at least a countable) number of distinct self-similar solutions (with monotone decay at infinity, or in other words, with non-oscillatory tails), which are characterized by the number of maxima of $|Q|$ when the dimension $d$ is close to 2, and thus, called `multi-bump' solutions. Among those solutions, only one monotonically decreasing solution, denoted by $Q_{1,0}$, is the actual profile of the {\it stable} blow-up (from the generic initial data) in the cubic NLS equation. It is indeed the only monotone solution when the dimension is asymptotically close to 2. As the dimension increases away from two, there may be other monotone solutions as was demonstrated for the cubic NLS in \cite{BCR1999}, we also show examples of several monotone solutions for other nonlinearities in Section \ref{sec: Qprofile}, see Figures \ref{Q3dp5}-\ref{Q3dp7}. Observe that the profile $Q_{1,0}$ is exactly the one found by LeMesurier, Papanicolaou, Sulem and Sulem in \cite{LePSS1988a}). 
We provide some more details about the 
profile solutions $Q_{J,K}$ in Section \ref{sec: Qprofile} and show examples; for further reading on the structure of profile solutions as well as estimates on the location and value of the maxima, see \cite{BCR1999}, \cite{Budd2002}; for further discussion on numerical treatments refer to \cite{BKW2006}, \cite{BHR2009}, \cite{ADKM2003}, \cite{SS1999}. An attempt to investigate multi-bump solutions analytically, via dynamical systems approach, was done by Rottsh\"{a}fer and Kaper in a very interesting paper \cite{KR-2003}. There, far range asymptotics was glued to the origin, resulting in the so-called midrange part, where non-monotone behavior is possible, and that led to a variety of multi-bump solutions. We also note that the papers \cite{KL1995}, \cite{KR-2002}, \cite{KR-2003} are the only analytical constructions of $Q$ in the mass-supercritical setting.

In this paper, we investigate mass-supercritical cases of the NLS equation, including energy-supercritical cases. We first show the existence and local uniqueness of solutions to the profile equation \eqref{Q-super}--\eqref{Q-BC}, including nonlinearities $p \neq 3$, and review the decay of $Q$ solutions. We then investigate the profile equation \eqref{Q-super} and study the multi-bump solutions $Q_{J,K}$ numerically for the specific powers of nonlinearities $p=3, 5, 7$ (or $\sigma = 1, 2, 3$) in a variety of dimensions (from two to five). Finally, using the dynamic rescaling method, we obtain stable blow-up solutions to the NLS equations and show that the rate of the blow-up is indeed given by \eqref{L T} and the blow-up profiles converge to the specific solution $Q_{1,0}$ of the profile equation \eqref{Q-super} with \eqref{Q-BC2}.

The paper is organized as follows: in Section \ref{S:2}, we discuss the existence and uniqueness theory for solutions of \eqref{Q-super}--\eqref{Q-BC} to all mass-supercritical cases for $d \geq 2$, as well as the decay of $Q$ as $|x| \to \infty$. In Section \ref{sec: Qprofile}, we provide a numerical method to obtain the solutions $Q$ to \eqref{Q-super} \& \eqref{Q-BC}, and in particular, numerically identify the profiles $Q_{1,0}$. We observe no difference in obtaining $Q$ solutions for $s_c<1$, $s_c=1$ and $s_c>1$, though these three cases lead, respectively, to the zero Hamiltonian, constant Hamiltonian and negative Hamiltonian solutions from our numerical observations (the negative Hamiltonian solutions are actually due to the finite interval, we show that the Hamiltonian decreases to negative infinity as the computational domain increases). In Section \ref{S:4}, we simulate the blow-up solutions by the dynamic rescaling method and show the results of convergence of blow-up solutions to the profiles $Q_{1,0}$ and the square root rate of the blow-up. We provide various error estimates 
between the blow-up rate and the predicted rate. All the error quantities are satisfactorily small (e.g., on the order $10^{-5}$). Such behaviors are observed among all considered cases with $0< s_c<1$, $s_c=1$ and $s_c>1$.

{\bf Acknowledgments.} S.R. was partially supported by the NSF CAREER grant DMS-1151618 as well as part of the K.Y.'s graduate research fellowship to work on this project came from the above grant. Y.Z. was partially supported by the Columbian College Facilitating Funds (CCFF 2018) and the Simons Foundation through grant No. 357963.

\section{Existence theory of $Q$}\label{S:2}
\subsection{Existence of $Q$ for general nonlinearity in radial setting.}

We start with the existence theory of the equation \eqref{Q-super}--\eqref{Q-BC}, which is derived from the equation \eqref{DRNLS} as a stationary solution by separation of variables $v(\xi,\tau)=e^{i\tau}Q(\xi)$ (recall $\xi$ and $\tau$ from \eqref{rescaling initial}). Denoting $\displaystyle\lim_{\tau \rightarrow \infty}a(\tau)=a$, we obtain
\begin{align}\label{Q eqn}
\begin{cases}
\Delta_{\xi} Q -Q+ia\left( \dfrac{Q}{\sigma}+\xi Q_{\xi} \right) +|Q|^{2\sigma}Q=0,\\
Q_{\xi}(0)=0,\qquad Q(0)=\mathrm{real}, \quad \mbox{and} \quad Q(\infty)=0.
\end{cases}
\end{align}
Here, $\sigma>0$, and $\Delta_{\xi}:=\partial_{\xi \xi}+\dfrac{d-1}{\xi} \partial_{\xi}$ denotes the Laplacian with the radial symmetry. 

This equation is the key to understanding the profiles for the stable blow-up solutions in the mass-supercritical case. In the cubic case ($\sigma=1$), the existence of solutions to \eqref{Q eqn} was shown in \cite{LPSS1988}, \cite{SS1999}, \cite{Wang1990} and \cite{BCR1999}. Here, we consider mass-supercritical cases with $d\geq 2$ and $p>1$, or $\sigma>0$. First, we discuss some useful properties of solutions, then we address the question of existence by incorporating the approach from \cite{BCR1999}, and close this section with an alternative proof of existence using the method from \cite{Wang1990}, which also gives decay estimates. 

\begin{lemma}\label{L: Q identity}
If $Q(\xi)$ is the solution of the equation (\ref{Q eqn}) in any dimension $d$, then, it satisfies the identities
\begin{align}\label{Q identity 1}
\left|\xi Q_{\xi}+\dfrac{Q}{\sigma}\right|^2 &+ 2\left( d-2-\dfrac{1}{\sigma} \right) \int_0^{\xi} s |Q_{s}|^2  +\left( 2-\dfrac{2}{\sigma}\right)\int_0^{\xi} s|Q|^2 \\
&-\frac{d-2}{\sigma}|Q(0)|^2 -\xi^2|Q|^2+\left( \dfrac{d-2}{\sigma}-\dfrac{1}{\sigma^2} \right)|Q(\xi)|^2 \nonumber\\
&+\dfrac{1}{\sigma+1}\xi^2|Q|^{2\sigma+2}+\dfrac{2}{\sigma (\sigma+1)}\int_0^{\xi}s|Q|^{2\sigma+2}=0,\nonumber
\end{align}
and 
\begin{align}\label{Q identity 2}
2\Im(\xi Q_{\xi}\bar{Q})+2(d-2)\Im \int_0^{\xi}Q_{s}\bar{Q}+2a\left( \dfrac{1}{\sigma}-1 \right)\int_0^{\xi}s|Q|^2+a|\xi|^2|Q|^2=0.
\end{align}
\end{lemma}

\begin{proof}
The identity (\ref{Q identity 1}) is obtained by multiplying the equation in (\ref{Q eqn}) by $2\xi (\xi \bar{Q}_{\xi}+\bar{\frac{Q}{\sigma}})$, taking the real part and then integrating from $0$ to $\xi$.  The identity (\ref{Q identity 2}) is obtained by multiplying the equation in (\ref{Q eqn}) by $2\xi \bar{Q}$, taking the imaginary parts and then integrating from $0$ to $\xi$. 
\end{proof}
In the next lemma we show the boundedness of solutions to \eqref{Q eqn}. Recall $s_c=\frac{d}{2}-\frac{1}{\sigma}$ and $\sigma >0$.

\begin{lemma}\label{L Q bound}
If $s_c>0$, $d > 1+\frac1{\sigma}$, and $Q$ 
is the $C^2$ solution of \eqref{Q eqn}, then $|Q(\xi)|$ is bounded {for all $\xi >0$}.
\end{lemma}
\begin{proof}
First note that a $C^2$ solution to \eqref{Q eqn} is bounded on a finite interval $\xi \in [0,M]$ for any finite $M>0$. 
Therefore, it suffices to show boundedness of $Q$ for $\xi > M$, or as $\xi \rightarrow \infty$. For that we first rewrite the identity \eqref{Q identity 1} as follows
\begin{align}\label{Q identity 11}
&\dfrac{1}{\xi^2}\left| \xi\, Q_{\xi}+\frac{Q}{\sigma} \right| ^2
+|Q|^2\left( \frac{1}{\sigma+1}|Q|^{2\sigma}-1 \right)
+\frac{1}{\xi^2}\left(\frac{d-2}{\sigma}-\frac{1}{\sigma^2} \right)|Q|^2 \\
&\qquad\qquad\qquad 
+ \frac1{\xi^2} \left(2-\frac{2}{\sigma}\right)\int_0^{\xi}s\,|Q|^2
+\frac{2}{\sigma(\sigma+1)\xi^2} \int_0^{\xi}s\,|Q|^{2\sigma+2} \nonumber\\
=&\dfrac{1}{\xi ^2} \left(\frac{d-2}{\sigma}|Q(0)|^2+2\left(2+\frac{1}{\sigma}-d \right)\int_0^{\xi}s|Q_s|^2 \right). \nonumber
\end{align}
Our claim is that if $|Q_{\xi}|$ is bounded, then so is $|Q|$. 
To the contrary, suppose that $|Q|$ is not bounded, that is, $|Q(\xi)| \to \infty$ as $\xi \to \infty$. 
By combining the second and third terms in the left-hand side (LHS) of \eqref{Q identity 11}, we obtain 
$$
|Q|^2 \left[\frac{1}{\sigma+1}|Q|^{2\sigma}-1 + \frac{1}{\xi^2}(\frac{d-2}{\sigma}-\frac{1}{\sigma^2}) \right].
$$ 
Observe that since $|Q| \to \infty$ by our assumption, the first term with $|Q|^{2\sigma}$ in the square brackets is dominant (and grows as $\xi$ grows), and the last term is decreasing as $\xi^{-2}$ regardless of the sign in $\frac{d-2}{\sigma}-\frac{1}{\sigma^2}$, and thus, the whole expression is increasing to infinity as $\xi \to \infty$, so for sufficiently large $\xi$ it will certainly be positive. 
Adding the second line from \eqref{Q identity 11} and comparing it with the entire expression, we obtain for large enough $\xi$
\begin{align}\label{Q boundedness}
0  \leq\  & |Q|^2\left[\frac{1}{\sigma+1}|Q|^{2\sigma}-1
+\frac{1}{\xi^2}\left(\frac{d-2}{\sigma}-\frac{1}{\sigma^2}\right)\right]  \\
&+\frac{1}{\xi^2}\int_0^{\xi}s|Q|^2\left[ \frac{2}{\sigma(\sigma+1)}|Q|^{2\sigma} +\left(2-\frac{2}{\sigma}\right) \right]ds \nonumber\\
 \leq\ & \textrm{LHS of (\ref{Q identity 11})} \nonumber \\
 =\ & \textrm{RHS of (\ref{Q identity 11})}. \nonumber
\end{align}   
Now, the RHS of \eqref{Q identity 11} has two terms: the first term behaves as $\frac{c}{\xi^2}$ and the integral in the second term is on the order of $c \, \xi^2$ as $\xi \to \infty$ (recalling the hypothesis $|Q_{\xi}|$ being bounded). Hence, overall, the RHS is bounded by a constant when $\xi \rightarrow \infty$. This gives a contradiction as the left side in the inequality (\ref{Q boundedness}) grows to $\infty$ as $\xi \rightarrow \infty$, while the RHS remains bounded. Therefore, we conclude that $|Q|$ can not grow to infinity as $\xi \to \infty$ and has to be bounded provided $|Q_{\xi}|$ is bounded.

Next, we show that $|Q_{\xi}|$ is indeed bounded when $\xi \rightarrow \infty$. Again, we prove the boundedness of $|Q_{\xi}|$ by contradiction. 
Suppose $|Q_{\xi}|$ is not bounded, that is $\limsup_{\xi\rightarrow \infty}|Q_{\xi}(\xi)|=\infty$. Then, there exists a monotonically increasing sequence $\lbrace \xi_j\rbrace_0^{\infty}$ for both $\xi_j$ and $Q(\xi_j)$ such that $|Q_{\xi}(\xi_j)|\rightarrow \infty$ as $\xi_j \rightarrow \infty$, and $|Q_{\xi}(\xi_j)|>|Q_{\xi}(\xi_k)|$ for $j>k$. 

For the RHS of the identity \eqref{Q identity 11}, we first consider the case when $2+\frac{1}{\sigma}-d \leq 0$ (or $d \geq 2+\frac1{\sigma}$). Then the RHS is negative. On the other hand, the LHS of \eqref{Q identity 11}, is always positive. We consider separately the case when $Q$ is bounded and when $Q$ is unbounded.
If $Q$ is bounded, we move the ``possible negative terms" $\frac{1}{\xi^2}(\frac{d-2}{\sigma}-\frac{1}{\sigma^2})|Q|^2$ and $|Q|^2(\frac{1}{\sigma+1}|Q|^{2\sigma}-1)$ in (\ref{Q identity 11}) to the RHS and then put the prior bound of $|Q|$ on the terms $\frac{1}{\xi^2}(\frac{d-2}{\sigma}-\frac{1}{\sigma^2})|Q|^2$ and $|Q|^2(\frac{1}{\sigma+1}|Q|^{2\sigma}-1)$. Then, these two terms will be absorbed by a constant $c$ and the RHS will still be negative by choosing sufficiently large $j$. If $Q$ is not bounded, then all the terms in the LHS of the identity (\ref{Q identity 11}) are positive. In both cases, we obtain the identity with strictly positive LHS and strictly negative RHS for some sufficiently large $\xi_j$. Therefore, we reach a contradiction in the case $2+\frac{1}{\sigma}-d \leq 0$ as $\xi \rightarrow \infty$.

Now we consider the case $2+\frac{1}{\sigma}-d > 0$ (or $1+\frac1{\sigma} < d < 2+ \frac1{\sigma}$). In this case, the RHS of the identity (\ref{Q identity 11}) satisfies
\begin{align}\label{Q bound 1}
c+\dfrac{2(2+\frac{1}{\sigma}-d)}{\xi_j^2}\int_0^{\xi_j}s|Q_s(s)|^2ds  & \leq c+\dfrac{2(2+\frac{1}{\sigma}-d)}{\xi_j^2}\int_0^{\xi_j}s|Q_s(\xi_j)|^2ds \\
& =c+\left(2+\frac{1}{\sigma}-d \right)|Q_{\xi}(\xi_j)|^2. \nonumber
\end{align}
Since $d > 1+\frac1{\sigma}$, there exists $\delta > 0$ such that $1-\delta>2+\frac{1}{\sigma}-d$ (that is, $0<\delta<d-1-\frac{1}{\sigma}$). Again we discuss two options for $Q$: $Q$ being bounded or $Q$ being unbounded.
If $Q$ is bounded, we move the ``possible negative terms" $\frac{1}{\xi^2}(\frac{d-2}{\sigma}-\frac{1}{\sigma^2})|Q|^2$ and $|Q|^2(\frac{1}{\sigma+1}|Q|^{2\sigma}-1)$ in (\ref{Q identity 11}) to the RHS and then put the prior bound on $|Q|$. These two terms will be absorbed by a constant $c$. If $Q$ is not bounded, then all terms in the LHS of the identity (\ref{Q identity 11}) are not negative. In both scenarios, we obtain
\begin{align}\label{Q bound 2}
(1-\delta)|Q_{\xi}(\xi_j)|^2 < \textrm{LHS of (\ref{Q identity 11})} \leq  c+(2+\frac{1}{\sigma}-d)|Q_{\xi}(\xi_j)|^2.
\end{align}
Note that $1-\delta>2+\frac{1}{\sigma}-d$, and hence, we reach a contradiction in the inequality (\ref{Q bound 2}) by choosing a sufficiently large $\xi_j$. Therefore, we conclude that $|Q_{\xi}|$ is bounded. Combining the cases of $\xi<M$ and $\xi \rightarrow \infty$, we conclude that $|Q(\xi)|$ is bounded for all $\xi >0$.
\end{proof}

\begin{remark}
The only limitation for extending our results to $d=1$ is the argument on choosing the parameter $\delta > 0$, where we want $1-\delta>2+\frac{1}{\sigma}-d$ to be true. This implies that we need $2+\frac{1}{\sigma}-d<1$, implying $d > 1+\frac1{\sigma}>1$.  
\end{remark}

Before we state our existence and uniqueness results for the equation \eqref{Q eqn}, we recall three lemmas on Volterra integral equations from \cite{Bu1983} on existence, continuity and extension of solutions.

\begin{lemma}[Existence and uniqueness for the Volterra integral equation]\label{L: IE Existence}
Consider the Volterra integral equation
\begin{align}\label{E: Volterra equation}
x(t)=f(t)+\int_0^t g(t,s,x(s))ds.
\end{align} 
Let $a$, $b$ and $L$ be positive numbers, and for some fixed $\alpha \in (0,1)$ we define $c=\alpha/L$. Suppose
\begin{itemize}
\item 
$f$ is continuous on $[0,a]$,
\item 
$g$ is continuous on 
$\displaystyle U=\lbrace  (t,s,x):0 \leq s \leq t \leq a \, \mathrm{and} \, |x-f(t)| \leq b \rbrace$, 
\item 
$g$ is Lipschitz with respect to $x$ on $U$, i.e., $\displaystyle |g(t,s,x)-g(t,s,y)| \leq L|x-y|$, if $(t,s,x), \, (t,s,y) \in U$.
\end{itemize}
If $M=\max_U|g(t,s,x)|$, then there exists a unique solution of \eqref{E: Volterra equation} on $[0,T]$, where $T=\min[a,b/M,c]$.
\end{lemma}

\begin{lemma}[Continuity for the Volterra integral equation]\label{L: IE Continuity}
Let $f:[0,a] \rightarrow \mathbb{R}$ and $g: U \rightarrow \mathbb{R}$ both be continuous, where $\displaystyle U=\lbrace  (t,s,x):0 \leq s \leq t \leq a \, \mathrm{and} \, |x-f(t)| \leq b \rbrace$.
Then there exists a continuous solution of \eqref{E: Volterra equation} on $[0,T]$, where $T=\min[a,b/M]$ and $M=\max_U \vert g(t,s,x) \vert$.
\end{lemma}

\begin{lemma}[Extension of the solution]\label{L: IE Extension}
Let $f:[0,\inf) \rightarrow \mathbb{R}$ and $g: U \rightarrow \mathbb{R}$ be continuous, where $\displaystyle U=\lbrace  (t,s,x):0 \leq s \leq t \leq a \, \mathrm{and} \, |x-f(t)| \leq b \rbrace$. 
If $x(t)$ is a solution of \eqref{E: Volterra equation} on an interval $[0,T)$, then there exists a $\tilde{T}>T$ such that $x(t)$ can be continued to $[0,\tilde{T}]$.
\end{lemma}

\begin{theorem}[Existence and uniqueness of $Q$]\label{T: Q existence}
If $s_c>0$ and $d \geq 2$, for any given initial value $Q(0) \in \mathbb{R}$ and a constant $a>0$, the equation (\ref{Q eqn}) has a unique $C^2$ solution. 
\end{theorem}

\begin{proof}
The equation \eqref{Q eqn} is equivalent to the following Volterra integral equations. For $d=2$:
\begin{align}\label{Q vol1}
\begin{cases}
Q(\xi) = Q(0)-ia\int_0^{\xi} sQ(s)ds \\
\hspace{3cm}+  \int_0^{\xi} \left[ 1+ia(2-\frac{1}{\sigma})-|Q(s)|^{2\sigma}\right]Q(s)\left[s(\ln \xi -\ln s) \right]ds,\\
 Q(\infty)=0.
\end{cases}
\end{align}
For $d>2$:
\begin{align}\label{Q vol2}
\begin{cases}
\ds Q(\xi) = Q(0)-ia\int_0^{\xi} sQ(s)ds\\ 
\hspace{3cm}+  \dfrac{1}{d-2} \int_0^{\xi} \left[ 1+ia(d-\frac{1}{\sigma})-|Q(s)|^{2\sigma}\right]Q(s)\left(s-\frac{s^{d-1}}{\xi^{d-2}} \right)ds,\\
 Q(\infty)=0.
\end{cases}
\end{align}
Both of the equations \eqref{Q vol1} and \eqref{Q vol2} are of the form 
\begin{align}\label{Q form}
Q(\xi)=Q(0)+\int_0^{\xi} g(s,\xi,Q(s))ds, \quad Q(\infty)=0.
\end{align}

By Lemmas \ref{L: IE Existence} and \ref{L: IE Continuity} (arguing similarly to \cite{BCR1999}), one can easily see that 
\eqref{Q form} has a continuous unique solution on the interval $\xi \in [0,M]$ for a fixed $M>0$ as $g(s,\xi,Q(s))$ is continuous. Both existence and uniqueness are extended to $M=\infty$ by Lemma \ref{L: IE Extension} as $|Q(\xi)|$ is bounded on any finite interval $[0,M]$. We next note that $Q$ is the solution not only to the 
equation \eqref{Q vol1}, or \eqref{Q vol2}, but actually to the differential equation \eqref{Q eqn}, and thus, differentiating $Q$ twice classically, it is straightforward to see that $Q$ is of class $C^2$, which finishes the proof.
\end{proof}

\begin{corollary}\label{C: Q decay}
For $d \geq 2$ and $s_c>0$, if $\sigma=1$, then $|Q(\xi)| \lesssim \xi^{-1}$ for $\xi$ large enough (recall that $\xi$ is radial variable here, and thus, non-negative).
\end{corollary}
\begin{proof}
When $\sigma=1$, the term $2a\left(\frac{1}{\sigma}-1 \right) \int_0^{\xi}s|Q|^2ds$ in \eqref{Q identity 2} cancels. Then, the rest of the proof is the same as in \cite[Theorem 2.2]{BCR1999}.
\end{proof}

\subsection{Decay of $Q$ and an alternative proof of existence for the case $s_c=\frac{1}{2}$.}
The argument of Thereom \ref{T: Q existence} does not provide information on the decay rate of $Q$ except for $\sigma = 1$ as in Corollary \ref{C: Q decay} (though the asymptotic analysis in \cite{LePSS1988} showed it should be $|\xi|^{-\frac{1}{\sigma}}$, and the argument for the cubic nonlinearity from \cite{BCR1999} also gave the decay rate $|\xi|^{-1}$ as we showed above). An argument of X.-P. Wang from \cite{Wang1990} allows us to obtain extension of the existence theory of $Q$ to other $s_c>0$ with $\sigma >\frac{1}{2}$ as well as the decay properties, which we state in Theorem \eqref{T: Q Wang}. {We note that this theorem should hold for all $s_c>0$ but for brevity we present the argument for $s_c=\frac{1}{2}$, see also our Remark \ref{R:1} on how to handle other cases.}  This alternative approach implies that $|Q(\xi)| \leq |\xi|^{-\frac{1}{\sigma}}$, the same conclusion as in the asymptotic analysis in \cite{LePSS1988}.   

\begin{theorem}\label{T: Q Wang}
For any $a>0$ satisfying $s_c=\frac{1}{2}$ and $\sigma> \frac{1}{2}$, there exists a global nontrivial solution of \eqref{Q eqn} such that $Q(\xi)=\frac{1}{\xi^{\frac{1}{\sigma}}}R(\xi)$, where $R(\xi)$ is bounded.
\end{theorem}

\begin{proof}
In order to apply the fixed point argument, we first define the Banach space
\begin{align*}
\mathbb{B}=\left\lbrace y(t)\vert y(t)\in C^1[0,\infty), \, |y(t)|, \,  \frac{|y'(t)|}{t+1}  \textrm{ bounded}, \, y(0)=(0,0)^T \right\rbrace
\end{align*}
with the norm
\begin{align*}
\|y(t)\|_{\mathbb{B}}= \max \left\lbrace \sup_{t\in [0,\infty)}|y(t)|, \, \sup_{t\in [0,\infty)} \frac{|y(t)|}{t+1} \right\rbrace.
\end{align*}

We go back to the equation \eqref{Q eqn}. We rescale the solution $Q$ as 
\begin{align}
\tilde{Q}(t)=\dfrac{1}{a^{\frac{1}{2\sigma}}}Q\left( \frac{t}{a^{\frac{1}{2}}}\right).
\end{align}
Then the equation \eqref{Q eqn} becomes
\begin{align}\label{Q rescaled}
\tilde{Q}_{tt}+\dfrac{d-1}{t}\tilde{Q}_t-\dfrac{1}{a}\tilde{Q}+i(t\tilde{Q}_t+\dfrac{\tilde{Q}}{\sigma})+|\tilde{Q}|^{2\sigma}\tilde{Q}=0.
\end{align}
Further substitution 
\begin{align*}
\tilde{Q}(t)=\dfrac{1}{t^{\frac{1}{\sigma}}} e^{-it^2/4} P(t),
\end{align*}
and recalling $s_c=\frac{d}{2}-\frac{1}{\sigma}$, transforms the equation \eqref{Q rescaled} into 
\begin{align} \label{P eqn}
P_{tt}+\dfrac{2s_c-1}{t}P_t+\dfrac{1}{\sigma} \left( 2+\dfrac{1}{\sigma}-d\right) \dfrac{1}{t^2}P-\dfrac{1}{a}P+\dfrac{t^2}{4}P-is_cP+\dfrac{|P|^{2\sigma}P}{t^2}=0.
\end{align}

The term $\frac{2s_c-1}{t}P_t$ is eliminated by taking $s_c=\frac{1}{2}$. By splitting $P=x_1+ix_2$, where $x_1$ and $x_2$ are real and $x= (x_1,x_2)^T$, and rewriting \eqref{P eqn} as a first order system by letting 
$X=(x(t),x_t(t))^T$, we obtain
\begin{align}\label{X eqn}
\begin{cases}
X'=A(t)X+H_y(t)X \\
X(0)=(0,0,b,0)^T,
\end{cases}
\end{align} 
where $b$ is some constant, and
\begin{align}\label{X split}
\begin{split}
& A(t)=\left[ \begin{matrix}
0 & I\\
-a(t)I+\frac{1}{2} J&0
\end{matrix} \right], \qquad 
H_y(t)=\left[ \begin{matrix}
0 & 0 \\
\frac{|x(t)|^{2\sigma}}{t^{2}} I & 0
\end{matrix} \right], \\ 
& a(t)=\frac{t^2}{4}-\frac{1}{a}+\frac{1}{\sigma}\left(\frac{3}{2}-\frac{d}{2}\right)\frac{1}{t^2}, \qquad J=\left[ \begin{matrix}
0&-1\\
1&0
\end{matrix} \right],
\end{split}
\end{align}
and $I$ is the $2 \times 2$ identity matrix.

From \eqref{X split}, we observe that the linear part $A(t)$ is of the same form as in \cite{Wang1990}, except with the extra term $\frac{1}{\sigma}(\frac{3}{2}-\frac{d}{2})\frac{1}{t^2}$ in $a(t)$, which is of the higher order and can be neglected in the analysis. Therefore, the uniform bound on the linear term $A(t)$ is obtained according to the argument in \cite[Section 3 and Section 4, Chapter 2]{Wang1990}.

\begin{remark}\label{R:1}
The statement of Theorem \ref{T: Q Wang} holds for other $s_c\neq \frac{1}{2}$, which can be obtained as follows: if $s_c \neq \frac{1}{2}$, then the matrix $A(t)$ in \eqref{X split} becomes
$$  A(t)=\left[ \begin{matrix}
0 & I\\
-a(t)I+\frac{1}{2} J& \frac{1-2s_c}{\xi} I
\end{matrix} \right].$$
Consequently, the four eigenvalues of the matrix $A(t)$ are
\begin{align*}
& \lambda_1= \sqrt{-a+\frac{i}{2}}-2s_c+1,  \qquad
 \lambda_2= \sqrt{-a-\frac{i}{2}}-2s_c+1, \\
& \lambda_3= -\sqrt{-a+\frac{i}{2}}-2s_c+1,  \qquad
  \lambda_3= -\sqrt{-a-\frac{i}{2}}-2s_c+1.
\end{align*}
Next, after some tedious computations, we get the corresponding eigenfunctions,  then diagonalizing the matrix $A(t)$ yields a corresponding expression to \eqref{X split} and neglecting higher orders finishes this argument in a general case.
\end{remark}

Continuing with the proof of the theorem, we follow the same process as in \cite[Chapter 2]{Wang1990}: it is easy to see that for Lemma 2 in \cite[Section 4, Chapter 2]{Wang1990}, the prior bound for the nonlinear term $\tilde{B}_y(t)$, which comes from a series of linear transformations on $H_y(t)$ (see \cite[Section 3, Chapter 2]{Wang1990}), turns out to be $\|\tilde{B}_y(t)\|_{\mathbb{B}} \leq \frac{1}{t^{\alpha}}$ for large $t$, where $\alpha=\min \lbrace 3,2\sigma +1 \rbrace$. In order to apply the Gronwall's inequality for Lemma 3 in \cite[Section 4, Chapter 2]{Wang1990} for the next step, we need $2\sigma+1<2$. Thus, we need $\sigma>\frac{1}{2}$.

The proof of Theorem \ref{T: Q Wang} is completed by an application of the Schauder fixed point theorem. The rest of the details can be found in \cite[Section 4, Chapter 2]{Wang1990}.
\end{proof}

\begin{remark} \label{E:1}
Note that the above Theorems \ref{T: Q existence} and \ref{T: Q Wang} do not guarantee that for any $a>0$,  the solution $Q$ would be slowly decaying,  i.e., $Q=\alpha Q_1 +\beta Q_2$ with $\beta = 0$,  where $Q_1$ and $Q_2$ are as in \eqref{Q1 and Q2} in Section \ref{sec: Qprofile}. This unique solution $Q$ may include the fast oscillating tails from $Q_2$ part, which is not suitable for us (due to zero Hamiltonian). The main issue is to identify for which $a>0$, solutions $Q$ do not have fast oscillating parts,  which we start addressing in the next section.  
\end{remark}

\section{Profiles $Q$}\label{sec: Qprofile}

Observe that after rescaling \eqref{DRNLS}, we have the mass and energy in terms of $v$ as 
\begin{align}
\label{rescaled mass}& M[v]=L(t)^{2s_c} \int_{\mathbb{R}^d} |v|^2 dx,\\
\label{rescaled energy}& E[v]= L(t)^{2s_c-2} \, \frac12 \,  \int_{\mathbb{R}^d} \left( |\nabla v|^2 - \frac{2}{p+1} |v|^{p+1}  \right)  dx \defeq \frac{\omega_d}2 \, L(t)^{2s_c-2} \, H[v]. 
\end{align}
It was shown in \cite{BCR1999}, see also \cite{KL1995}, that in the cubic NLS case (and $2<d<4$, or equivalently, $0<s_c<1$), the Hamiltonian  $H[Q]=0$ if and only if
\begin{align}\label{Q bc}
\left|\left( \dfrac{1}{\sigma}+\dfrac{i}{a} \right) Q(\xi) + \xi Q_{\xi}(\xi) \right| \rightarrow 0 \quad \mathrm{as} \,\, |\xi| \rightarrow \infty. 
\end{align}
We note that this property is an essential ingredient in numerical study of solutions to \eqref{Q-super}, since it is a good approximation for the boundary condition $Q(\infty)=0$ by taking $\xi=K$ in \eqref{Q bc} for some $K$ large enough (e.g., $K=200$).

We next investigate suitable boundary conditions in the case $s_c \geq 1$. Recall that the asymptotic analysis in \cite{LePSS1987} for the equation (\ref{Q eqn}) shows that one can drop the (higher order) nonlinear term $|Q|^{2\sigma}Q$ in \eqref{Q-super} and obtain $Q$ as a linear span of two linearly independent solutions 
\begin{align}\label{Q1 and Q2}
Q_1 \sim |\xi|^{-\frac{i}{a}-\frac{1}{\sigma}}, \qquad Q_2 \sim e^{-\frac{ia\xi^2}{2}}|\xi|^{\frac{i}{a}-d+\frac{1}{\sigma}},
\end{align} 
and hence, $Q=\alpha Q_1+\beta Q_2$. 
As discussed in the introduction, we exclude the fast oscillating solution $Q_2$, since  only $Q_1$ with the slow decaying tails would produce potential candidates for the stable blow-up profile in the NLS equation (see \cite{LPSS1988}, \cite{SS1999} for further details). 
Note that if $Q \sim Q_1$, then the decay of $Q$ is $|\xi|^{-\frac1{\sigma}}$, as we also show in Theorem \ref{T: Q Wang}. Computing $H[Q]$, we note that both terms do not converge in the energy-supercritical case ({\color{red} the asymptotic behavor is} $\xi^{2(s_c-1)}$ as $\xi \to \infty$) and give some constant when $s_c=1$, see exact computation at the end of this section. 
Therefore, we no longer have the zero Hamiltonian property for $Q$ when $s_c \geq 1$, however, the equation (\ref{Q bc}) is still a good approximation for the boundary condition $Q(\infty)=0$, since the solution $Q$ must be linearly dependent on $Q_1$ when $\xi \rightarrow \infty$; therefore, computing the Wronskian for $Q$ and $Q_1$, gives us 
\begin{align*}
\left( \dfrac{1}{\sigma}+\dfrac{i}{a} \right) Q(\xi) + \xi Q_{\xi}(\xi)=0 ~~~\mbox{as} ~~~ \xi \rightarrow \infty,
\end{align*}
which yields \eqref{Q bc} in all cases $s_c > 0$. Thus, we approximate the boundary condition $Q(\infty)=0$ by 
\begin{align}\label{Q eqn bc2}
\left( \dfrac{1}{\sigma}+\dfrac{i}{a} \right) Q(K) + KQ_{\xi}(K)=0,
\end{align} 
taking sufficiently large $K$.

We are now ready to compute the profiles for the  mass-supercritical cases, including energy-critical and energy-supercritical regimes. We first confirm the results of Budd et al in \cite{BCR1999} (in particular, we show our computations in the 3d cubic case ($d=3$, $\sigma=1$) as the computational consistency check), and then show the results for other  nonlinear powers such as $p=5$ and $p=7$ (or $\sigma=2,3$, correspondingly) and dimensions $d=2, 3, 4, 5$.

\subsection{Computation of $Q$}

We split the solution $Q$ into the real and imaginary parts $Q=P+iW$. Then, the equations (\ref{Q eqn}) and (\ref{Q eqn bc2}) become

\begin{eqnarray}\label{Q compute}
\begin{cases}
\Delta P-P-a(\dfrac{W}{\sigma}+\xi W_{\xi})+(P^2+W^2)^{\sigma}P=0,\\
\Delta W-W+a(\dfrac{P}{\sigma}+\xi P_{\xi})+(P^2+W^2)^{\sigma}W=0,\\
P_{\xi}(0)=0,\\
W(0)=0,\\
W_{\xi}(0)=0,\\
\frac{1}{\sigma}P-\frac{1}{a}W+KP_{\xi}=0,\\
\frac{1}{a}P+\frac{1}{\sigma}W+KW_{\xi}=0.
\end{cases}
\end{eqnarray}

The equation system (\ref{Q compute}) is discretized by the Chebyshev collocation points and differential matrices (see \cite{Tr2000}, \cite{STL2011}). Then, it reduces into the nonlinear algebraic system which can be solved by the matlab solver \texttt{fsolve}. The initial guess for the solver \texttt{fsolve} for such case is obtained by solving the initial value problem of (\ref{Q eqn}) by the matlab solver \texttt{ode45}, with the estimation of the parameters $a$ and $P(0)$. This method requires a relative accurate estimation on the values of $a$ and $P(0)$. However, since we know that $a\approx 0.917$ and $P(0)\approx 1.885$ for $\sigma=1$ from \cite{LPSS1988} and \cite{BCR1999}, the estimation on  $a$ and $P(0)$ for the other cases can be obtained by {\it continuous parameter} search. For example, if we want to compute the 4d cubic case ($d=4$ and $\sigma=1$), we use the values of $a$ and $P(0)$ to compute the initial guess of $d=3.1$, and then compute the solution $Q$ for $d=3.1$. Next, we use these parameters of $a$ and $P(0)$ to compute the case for $d=3.2$, and finally until $d=4$. During our computation, this estimation can be refined by extrapolation, e.g., $a \vert _{d=4}=3a \vert _{d=3.9}-3a \vert _{d=3.8}+a\vert _{d=3.7}$. This refinement allows us to take larger steps on the dimension or nonlinearity and also reduce the iterations in the \texttt{fsolve} in the next stage. 

We use $N=257$ Chebyshev collocation points and take the length of the computational interval $L_D=200$ during our computation. The tolerance on \texttt{fsolve} is set to be $10^{-15}$. For most of the cases, the residue for the algebraic system is on the order of $10^{-12}$.

\subsection{Numerical results on profiles Q}

We first show our results of $Q$ solutions in the 3d cubic case (see Figure \ref{Q3d3p}), {where we plot $|Q|$ (since $Q$ is complex-valued).}  
Our challenge is to 
select the initial condition $Q(0)$ and the parameter $a$ that yields the slowly oscillating type $Q_1$ solutions  of the profile equation. In \cite{BCR1999} it was done by considering the perturbations on the dimension ($d = 2.0001, 2.001, 2.01$, etc) and showing that (at least some) $Q_1$-type solutions of the profile equation can be organized in branches of multi-bump solutions $Q_{K,J}$. 

In these multi-bump solutions $Q_{K,J}$ the index $K = 0,1,2,...$ indicates the branch of the sequence. In the cubic NLS case in \cite{Budd2002} it was shown that solutions of each specific branch $K$ converge non-uniformly as $d \to 2 $ to the $H^1$ solution $Q^{K}$ of the `ground state' equation \eqref{Q-critical} 
(here, $Q^0 \equiv 0$, $Q^1$ is the ground state, $Q^{l}$, $l \geq 2$, are the excited, sign-changing, states). Thus, we are mostly interested in the branch $K=1$, or $\{Q_{1,J} \}$, solutions of which non-uniformly converge to the well-known ground state. Solutions within the branch were classified in \cite{BCR1999} (in the cubic NLS case) by the number of bumps (or maxima) in each sequence, $J=0,1,...$, when the dimension $d$ was close to 2. It was also shown that when $d$ is close to 2, the number $K+J$ indicated the number of turning points. As we increase the dimension, the number of turning points (as well as maxima) change for a fixed $J$, and some profiles of $|Q|$ loose some turning points, and some even become completely monotone (with no bumps), see Figures \ref{Q3d3p}-\ref{Q3dp7}; in any case, here we keep the notation $Q_{K,J}$ for consistency. 

We obtain $Q_{1,0}$ from initial guess for parameters $a$ and $Q(0)$ in the 3d cubic case we extracted from the literature \cite{LPSS1988}, \cite{BCR1999} -- see the blue curve in Figure \ref{Q3d3p}; this confirms matching of our computations with previous results. By using another initial guess for parameters $a$ and $Q(0)$, for example, from \cite{BCR1999}, we obtain the solution $Q_{1,1}$, which is the first bifurcation from $Q_{1,0}$ (see the red dashed curve in Figure \ref{Q3d3p}). 

To really understand the dependence of solutions on parameters $a$ and $Q(0)$, we study the pseudo-phase plane, which was introduced by Kopell-Landman in \cite{KL1995} and adopted in Budd-Chen-Russell \cite{BCR1999}. We write
\begin{align}\label{Q phase}
Q\equiv C(\xi)\exp \left( i\int_0^{\xi} \psi \right), \qquad D(\xi)={C_{\xi}}/{C}\equiv\Re({Q_{\xi}}/{Q}).
\end{align}
In other words, $C$ is the amplitude of $Q$, $C(\xi) = |Q(\xi)|$, $D$ is its logarithmic derivative, and $\psi$ is the gradient of the phase. 
{In the coordinates $(C,D)$ we will track the behavior of the graph as it is decreasing down to the origin as both $C$ and $D$ approach zero when $\xi \to \infty$. To see that recall that asymptotically 
$$
Q(\xi) \sim \alpha \, \xi^{-\frac{1}{\sigma}} \exp \left(-\frac{i}{a} \, \log(\xi) \right) + \beta \, \xi^{-(d-\frac1{\sigma})} \exp\left( -\frac{i\, a \, \xi^2}2  + \frac{i}{a} \log(\xi)\right),  
$$
where the first term is slowly decaying and the second term decays faster with fast oscillations. The solution $Q$ that varies slowly at infinity, would have no oscillations at the end of the curve (as $C \to 0$), since 
$$
C \sim \frac{\alpha}{\xi^{\frac{1}{\sigma}}} \quad \mbox{and} \quad D \sim -\frac1{\sigma \,\xi} \quad \mbox{as} \quad \xi \to \infty.
$$
Thus, such solutions will approach the origin in coordinates $(C,D)$ along the curve 
$D \sim -\frac1{\sigma \, \alpha^\sigma} \, C^{\sigma}$. In the case of $\sigma=1$ (cubic power), this will be a straight line with slope $-1/\alpha$, which we demonstrate in the paths shown in Figure \ref{Q3d3p} (right plot). In the case of $\sigma =2, 3$, this will be a parabola (quadratic or cubic, respectively), which we show in Figures \ref{Q3dp5} and \ref{Q3dp7} (plots on the right).

If the solution $Q$ oscillates fast at infinity, then its graph in the coordinates $(C,D)$ will approach the origin in the oscillating manner, since 
$$
C \sim \frac{\alpha}{\xi^{\frac{1}{\sigma}}} \quad \mbox{and} \quad D \sim
- \frac{\beta \, a}{\alpha} \, \frac1{\xi^{d-\frac2{\sigma}-1}} \, \sin\left(\frac{a \, \xi^2}2 - \frac2{a} \log (\xi)\right).
$$}
We show an example of such oscillating behavior for $\sigma=1$ (cubic nonlinearity in 3d) in Figure \ref{Q3d3p_a_behavior}, where we perturb the value of $a$, while keeping $Q(0)$ fixed (recall $Q(0)=Q_{1,0}(0)= 1.8856569903$): taking  $a=0.8$ (left plot) and $a=1$ (right plot). Both plots show severe oscillations as the curve approaches the origin (recall that for $a \approx 0.9$, or more precise, $a = 0.9173561446$, the curve has no oscillations as shown in Figure \ref{Q3d3p}). 
\begin{figure}[ht]
\begin{center}
\includegraphics[width=0.42\textwidth]{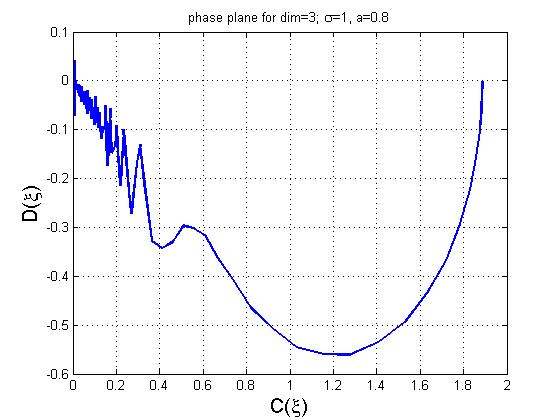}
\includegraphics[width=0.42\textwidth]{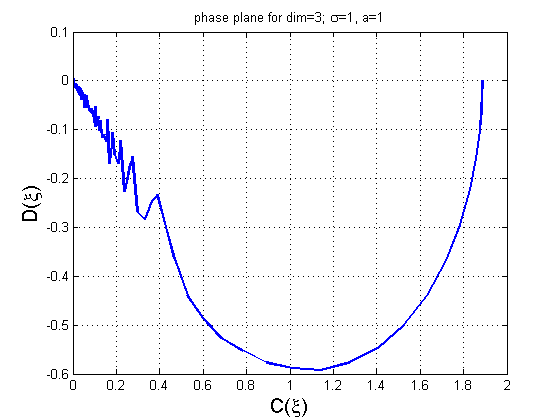}
\caption{ $Q$ solutions in the coordinates $(C,D)$ for different values of parameter $a$ with fixed value of $Q(0)\equiv Q_{1,0}(0)$. Left: $a=0.8$. Right: $a=1$. Both solutions show fast oscillating behavior when approaching the origin, thus, not suitable candidates for the blow-up profiles (the value of $a$ with no oscillations is close to $0.9$).}
\label{Q3d3p_a_behavior}
\end{center}
\end{figure}
We note that our results in Figure \ref{Q3d3p} match the ones obtained in \cite{BCR1999} both in the amplitude and phase plane representations, therefore, we conclude that our numerical approach is trustful.

\begin{figure}
\begin{center}
\includegraphics[width=0.42\textwidth]{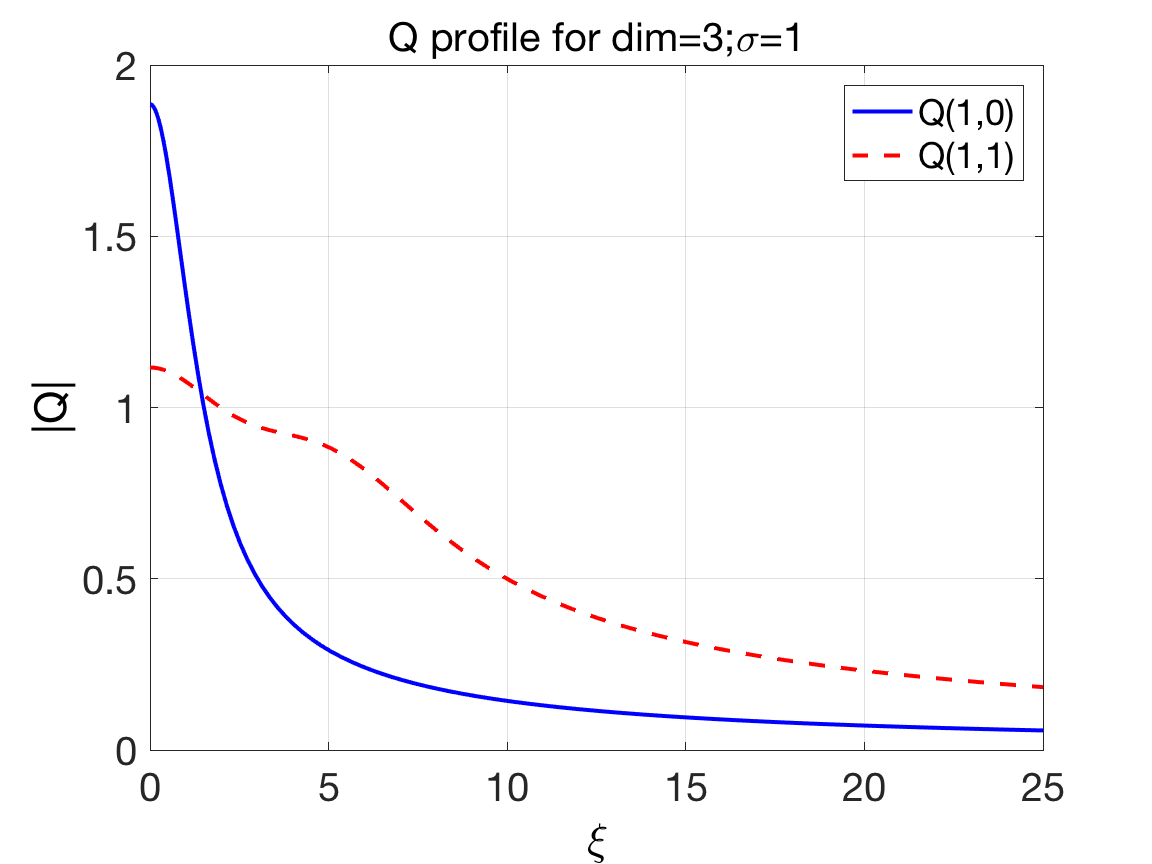}
\includegraphics[width=0.42\textwidth]{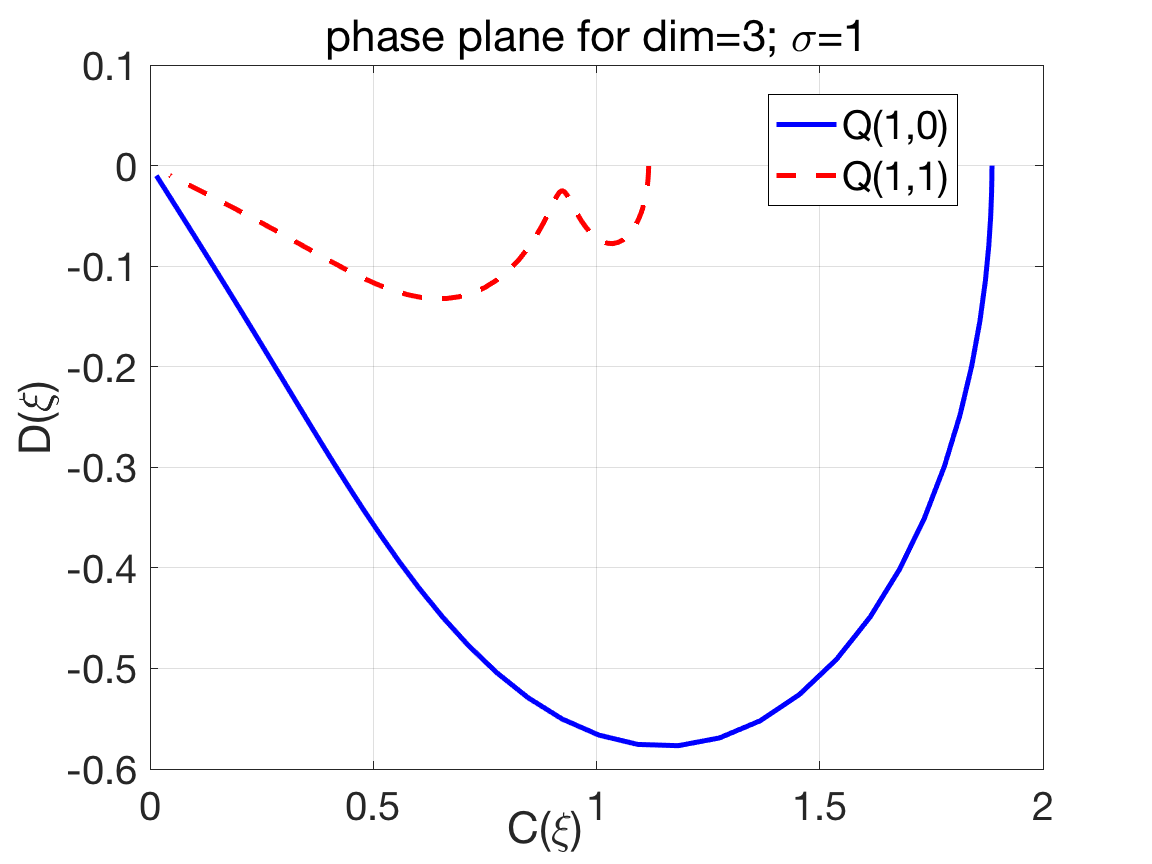}
\caption{ $Q$ profiles (we plot $|Q|$). Left: The blue curve is the first solution $Q_{1,0}$ in the branch $Q_{1,J}$ and the red one is the first bifurcation solution $Q_{1,1}$. Right: The phase plane for the 3d cubic case in coordinates $(C,D)$.} 
\label{Q3d3p}
\end{center}
\end{figure}

In Figures \ref{Q3dp5}--\ref{Q3dp7} we show $|Q|$ profiles for the 3d quintic ($d=3$ and $\sigma=2$) and 3d septic ($d=3$ and $\sigma=3$) cases , which are the energy-critical and energy-supercritical cases, respectively. 

\begin{figure}
\begin{center}
\includegraphics[width=0.42\textwidth]{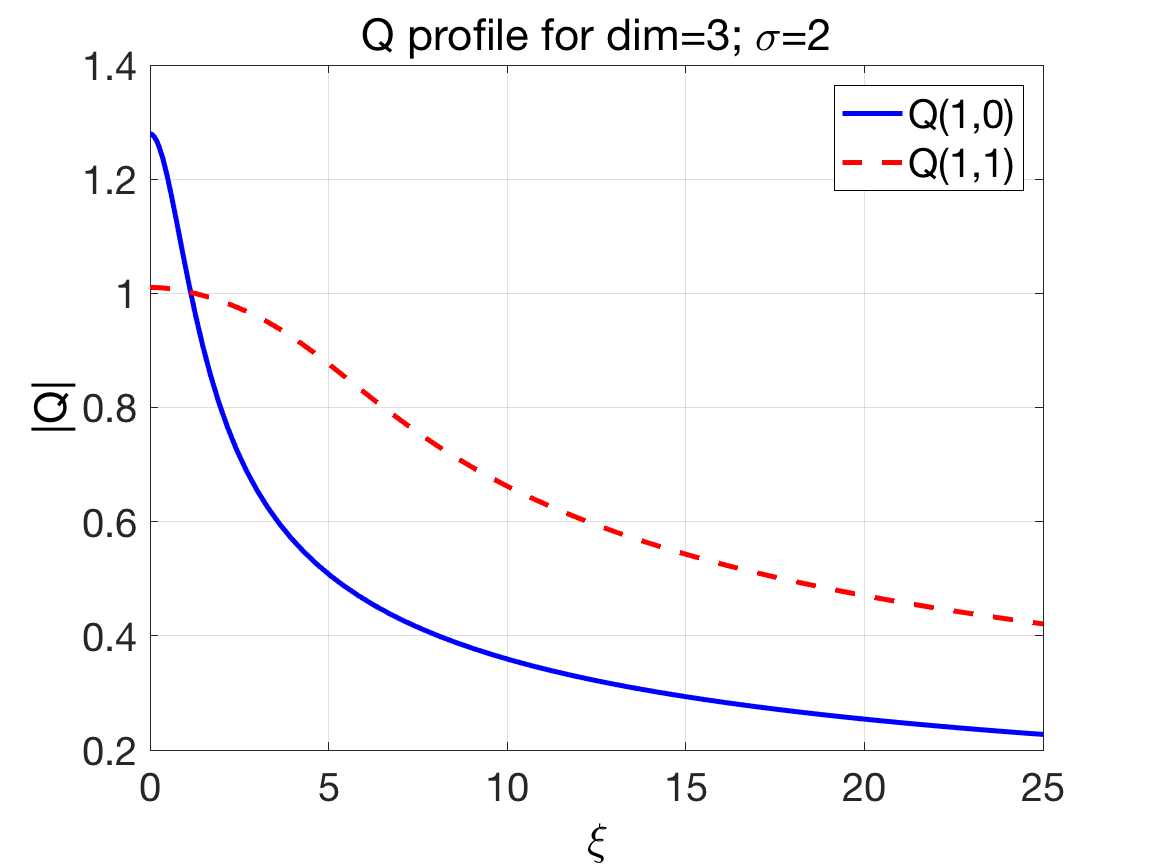}
\includegraphics[width=0.42\textwidth]{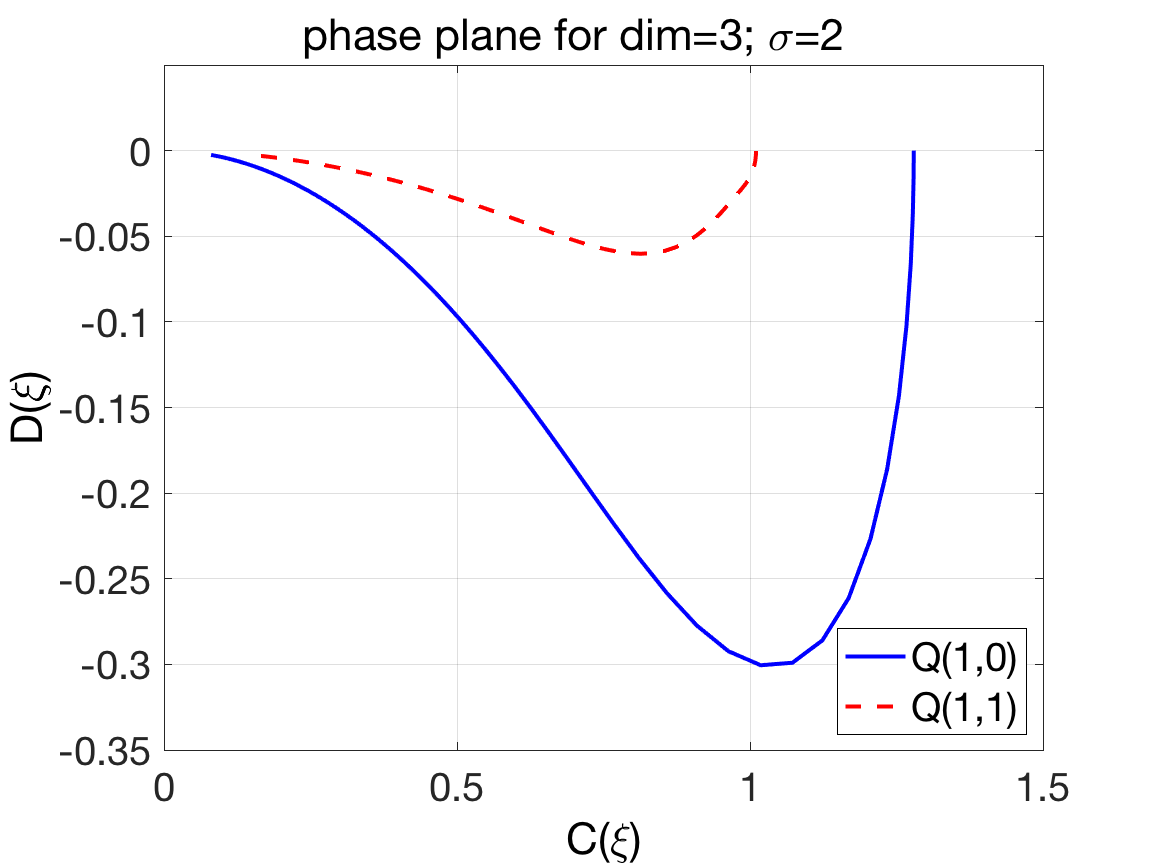}
\caption{ $Q$ profiles for $d=3$, $\sigma=2$ (left) and the phase plane for $Q$ (right).}
\label{Q3dp5}
\end{center}
\end{figure}

\begin{figure}
\begin{center}
\includegraphics[width=0.42\textwidth]{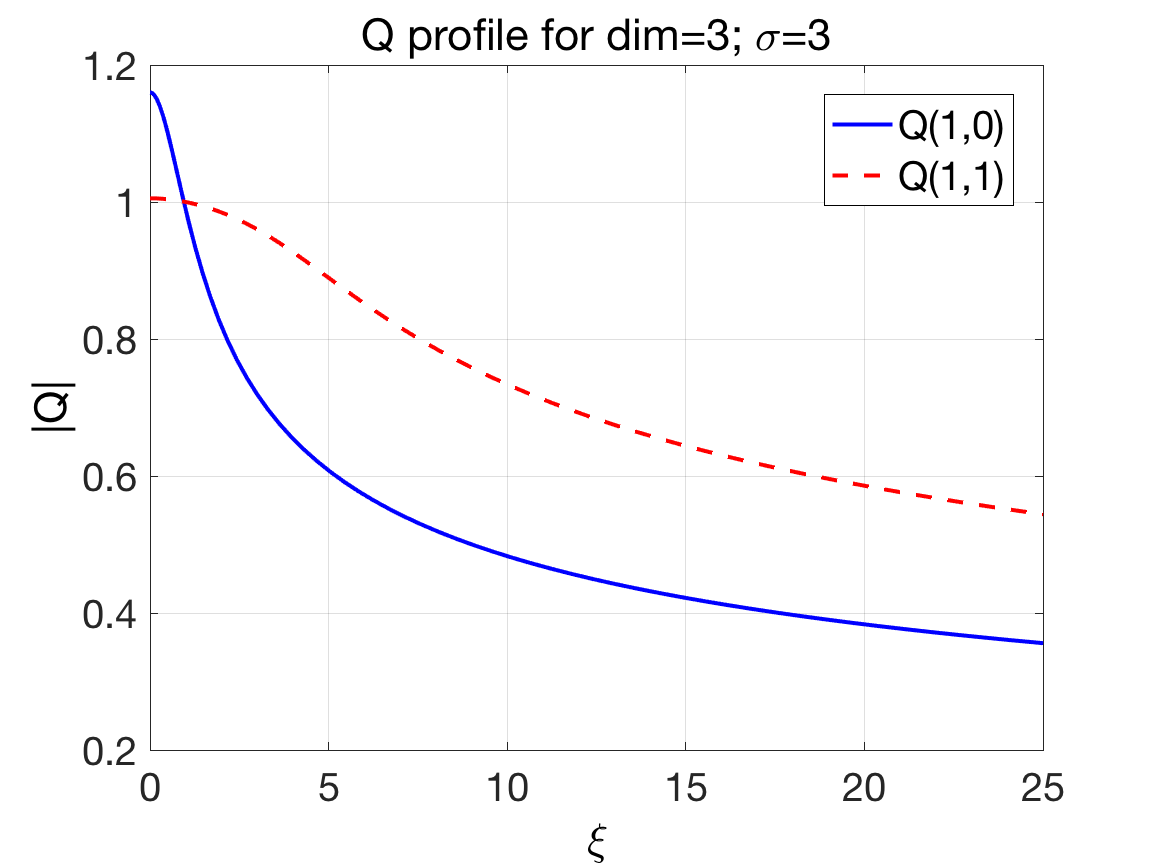}
\includegraphics[width=0.42\textwidth]{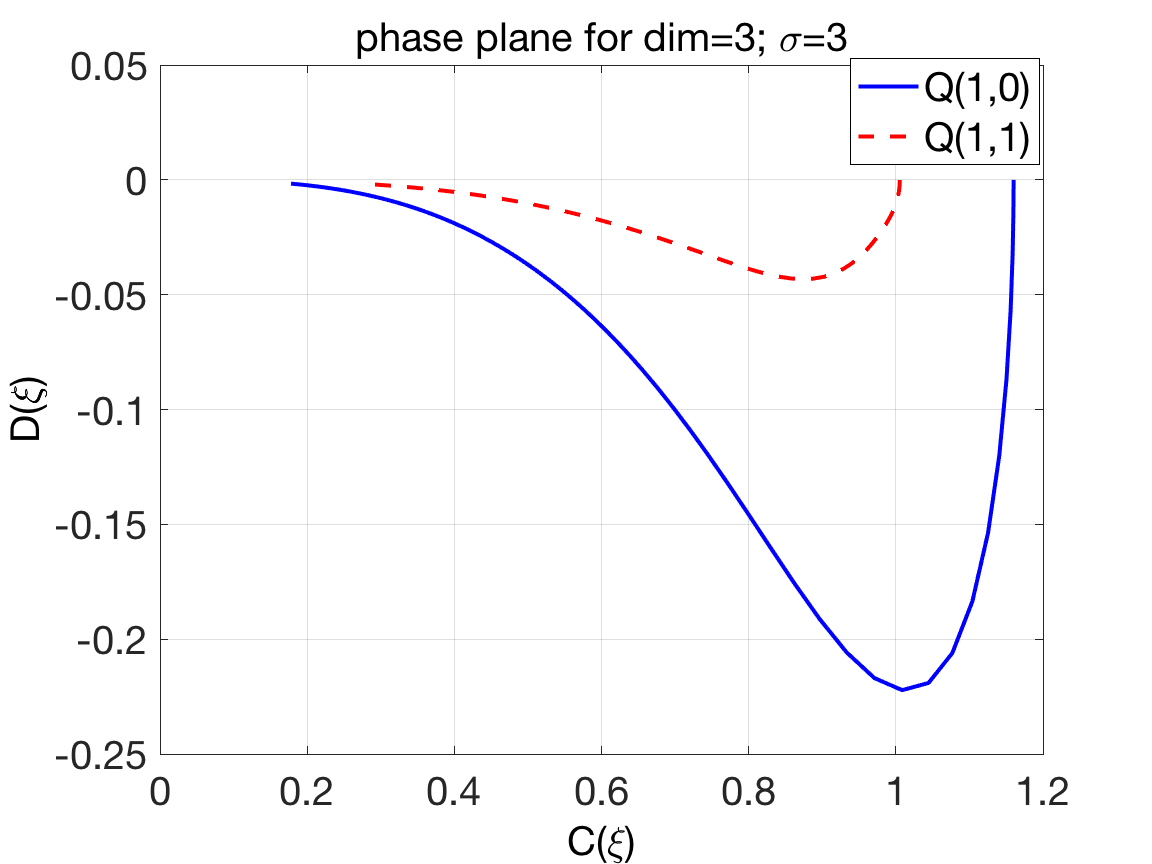}
\caption{ $Q$ profiles for $d=3$, $\sigma=3$ (left) and the phase plane for $Q$ (right).}
\label{Q3dp7}
\end{center}
\end{figure}

We next investigate how the values $Q(0) \defeq Q_{1,0}(0)$ and $a$ change with respect to the dimension $d$. The results are shown in Figure \ref{Q_cubic} (cubic), Figure \ref{Q_quintic} (quintic) and Figure \ref{Q_septic} (septic). Observe that both values produce a smooth curve from the energy-subcritical regime to the energy-supercritical regime. 

\begin{figure}
\begin{center}
\includegraphics[width=0.42\textwidth]{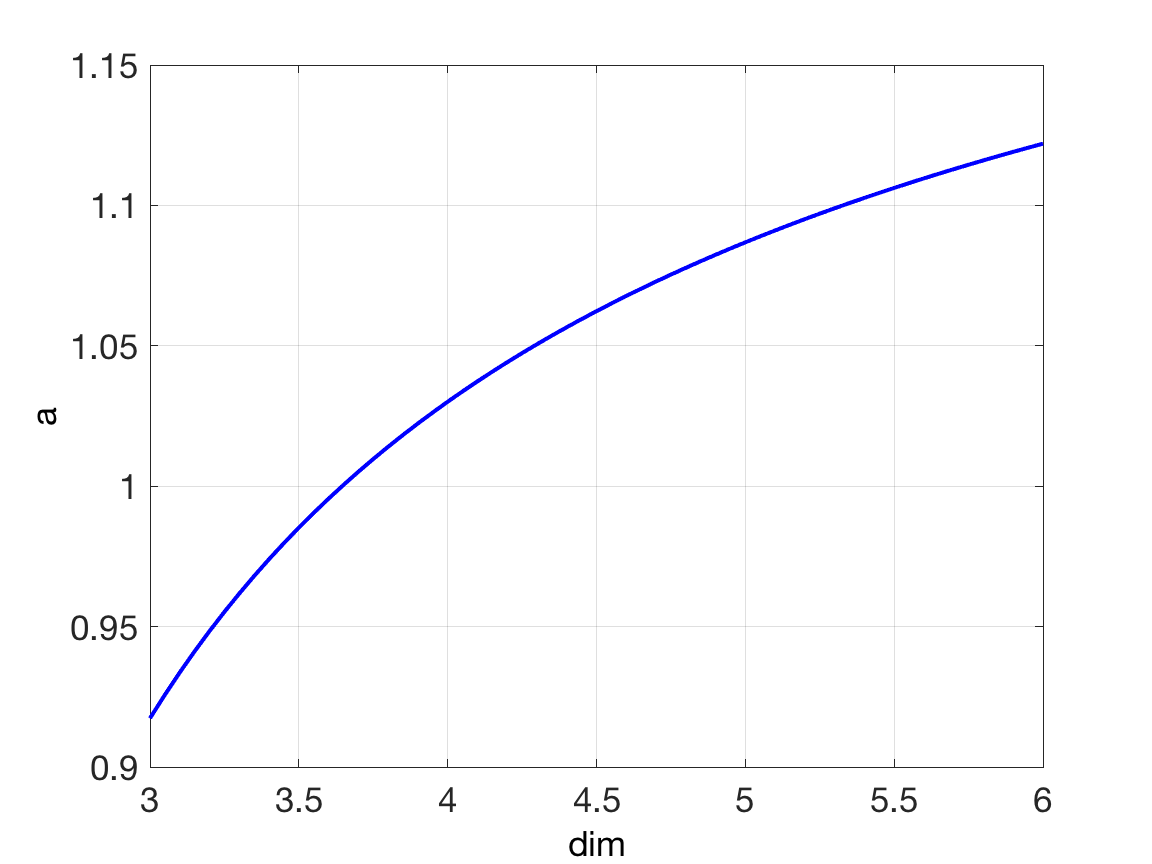}
\includegraphics[width=0.42\textwidth]{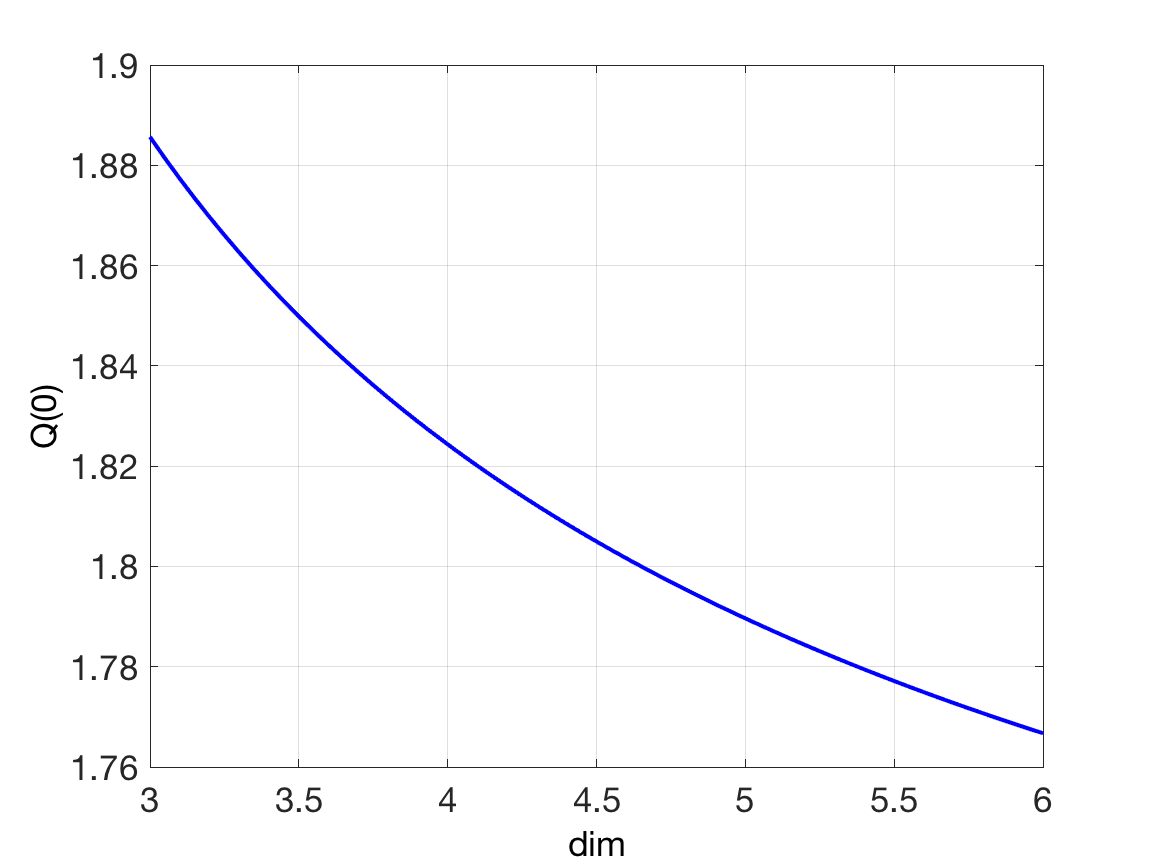}
\caption{ The change of $a$ and $Q(0) = Q_{1,0}(0)$ with respect to the dimension $d$ for the cubic case ($\sigma=1$).}
\label{Q_cubic}
\end{center}
\end{figure}

\begin{figure}
\begin{center}
\includegraphics[width=0.42\textwidth]{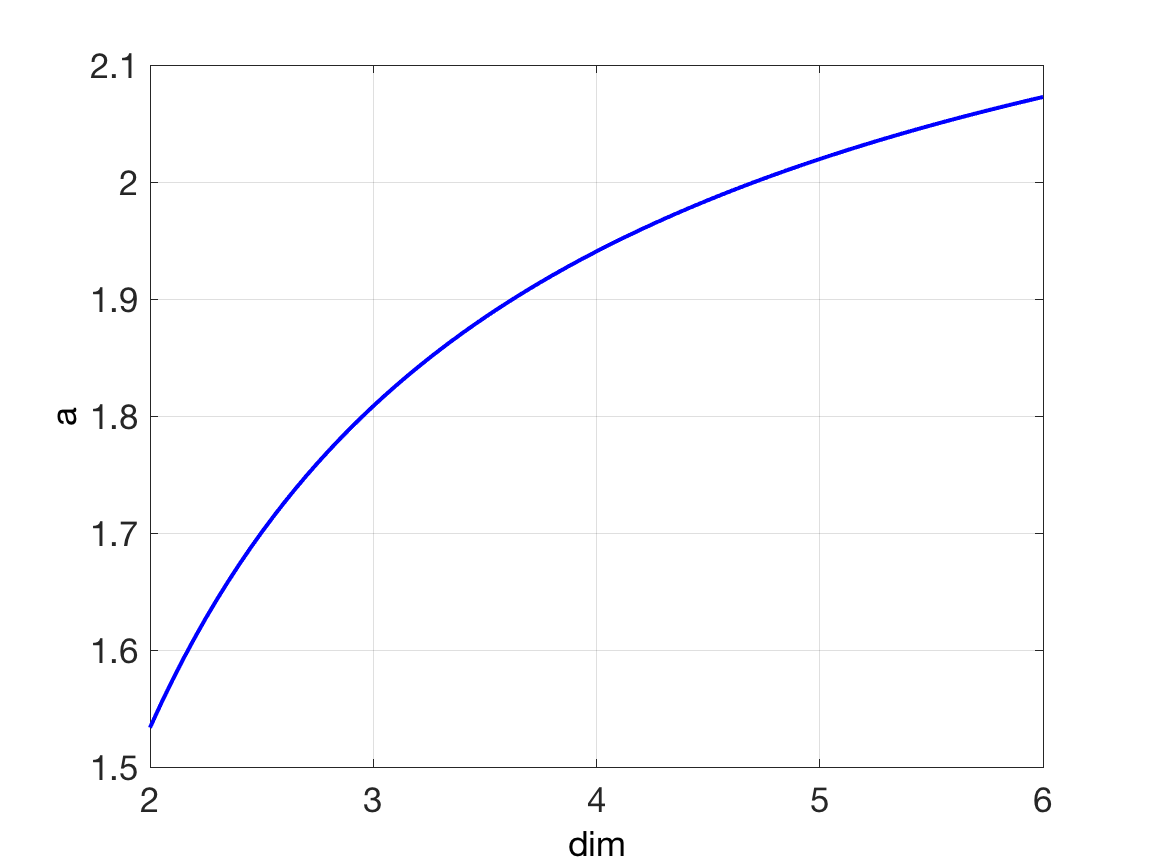}
\includegraphics[width=0.42\textwidth]{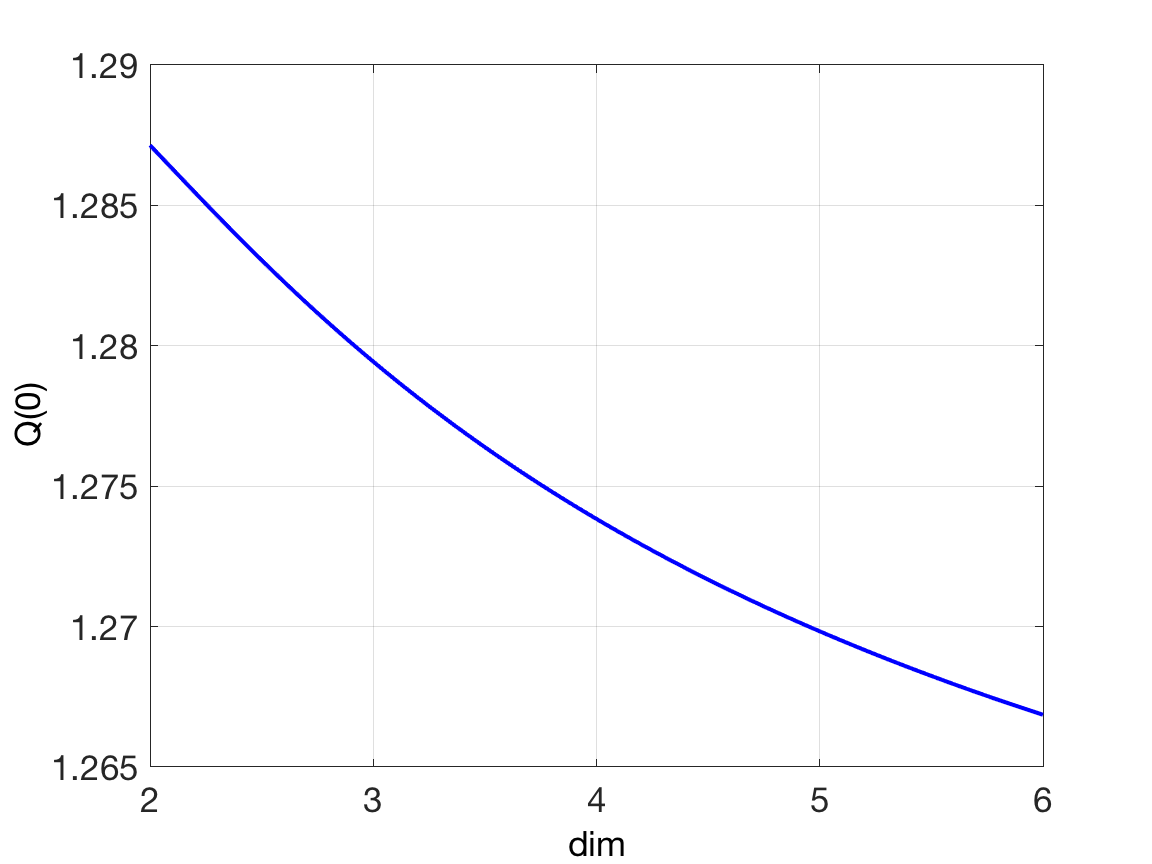}
\caption{ The change of $a$ and $Q(0) = Q_{1,0}(0)$ with respect to the dimension $d$ for the quintic case ($\sigma=2$).}
\label{Q_quintic}
\end{center}
\end{figure}

\begin{figure}
\begin{center}
\includegraphics[width=0.42\textwidth]{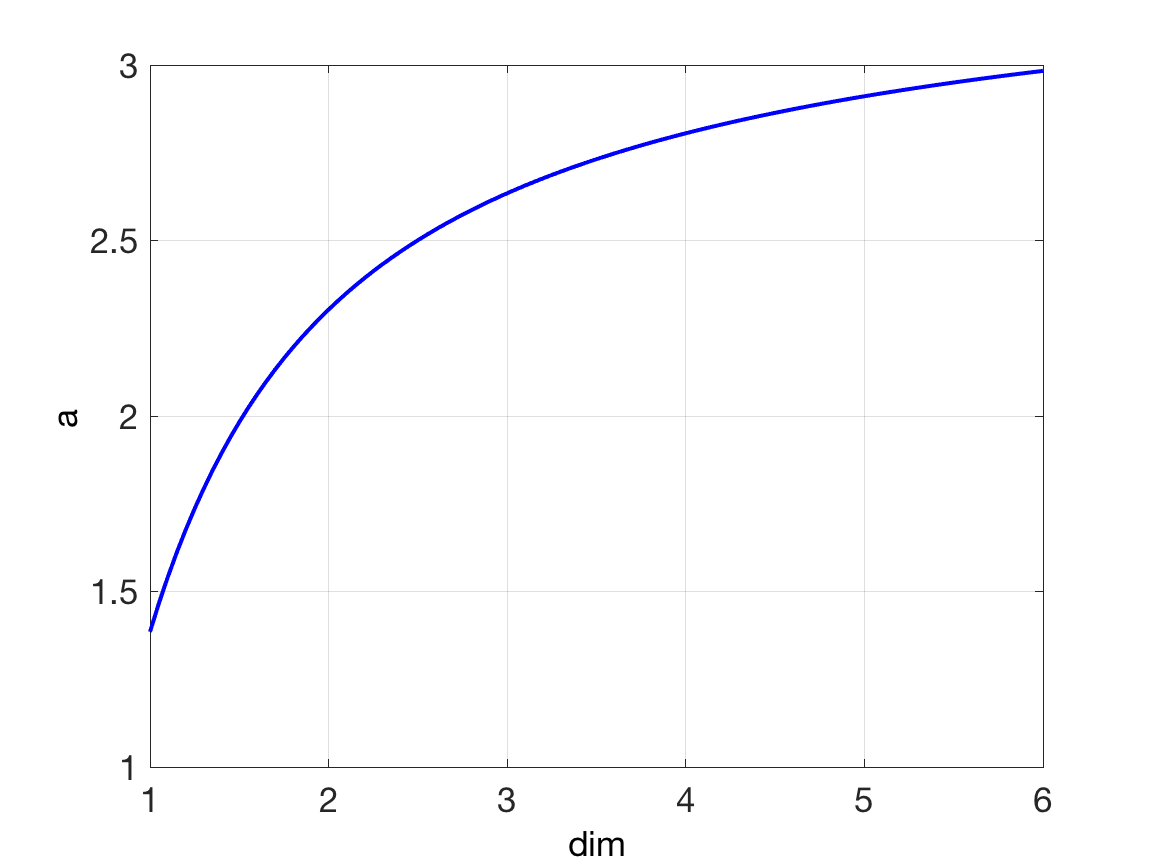}
\includegraphics[width=0.42\textwidth]{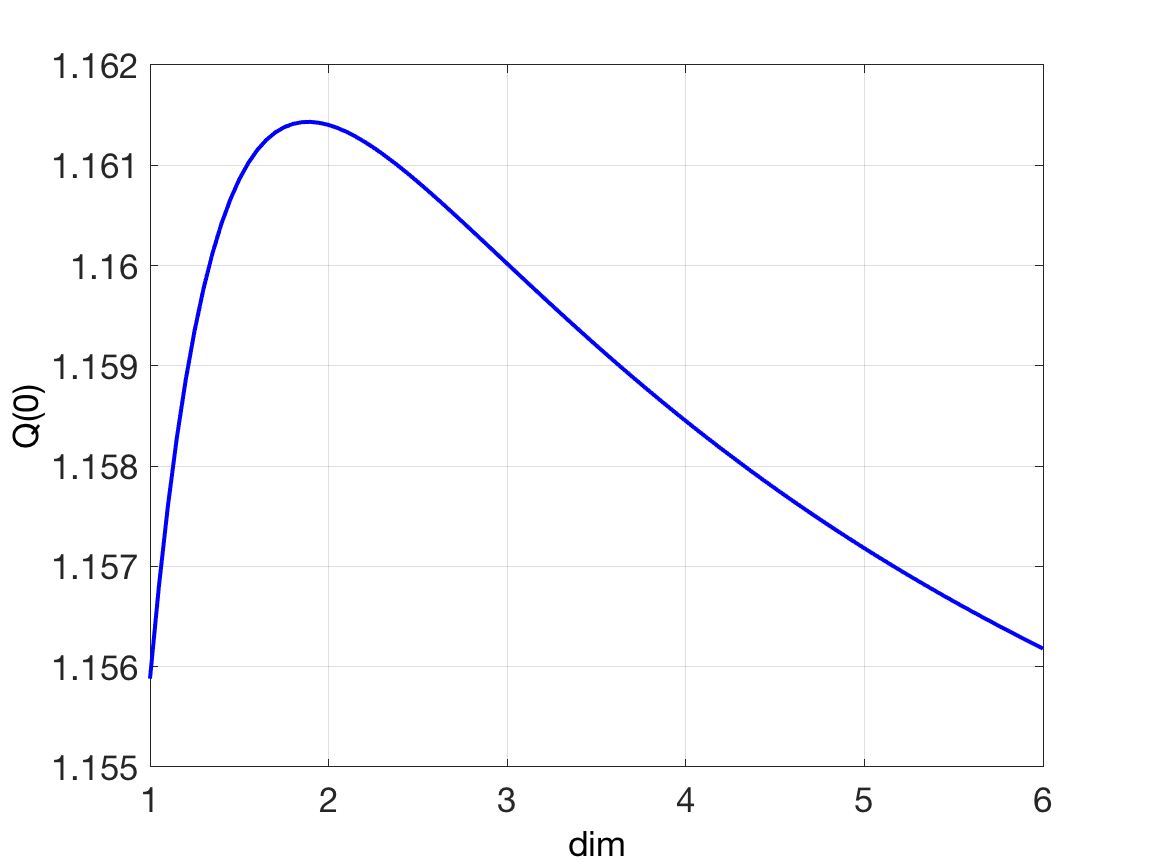}
\caption{ The change of $a$ and $Q(0) = Q_{1,0}(0)$ with respect to the dimension $d$ for the septic case ($\sigma=3$).}
\label{Q_septic}
\end{center}
\end{figure}

\newpage

We next study more closely the values of (conserved) energy depending on the critical scaling index $s_c$. Recall that motivated by the scaling invariance, we seek the self-similar blow-up solutions of \eqref{NLS} of the form 
\begin{align}\label{sol_Q_u}
u(r,t)= \dfrac{1}{(\sqrt{2a(T-t)})^{\frac{1}{\sigma}}} \,  Q \left( \dfrac{r}{\sqrt{2a(T-t)}} \right) \, \textrm{exp} \left( i \theta +\dfrac{i}{2a} \log \dfrac{T}{T-t} \right). 
\end{align}
The rescaled mass and energy for $u(r,t)$ in terms of $Q(\xi)$ are
\begin{align}
M\left[u(t)\right]=C_{a,d}(T-t)^{s_c}\int_0^{\infty} |Q(\xi)|^2 \, \xi^{d-1} \,d\xi,
\end{align}
\begin{align}\label{rescaled Hamiltonian}
E\left[u(t)\right]=C_{a,d} (T-t)^{s_c-1}\int_0^{\infty} \left( |Q_{\xi}|^2-\frac{1}{\sigma+1} |Q|^{2\sigma+2} \right) \xi^{d-1}\,d\xi=C_{a,d} (T-t)^{s_c-1}H[Q].
\end{align}

For $s_c<1$, $(T-t)^{s_c-1} \rightarrow \infty$ as $t \rightarrow T$. From the energy conservation in \eqref{rescaled Hamiltonian}, $H[Q]$ should be zero, since $E\left[ u(t) \right]$ remains constant in time $t$. For $s_c=1$, $(T-t)^{s_c-1}=1$, and thus, $H[Q]$ should be a constant. For $s_c>1$, $(T-t)^{s_c-1} \rightarrow 0$ as $t \rightarrow T$, and thus,
$$
\int_0^{\xi} \left(|Q_{\xi}|^2-\frac{1}{\sigma+1} |Q|^{2\sigma+2} \right) \xi^{d-1} \, d\xi \rightarrow \infty \quad \mbox{as} \quad \xi \rightarrow \infty.
$$
Figure \ref{Q_L} justifies the above reasoning. We calculate the energy of $Q$ in the quintic NLS case ($\sigma=2$) in various dimensions, truncated at different lengths of the interval $K$. The top left subplot shows that in the 2d case the energy $E[Q]$ goes to zero as the interval $K \rightarrow \infty$, this is consistent for the energy-subcritical setting, here $s_c=\frac{1}{2} < 1$. The top right subplot shows that in the 3d case the energy $E[Q]$ goes to a constant as $K \rightarrow \infty$, this is the energy-critical case, $s_c=1$. The bottom left subplot shows that in the 4d case the energy $E[Q]$ goes to negative infinity linearly as $K \rightarrow \infty$, here $s_c=\frac32 > 1$. The bottom right subplot shows that in the 5d case the energy $E[Q]$ goes to negative infinity quadratically as $K \rightarrow \infty$, here $s_c=2 > 1$.  Furthermore, the bottom two subplots justify that the solution $Q$ decays with a rate of $|Q| \sim \xi^{-1/\sigma}$, since substituting $|Q|=\xi^{-1/\sigma}$ into \eqref{rescaled Hamiltonian} and integrating from $K_0$ to $\infty$, one gets
$$
E(Q) \sim \int_{K_0}^{\infty} \xi^{2s_c-3} d\xi.
$$
This matches the linear decay for $s_c=\frac{3}{2}$ (e.g., see $d=4$ and $\sigma=2$ in Figure \ref{Q_L}), and the quadratic decay for $s_c=2$ (e.g., see $d=5$ and $\sigma=2$ in Figure \ref{Q_L}).

\begin{figure}
\begin{center}
\includegraphics[width=0.42\textwidth]{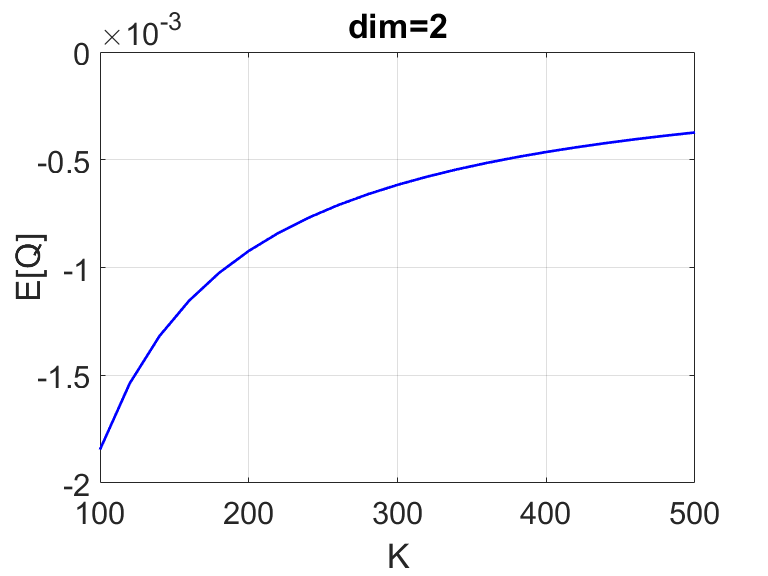}
\includegraphics[width=0.42\textwidth]{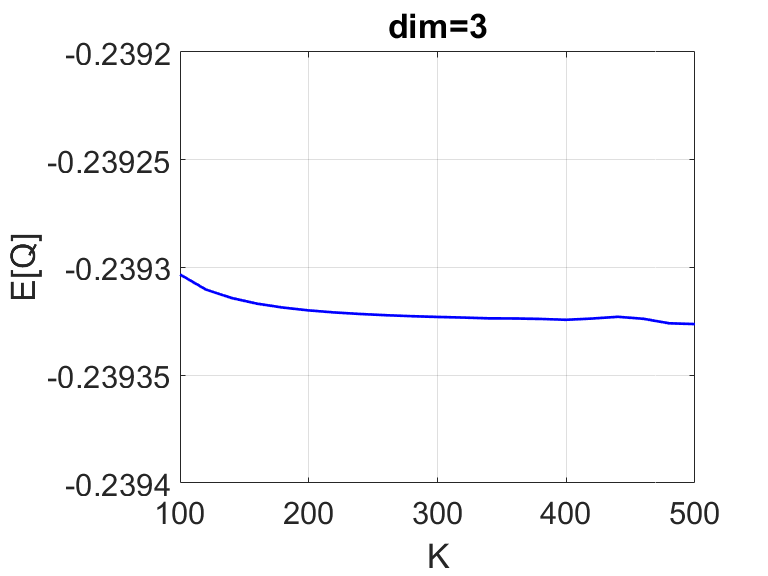}
\includegraphics[width=0.42\textwidth]{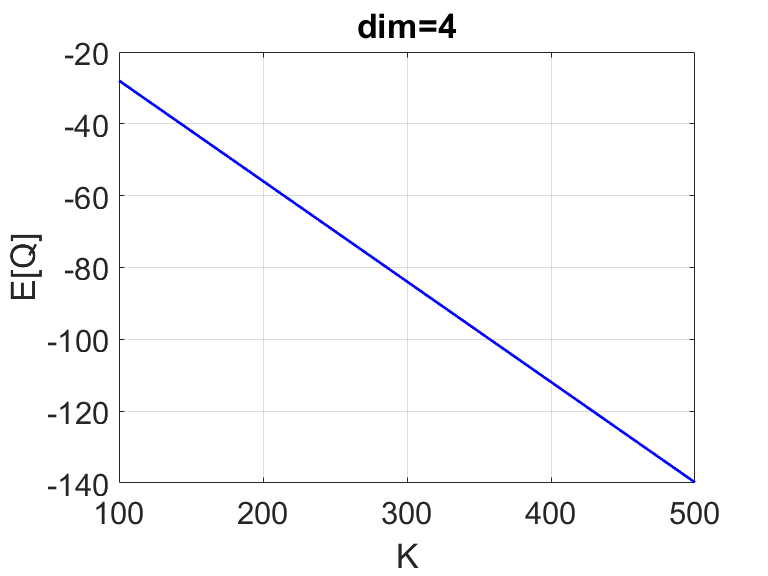}
\includegraphics[width=0.42\textwidth]{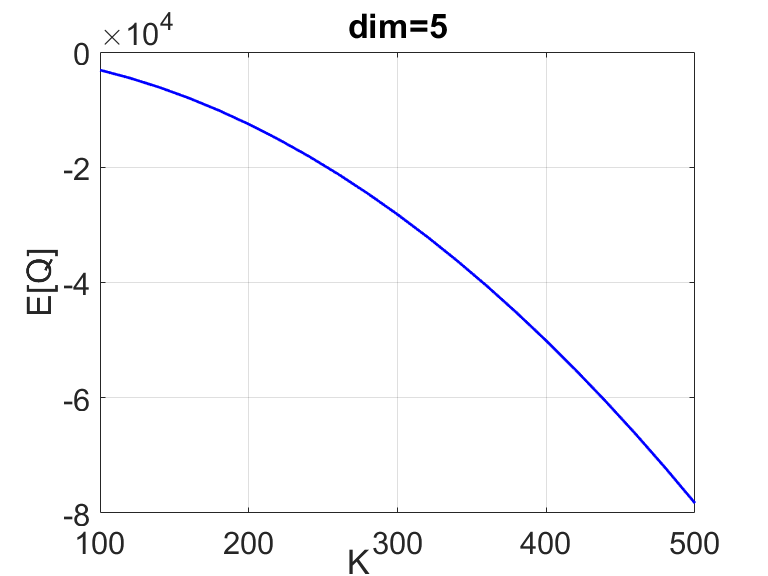}
\caption{ The change of the energy with respect to the computational interval $K$ in dimension $d=2,3,4,5$ for the quintic case ($\sigma=2$).}
\label{Q_L}
\end{center}
\end{figure}

\subsection{Further justification of the constant Hamiltonian for the energy critical case}
The conjecture that $H[Q]={const}$ for the energy critical case ($s_c=1$) can be justified by the following argument.

From \eqref{Q1 and Q2} the asymptotic behavior of $Q$ satisfies $Q(\xi)\approx C_0 \,\xi^{-\frac{i}{a}-\frac{1}{\sigma}}$ for $\xi \gg 1$. Assuming $H[Q]$ being finite, we have 
\begin{align}\label{H[Q] split}
H[Q]=\int_0^{\xi_0} \bigg(|Q_{\xi}|^2-\frac{1}{\sigma+1} |Q|^{2\sigma+2} \bigg) \xi^{d-1} \, d \xi 
+\int_{\xi_0}^{\infty} \left(|Q_{\xi}|^2-\frac{1}{\sigma+1} |Q|^{2\sigma+2}\right) \xi^{d-1} \, d \xi < \infty.
\end{align}
The first integral of \eqref{H[Q] split} gives a constant. Since neither of the terms $|Q_{\xi}|^2$ and $|Q|^{2\sigma+2}$ are integrable, these two terms must cancel each other in the second integral. The direct calculation {for the second integral in \eqref{H[Q] split} gives} 
\begin{align}\label{Q coefficient}
C_0= \left[\left(\sigma+1\right) \left(\frac{1}{\sigma^2}+\frac{1}{a^2}\right) \right]^{\frac{1}{2\sigma}}.
\end{align}

Figure \ref{Q constant} shows the difference of $C_0$ numerically calculated from \eqref{Q compute} ($C_0 :=C_{\textrm{num}}$) and $C_0$ calculated from \eqref{Q coefficient} ($C_0:=C_{\textrm{pred}}$). Observe that the difference is on the order of $10^{-9}$ for $Q_{1,0}$ in both $d=3$ (blue solid line) and $d=4$ (red dash line). Moreover, both are decreasing as  $\xi$ is increasing.

\begin{figure}
\begin{center}
\includegraphics[width=0.45\textwidth]{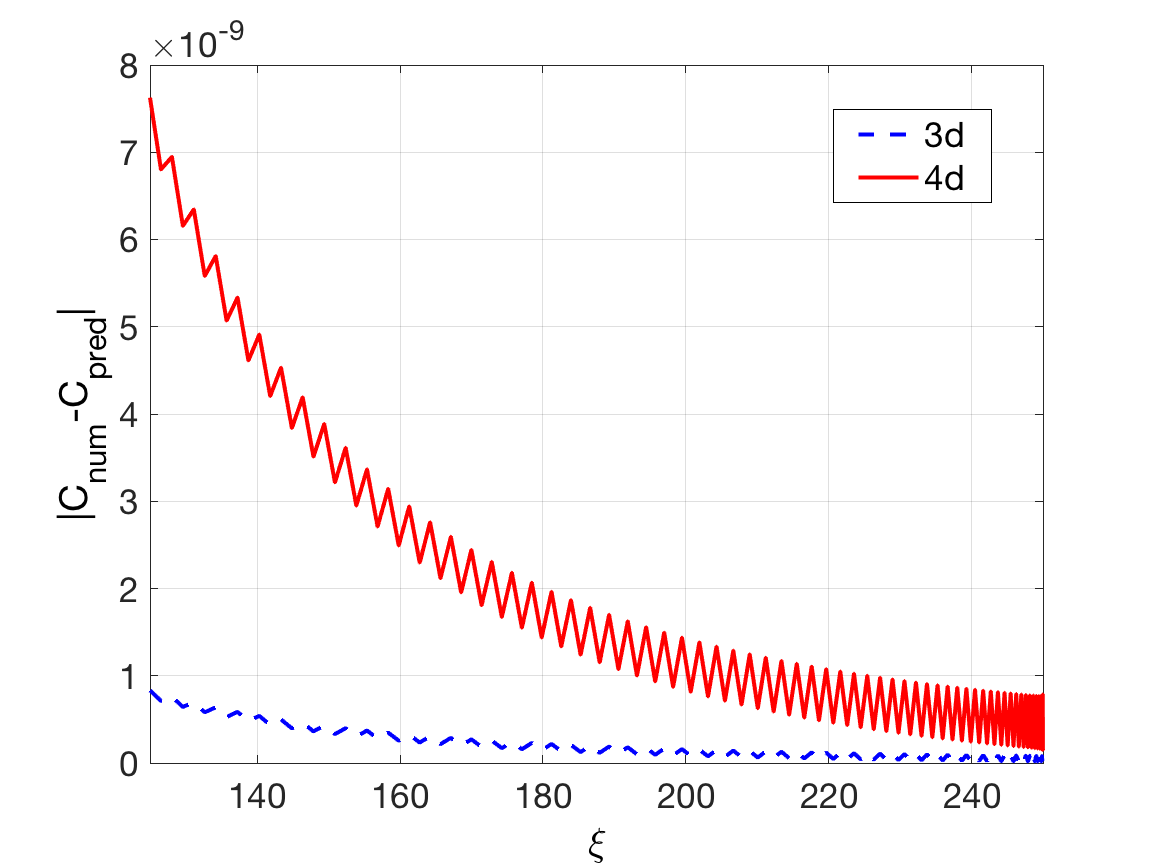}
\includegraphics[width=0.45\textwidth]{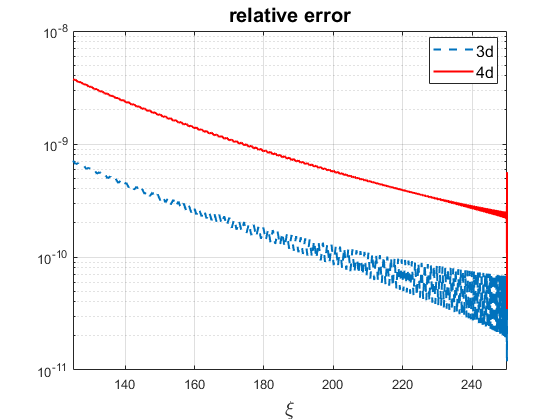}
\caption{Left: $|C_{\textrm{num}}-C_{\textrm{pred}}|$ for the 3d (blue solid line) and the 4d (red dash line) energy-critical cases. One can see that both errors are on the order of $10^{-9}$. Right: relative error $|\frac{C_{\textrm{num}}}{C_{\textrm{pred}}}-1|$ of the above quantities.}
\label{Q constant}
\end{center}
\end{figure}

Now that we have developed a certain description of profiles $Q$, in particular, numerical identification and properties of $Q_{1,0}$, we proceed to investigating the dynamics of stable blow-up solutions in the mass-supercritical cases.

\section{Blow-up dynamics for the NLS equation}\label{S:4}

\subsection{Numerical method}
We compute the rescaled equation (\ref{DRNLS}) and then reconstruct the solution $u(x,t)$ from the rescaled equation, since the solution $v(\xi,\tau)$ to the rescaled equation exists globally in time. This method is called the {\it dynamic rescaling method}, which was first introduced by LeMesurier, Papanicolaou, Sulem and Sulem in \cite{MPSS1986}, \cite{LePSS1987}. This method needs the prior knowledge of the scaling of the singular part of the solution (scaling property). 

There are several other methods that can track the blow-up dynamics. For example, one may use the adaptive mesh method in \cite{ADKM2003}, moving mesh method in \cite{BCR1999}, iterative grid redistribution method in \cite{RW2000} for multi-dimensions and \cite{DG2009} for the 1d case, see also discussion on numerical treatments in \cite{BKW2006}. These methods, unlike the dynamic rescaling method, do not need the prior knowledge of the scaling of the singular part and can deal with more general blow-up dynamics cases. Here, however, we consider generic data which leads to the peak-type solutions, i.e., solutions attain their maximum at the origin, see \cite[Chapter14]{F2015}, and the prior knowledge is already given by rescaling. Therefore, the dynamic rescaling method is more effective for our task.

We fix the value $\|v(\xi,\tau)\|_{L^{\infty}_{\xi}} \equiv 1$ in time $\tau$ as in  \cite{LePSS1988} and \cite{KSZ1991}, and write  
\begin{align}
L(t)=\left(\frac{1}{\|u(t)\|_{L^\infty}}\right)^{\sigma},
\end{align}
and
\begin{align}\label{NLS a compute}
a(\tau)= -{\sigma}\Im\left( \bar{v} \Delta v \right)_{\vert_{(0,\tau )}}.
\end{align}

The equation (\ref{DRNLS}) can be written as 
\begin{align} \label{DRNLSF}
i \,  v_{\tau} +\Delta v +\mathcal{N}(v)=0, \quad \tau \in [0, \infty), \quad \xi \in [0, \infty),
\end{align}
with
$$
\mathcal{N}(v)=ia(\tau)\left(\xi v_{\xi}+\frac{v}{\sigma}\right)+ |v|^{2\sigma}v.
$$
The initial value $v_0(\xi)$ is calculated from \eqref{rescaling initial} by setting $\| v_0\|_{L^{\infty}}=1$.

Let $N\in \mathbb{Z}$ be a fixed integer, $[0,L_D]$ the computational domain with $L_D\gg1$, and $h = \frac{L_D}{N}$ the uniform grid size in space. Let $v_j\approx v(jh,\tau)$ denote the semi-discrete approximate solution at $jh, j = 0,1,\cdots, N$. The spatial derivative is approximated by the sixth order central difference scheme:
\begin{align}
 v_{\xi}(jh,\tau) \approx D^{(1)}_6 v_j = \frac{1}{60h}[-v_{j-3} + 9 v_{j-2} - 45 v_{j-1} + 45 v_{j+1} - 9 v_{j+2} + v_{j+3}],
 \end{align}
 \begin{align}
 & v_{\xi \xi}(jh,\tau)  \approx D^{(2)}_6 v_j \\
                      = & \frac{1}{180h^2}[2v_{j-3} - 27 v_{j-2} + 270 v_{j-1} -490 v_j + 270 v_{j+1} - 27 v_{j+2} + 2v_{j+3}], \nonumber
\end{align}
and the Laplacian operator is approximated by
\begin{align}
\Delta v(jh, \tau) \approx \Delta_h v_j = v_{\xi \xi}(jh,\tau)+\dfrac{d-1}{jh} v_{\xi}(jh,\tau).
\end{align}
When the grid points beyond the right-hand side computational domain are needed, we set up the fictitious points obtained by extrapolation, 
$$
v_{N+2}=8v_{N+1}-28v_N+56v_{N-1}-70v_{N-2}+56v_{N-3}-28v_{N-4}+8v_{N-5}-v_{N-6}.
$$
For the grid points beyond the left-hand side computational domain, note that $v(\xi)$ is radially symmetric, and thus, we use the fictitious points $v_{-j}=v_{j}$. The singularity at $\xi=0$ in the Laplacian term $\Delta_h$ is eliminated by the L'Hospital's rule
$$
\lim_{\xi \rightarrow 0} \dfrac{d-1}{\xi} v_{\xi} = (d-1)v_{\xi \xi}.
$$

The time discretization is similar to our previous work \cite{RYZ2017} for the mass-critical case (see also in \cite{SS1999}). Let $\Delta \tau$ denote the uniform time step with respect to the rescaled time $\tau$ and $\tau_{m} = m \cdot \Delta \tau, m=1,2,\cdots$. Let $v^{(m)}_{j}\approx v(jh,m \cdot \Delta \tau)$ be the approximate solution at $(jh,m\cdot \Delta \tau)$, and $L_m$ the approximation of $ L(\tau_m)$. The time evolution of (\ref{DRNLSF}) can be approximated by the second order Crank-Nicolson-Adam-Bashforth method:
\begin{align}\label{AB}
 i\frac{v^{(m+1)}_j-v^{(m)}_j}{\Delta \tau}+\frac{1}{2}\left(\Delta_h v_j^{(m+1)}+{\Delta}_h v_j^{(m)} \right)+\frac{1}{2}\left( 3\mathcal{N}(v_j^{(m)})-\mathcal{N}(v_j^{(m-1)})\right)=0,
\end{align}
given the initial condition $v^{(0)}$ and $v^{(1)}$. Here, $v^{(1)}$ is obtained by the standard second order explicit Runge-Kutta method (RK2). We introduce the two-step Adams predictor--corrector method as in \cite[Chapter 28]{F2015} to increase the accuracy and 
stability, while remaining a second order scheme in time:
\begin{align}
&i\frac{v^{(m+1)}_{\text{pred},j}-v^{(m)}_j}{\Delta \tau}+\frac{{\Delta}_h v^{(m+1)}_{\text{pred},j}+{\Delta}_h v^{(m)}_j}{2}+  \dfrac{3}{2}\mathcal{N}(v^{(m)}_j)-\dfrac{1}{2}\mathcal{N}(v^{m-1}_j) = 0, \quad (\text{P}) \label{predictor} \\
&i\frac{v^{(m+1)}_j-v^{(m)}_j}{\Delta \tau}+\frac{{\Delta}_h v^{(m+1)}_j+{\Delta}_h v^{(m)}_j}{2} + \frac{1}{2}\mathcal{N}(v^{(m+1)}_{\text{pred},j}) + \frac{1}{2}\mathcal{N}(v^{(m-1)}_j) = 0. \quad (\text{C}) \label{corrector}
\end{align}

We use the method in \cite{LePSS1988} to reconstruct the solution in $(r,t)$ variable. After getting the value $v^{(m+1)}$, we update the value $a^{(m+1)}$ from \eqref{NLS a compute}. 
From \eqref{E:a}, $\ln L(\tau_{m+1})$ is obtained by the second order trapezoid rule:
\begin{align} \label{NLS lnL}
\ln L(\tau_{m+1})=\ln L(\tau_m) + \frac{\Delta \tau}{2}(a^{(m+1)}+a^{(m)}).
\end{align}
Then, we have $L(\tau_{m+1})=\exp(\ln L(\tau_{m+1}))$. Denoting $\Delta t_{m+1}:=t_{m+1}-t_m$, we obtain this difference from the last equation of \eqref{rescaling initial}
\begin{align}
\Delta t_{m+1}= \Delta \tau L^2(\tau_{m+1}).
\end{align}
Thus, the mapping for rescaled time $\tau$ back to the real time $t$ is calculated as 
\begin{align}
t(\tau_{m+1}) = t((m+1)\Delta \tau) := \sum_{j=1}^{m+1} \Delta t_j = \Delta \tau\sum_{j=1}^{m+1}  L(\tau_j)^2.
\end{align}
Finally, the numerical solution $u_j^{(m+1)} \approx u(\xi_jL(\tau_{m+1}),\tau_{m+1})$ can be reconstructed.

Note that as time evolves, the time difference $T - t(\tau_n)$ will become smaller and smaller, and eventually reach saturation level (with little change), therefore, we treat the stopping time $t(\tau_{\text{end}}) = t(\tau_{M})$ as the blow-up time $T$, where $M$ is the total number of iterations when reaching the stopping condition ($L<10^{-24}$). Then, we can take
\begin{align}
T = t(\tau_{\text{end}}) = t(M \Delta \tau) = \Delta \tau\sum_{j=1}^M  L(\tau_j)^2.
\end{align}
Consequently for any $t_i$, we calculate $T-t_i$ as 
\begin{align}\label{T-t}
T-t_i=\sum_{j=i+1}^M \Delta t_j = \Delta \tau \sum_{j=i+1}^M  L(\tau_j)^2.
\end{align}
This indicates that instead of recording the cumulative time $t_i$, we only need to record the elapsed time between the two recorded data points, i.e., $\Delta t_i =t_{i+1}-t_{i}$. By doing so, it can avoid the loss of significance when adding a small number onto a larger one.

We construct the artificial boundary condition on the right-hand side from the argument in \cite{KSZ1991} and \cite[Chapter 6.1]{SS1999}, since otherwise, the solution to (\ref{DRNLS}) has to be solved in the entire space $\xi >0$. For $\xi \gg 1$, the nonlinear term and the Laplacian terms is of the higher order compared with the other linear terms, and consequently, can be negligible. Consequently, the equation (\ref{DRNLS}) is reduced to 
\begin{align}\label{DRNLS_L}
v_{\tau}+a(\tau)\left( \frac{v}{\sigma} +\xi v_{\xi}\right)=0
\end{align}
near $\xi=L_D$, the right endpoint of the computational domain. The equation (\ref{DRNLS_L}) can be solved exactly (see \cite{SS1999}, \cite{KSZ1991})
\begin{align}\label{DRNLS_abc sol}
v(\xi,\tau)= v\left(\xi \frac{L(\tau)}{L(\tau_0)}, \tau_0 \right)\left( \frac{L(\tau)}{L(\tau_0)}\right)^{\frac{1}{\sigma}},
\end{align}
which suggests that at $\xi=L_D$,
\begin{align}\label{eqn:v_mp1}
v(L_D,\tau_{m+1})=v\left(L_D \frac{L(\tau_{m+1})}{L(\tau_m)},\tau_m \right)\left( \frac{L(\tau_{m+1})}{L(\tau_m)}\right)^{\frac{1}{\sigma}}.
\end{align}
Note that $a^{(m)}=-\frac{L'(\tau_m)}{L(\tau_m)}$ from \eqref{E:a}, and $\frac{L(\tau_{m})}{L(\tau_{m-1})}$ can be approximated with the  second order accuracy by 
$$\frac{L(\tau_{m})}{L(\tau_{m-1})}=e^{-\frac{\Delta \tau}{2}(a^{(m-1)}+a^{(m)})}+ O(\Delta \tau^3)$$
from \eqref{NLS lnL}. 
The value $L(\tau_{m+1})$ can be approximated by the second order central difference
\begin{align}\label{NLS L predict}
L(\tau_{m+1})=L(\tau_{m-1})+2\Delta \tau L_{\tau}(\tau_m) +O(\Delta \tau^3).
\end{align}
Multiplying the equation \eqref{NLS L predict} by $1/L(\tau_m)$, we obtain
\begin{align}\label{NLS L predict 2}
\dfrac{L(\tau_{m+1})}{L(\tau_{m})}=\dfrac{L(\tau_{m-1})}{L(\tau_{m})}-2\Delta \tau a^{(m)}+O(\Delta \tau^3).
\end{align}
Therefore, the right-hand side boundary condition is approximated with the second order accuracy 
\begin{align}\label{abc_RNLS}
v(L_D,\tau_{m+1})=v\left(L_D (e^{\frac{\Delta \tau}{2}(a^{(m-1)}+a^{(m)})}-2\Delta \tau a^{(m)}),\tau_m \right) \left( (e^{\frac{\Delta \tau}{2}(a^{(m-1)}+a^{(m)})}-2\Delta \tau a^{(m)})\right)^{\frac{1}{\sigma}}.
\end{align}
{Since we are simulating the generic blow-up solutions, the amplitude $\|u\|_{L^{\infty}}$ is increasing. Thus, for all $\tau>0$, the term $\frac{L(\tau_{m+1})}{L(\tau_m)}<1$.} Taking $\zeta=L_D\frac{L(\tau_{m+1})}{L(\tau_m)}$,  then $\zeta$ must be within the computational domain $[0,L_D]$, but not necessarily to be one of the grid points. A cubic spline interpolation is adopted to evaluate $v(\zeta, \tau_m)$ and consequently lead to $v(L_D,\tau_{m+1})$ from (\ref{eqn:v_mp1}).

An alternative method for obtaining the artificial boundry condition is to solve the equation \eqref{DRNLS_L} numerically, since at the point $\xi=L_D$, it reduces to the ODE with respect to $\tau$. In \cite{LPSSW1991}, the authors solved the equation \eqref{DRNLS_L} by using the second order Adam-Bashforth method
\begin{align}\label{abc_AB}
 v(L_D,\tau_{m+1})=v(L_D,\tau_m) -  
\dfrac{1}{2} \Delta \tau\bigg[ &3a^{(m)}\left(\dfrac{v(L_D,\tau_m)}{\sigma} + L_D  v_{ \xi}(L_D,\tau_m) \right)  \\
&- a^{(m-1)}\left(\dfrac{v(L_D,\tau_{m-1})}{\sigma}+L_D v_{ \xi}(L_D,\tau_{m-1}) \right)\bigg], \nonumber
\end{align}
where the terms $v(L_D,\tau_n)\approx v_N^{(n)}$ and $v_{\xi}(L_D,\tau_n)$ is calculated by the six order central difference in space from $v(L_D,\tau_n)$, where $n=m-1$ and $m$.

While both numerical boundary conditions \eqref{abc_RNLS} and \eqref{abc_AB} are of the second order accuracy and lead to the similar results, our numerical experiments suggest that using the method \eqref{abc_RNLS} allows us to take a larger time step $\Delta \tau$.

\subsection{The rescaling of $Q$ and $a$.}
Recall the solution $u(r,t)$ to the equation \eqref{NLS} satisfies the self-similar form
\begin{align}\label{NLS self-similar}
u(x,t) = \dfrac{1}{L(t)^{\frac{1}{\sigma}}} Q\left(\frac{x}{L(t)}\right) \exp \left({i \theta + \frac{i}{2a}\log \frac{T}{T-t}} \right),
\end{align}
where $L(t)$ is predicted to be
\begin{align}\label{blowup rate}
L_{\text{pred}}(t) \approx (2a(T-t))^{\frac{1}{2}}
\end{align}
from \cite{LePSS1988}.

Suppose $Q(\xi)$ is the profile from solving \eqref{Q compute}, and $\tilde{Q}(\eta)$ is another profile with $\| \tilde{Q}\|_{L^{\infty}}=\| v_0(0)\|_{L^{\infty}}$ (e.g., $\| \tilde{Q}\|_{L^{\infty}}=1$). From \eqref{NLS self-similar}, we have a family of the $Q$ profiles
\begin{align}\label{Q rescale}
Q(\xi)= \left(\dfrac{Q(0)}{\tilde{Q}(0)} \right) \tilde{Q} \left( \xi \left( \dfrac{Q(0)}{\tilde{Q}(0)} \right)^{{\sigma}} \right)
\end{align}
and consequently, the value $\tilde{a}$ corresponding to the value $a$ from \eqref{Q compute} is
\begin{align}\label{NLS a rescale}
\tilde{a}=a\left[\dfrac{|v_0(0)|}{Q(0)} \right]^{2{\sigma}}.
\end{align} 
For simplicity, we still use $Q$ to represent the family of $Q$ profiles, adding ``up to scaling".

\subsection{Numerical results.}
In this section, we list examples of initial data we choose and the quantities we track. We take $h=0.1$, $k=10^{-4}/2^{\sigma-2}$, $L_D=100$ for the dimension $d=2,3$, and $L_D=200$ for the dimension $d=4,5$, since larger dimensions may lead the approximation of the artificial boundary condition \eqref{abc_RNLS} or \eqref{abc_AB} being more reflective, and thus, need larger interval. Other choices of those parameters lead to the similar results (for example, we have tested for $h=0.05$ or $\Delta \tau=10^{-4}/2^{\sigma-1}$, see Figure \ref{NLS 3d5p error analysis}). The initial data is taken to be Gaussian $u_0=Ae^{-r^2}$ or rational function with fast enough decay $u_0=\frac{A}{(1+r^2)^4}$. Table \ref{Initial data} lists the examples of initial data together with the energy sign. 

\begin{table}[ht]
\begin{tabular}{|c|c|c|c|c|c|}
\hline
$d$&$\sigma$&$u_0$&$E[u_0]$&
$u_0$ &$E[u_0]$ \\
\hline
$3$&$1$ & $5e^{-r^2}$&$<0$&$\frac{6}{(1+r^2)^4}$&$>0$\\
\hline
$4$&$1$ & $6e^{-r^2}$&$<0$&$\frac{8}{(1+r^2)^4}$&$>0$\\
\hline
$5$&$1$ & $6e^{-r^2}$&$<0$&$\frac{8}{(1+r^2)^4}$&$>0$\\
\hline
$2$&$2$ &$2e^{-r^2}$&$>0$&$\frac{2.5}{(1+r^2)^4}$&$<0$\\
\hline
$3$&$2$ & $3e^{-r^2}$&$<0$&$\frac{3}{(1+r^2)^4}$&$>0$\\
\hline
$4$&$2$ & $3e^{-r^2}$&$<0$&$\frac{3}{(1+r^2)^4}$&$>0$\\
\hline
$3$&$3$ & $2.5e^{-r^2}$&$<0$&$\frac{2.5}{(1+r^2)^4}$&$>0$\\
\hline
\end{tabular}
\linebreak
\linebreak
\caption{Samples of initial conditions $u_0$ used in our simulations.}
\label{Initial data}
\end{table}
 
In our numerical simulations,  the following quantities are of the most interest:
\begin{itemize}
\item blow-up profiles $v(\xi,\tau)$ at different time when approaching blow-up time $T$;
\item blow-up rate $\ln L$ vs. $\ln(T-t)$;
\item the value of $a(\tau)$ with respect to the time $\tau$;
\item the dependence of the distance $\| |v(\tau)| -|Q| \|_{L^{\infty}_{\xi}}$ between $|Q|$ and $|v|$ on the rescaled time $\tau$;
\item the relative error between blow-up rate and the predicted blow-up rate 
$$
\mathcal{E}_{rel} = \left| \left(\dfrac{L(t)}{\sqrt{2\tilde{a} (T-t)}}\right)^{\frac{1}{\sigma}}-1 \right|,
$$
where the value $\tilde{a}$ is taken to be $\tilde{a} = a(\tau_{\text{end}})$, i.e., when the stopping criterion reaches
\begin{align}\label{stopping_criterion}
L < 10^{-24}.
\end{align}
\end{itemize}
Since we track the quantities $\ln(L)$, $\ln(T-t)$, $a(\tau)$ in the simulations, the relative error $\mathcal{E}_{rel}$ is calculated as 
\begin{align}\label{NLS relative error}
\mathcal{E}_{rel} = \left| \exp\left(\dfrac{1}{2\sigma}(2\ln(L)-\ln(T-t)-\ln2 -\ln \tilde{a}) \right) -1 \right|.
\end{align}
When each term in \eqref{NLS relative error} is moderate (not too large or too small), the accuracy is improved. Note that the value $\tilde{a}$ can also be calculated from the equation \eqref{NLS a rescale}, where the values of $a$ and $Q(0)$ are obtained from solving the equation \eqref{Q compute}. 

We numerically verify that the $\tilde{a}$ from these two methods only differs at an order $10^{-8}$. This indicates that the profile of $Q$ obtained in Section \ref{sec: Qprofile} is indeed the blow-up profile.

In Figure \ref{2d5p relative error compare}, the left picture shows the relative error in our simulations ending at $L \sim 10^{-24}$, and the right one is up to $L \sim 10^{-16}$. The relative error always stays small until the last few points close to our simulation ending time. It suggests that this phenomenon is due to the inaccurate estimation of $T$, instead of the solution behavior itself. In the rest of our work,  to make it less confusing, we only show the relative error up to $L \sim 10^{-20}$, while we end our simulation at $L < 10^{-24}$.

\begin{figure}
\begin{center}
\includegraphics[width=0.42\textwidth]{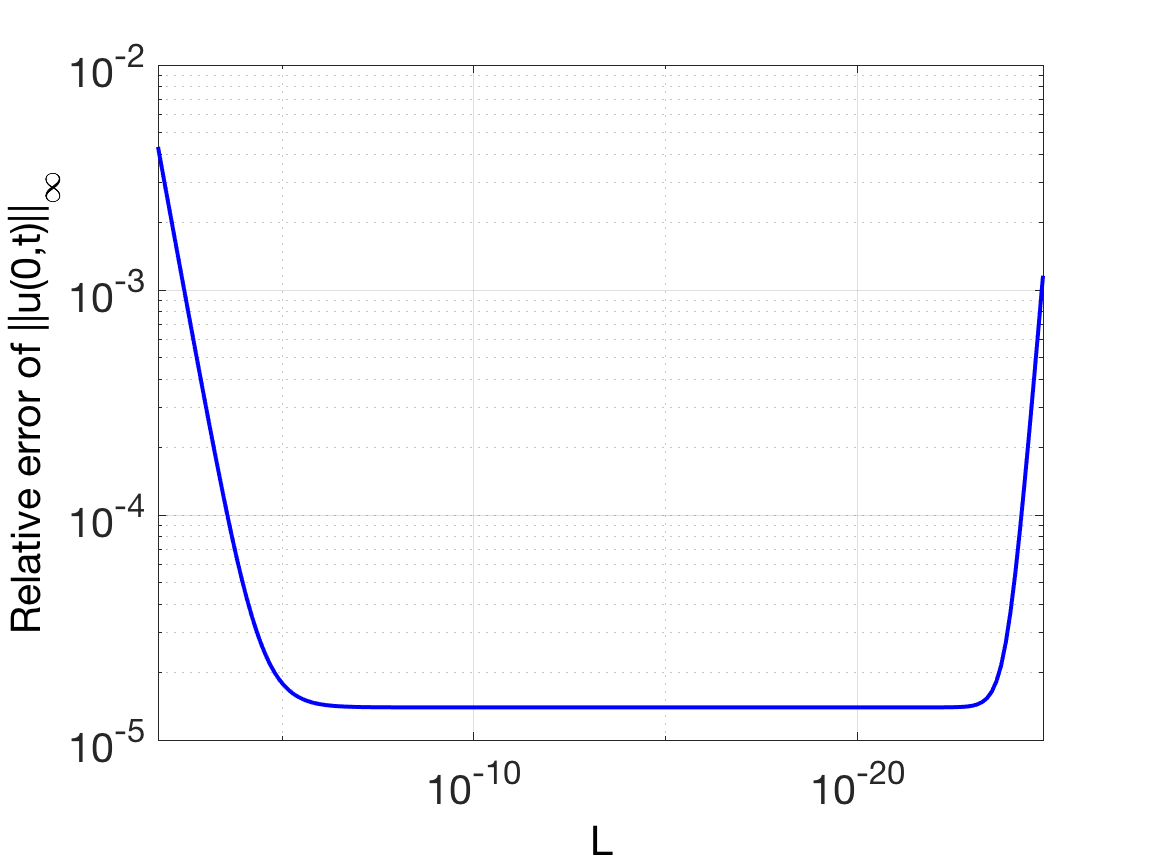}
\includegraphics[width=0.42\textwidth]{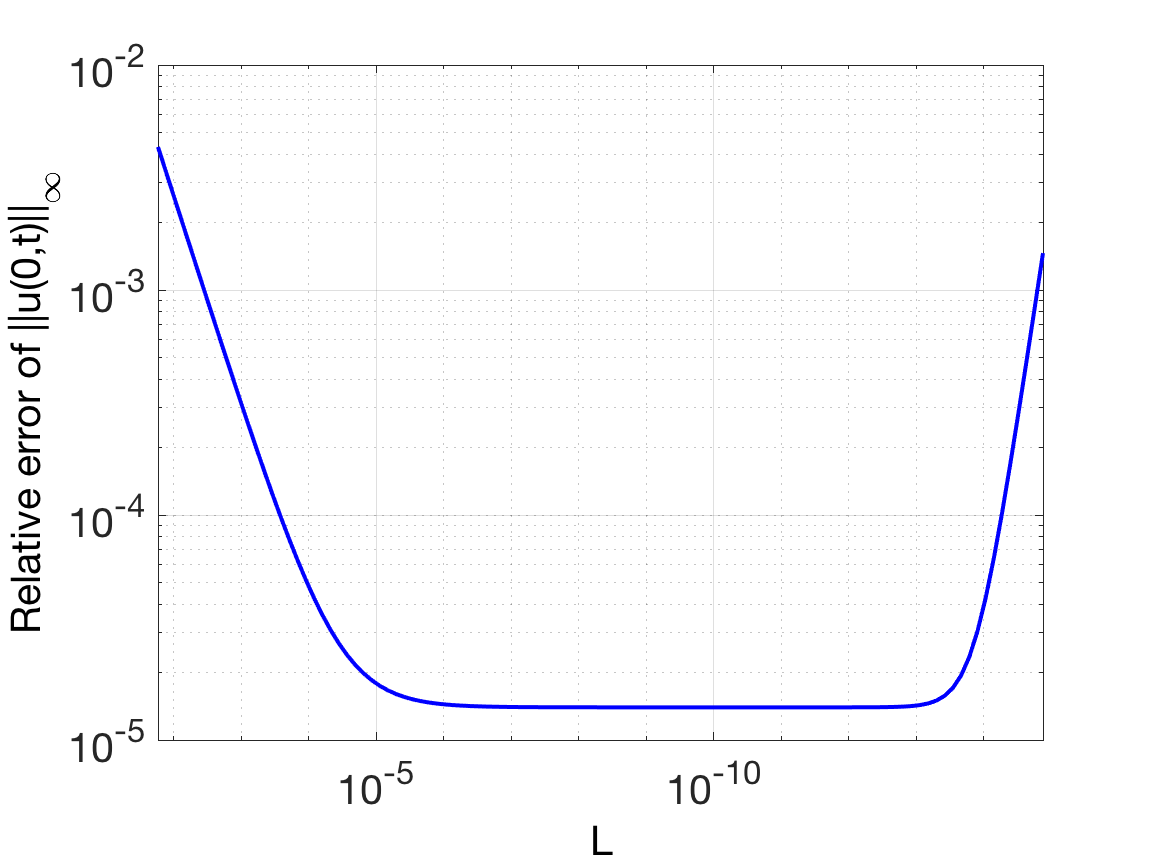}
\caption{ The relative error between the predicted blow-up rate and the numerical results for the 2d quintic case. The increase of the error at a very small $L$ is due to less accurate estimate of the blow-up time $T$.}
\label{2d5p relative error compare}
\end{center}
\end{figure}

\subsection{Consistency verification} We first report the data for the 3d cubic case, which has been considered in \cite{BCR1999}, \cite{LPSS1988} and \cite{LPSSW1991}, as the purpose of verifying the consistency. Figure \ref{3d3p profiles} and \ref{3d3p data} show the blow-up dynamics for the 3d cubic case. We also track the relative error $\mathcal{E}_{rel}$ for different $h$ and $k$. As an example, we list the 3d quintic case in Figure \ref{NLS 3d5p error analysis}, which shows that while the different space step size does not affect the relative error, shrinking the time step will lead to more accurate results.


\begin{figure}
\begin{center}
\includegraphics[width=0.42\textwidth]{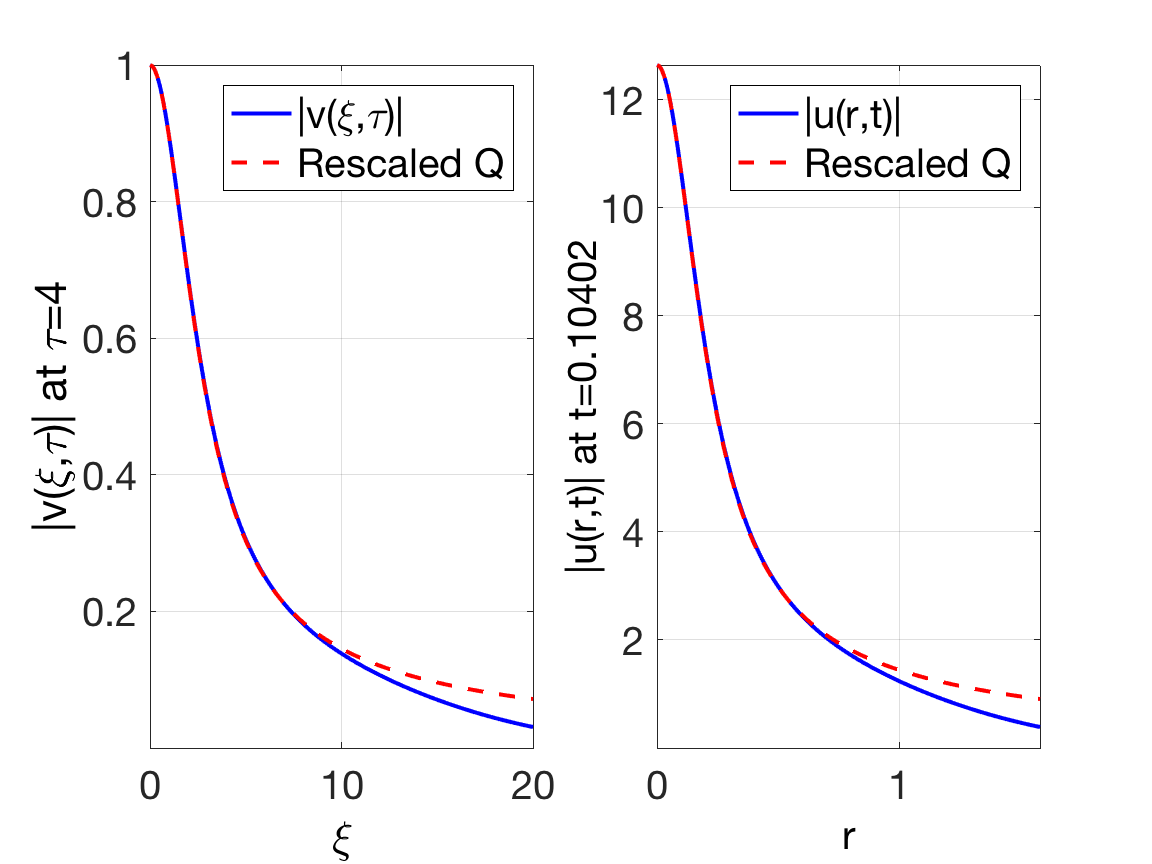}
\includegraphics[width=0.42\textwidth]{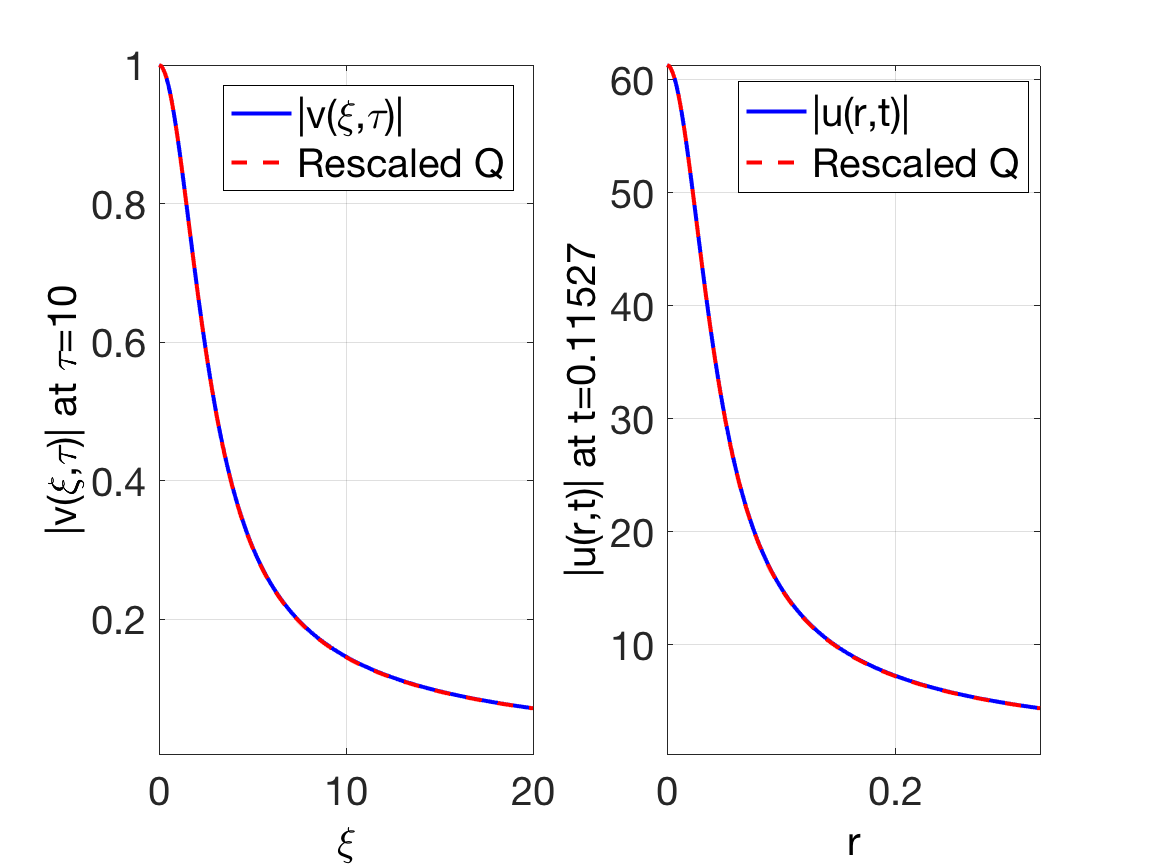}
\includegraphics[width=0.42\textwidth]{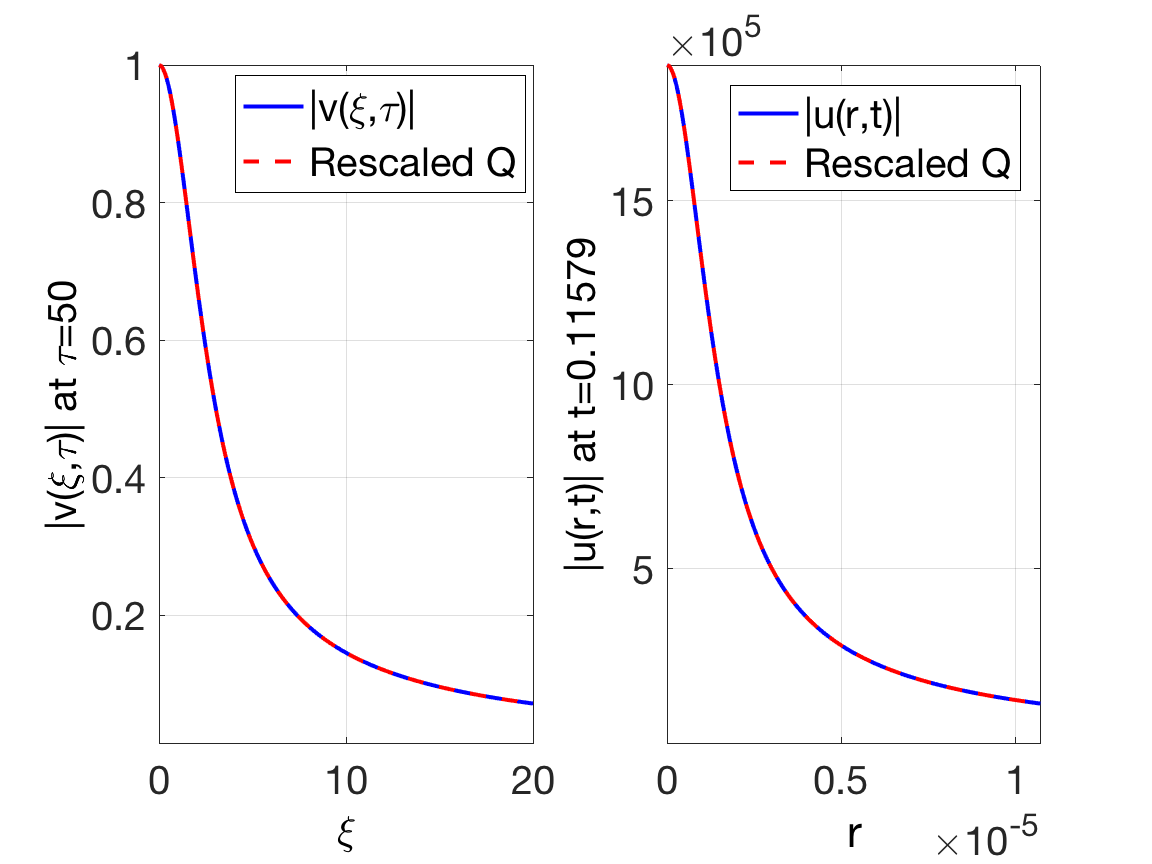}
\includegraphics[width=0.42\textwidth]{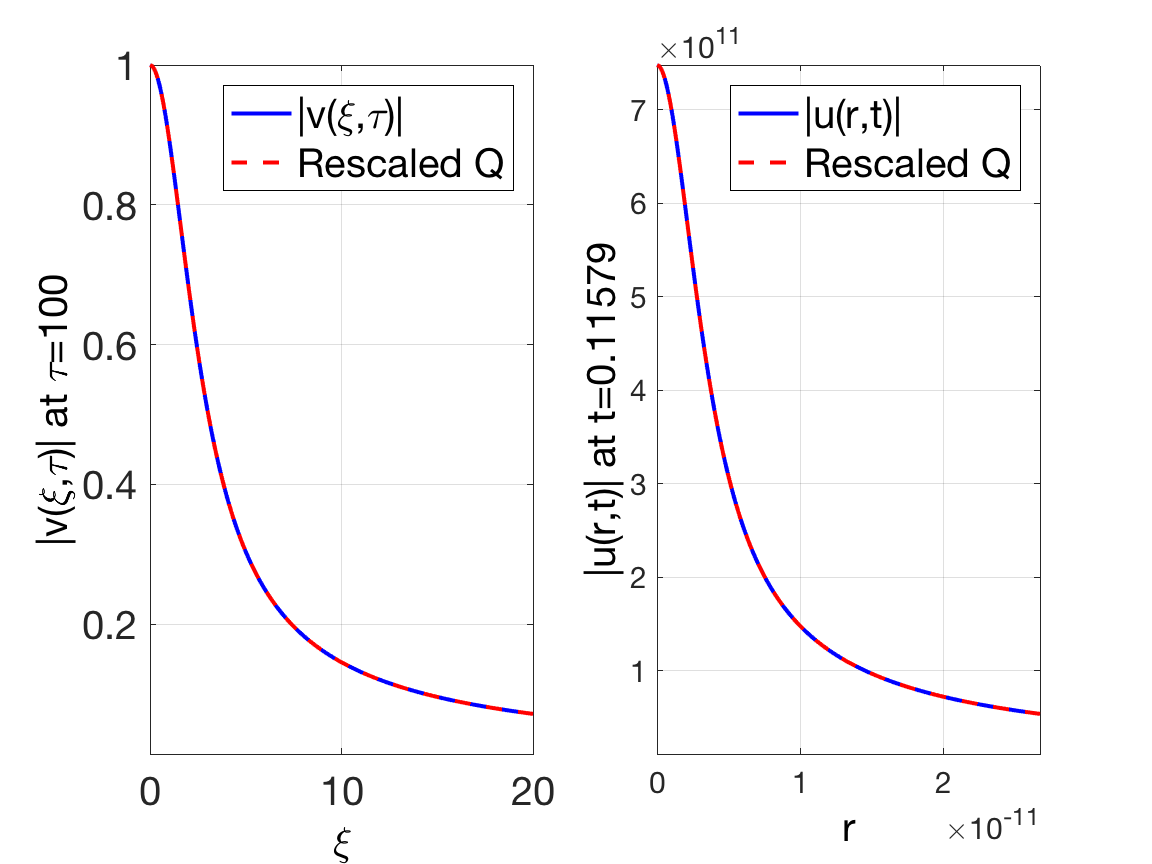}
\caption{ Blow-up profiles for the 3d cubic case at different times $\tau$ and $t$.}
\label{3d3p profiles}
%
\includegraphics[width=0.42\textwidth]{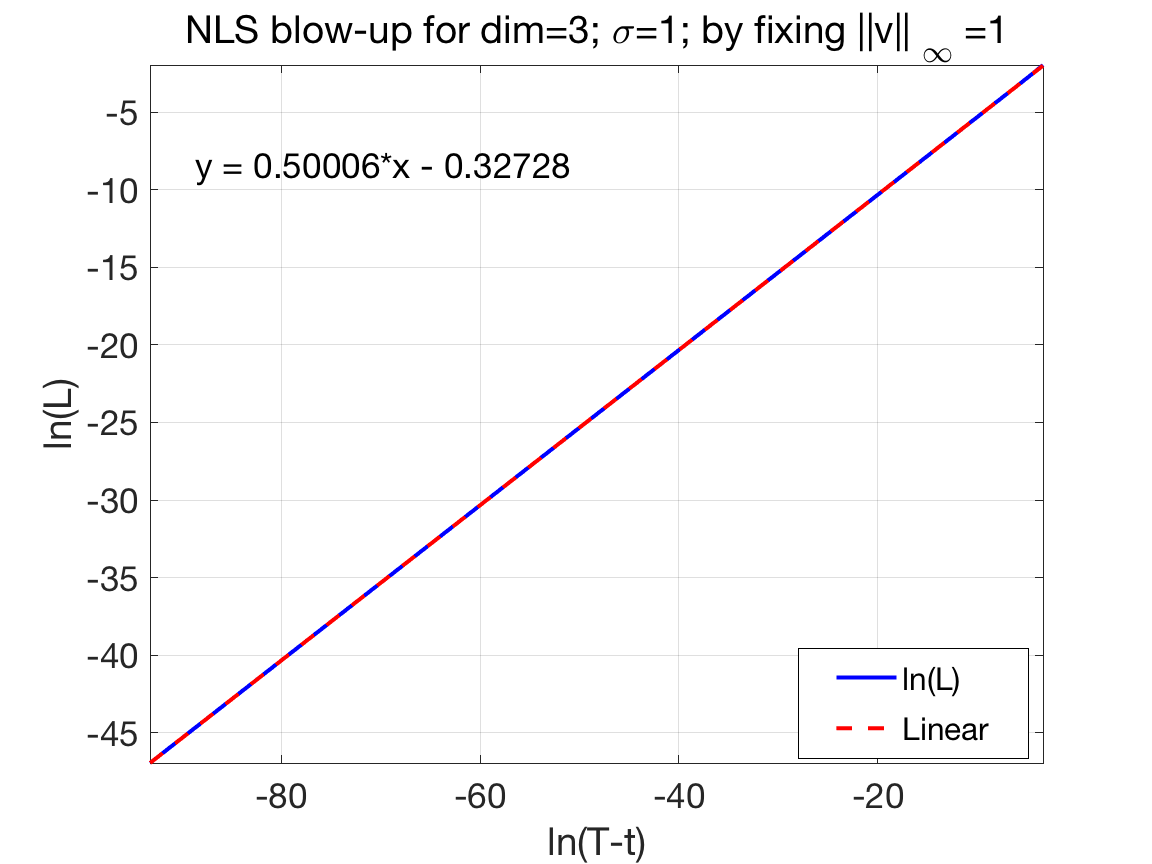}
\includegraphics[width=0.42\textwidth]{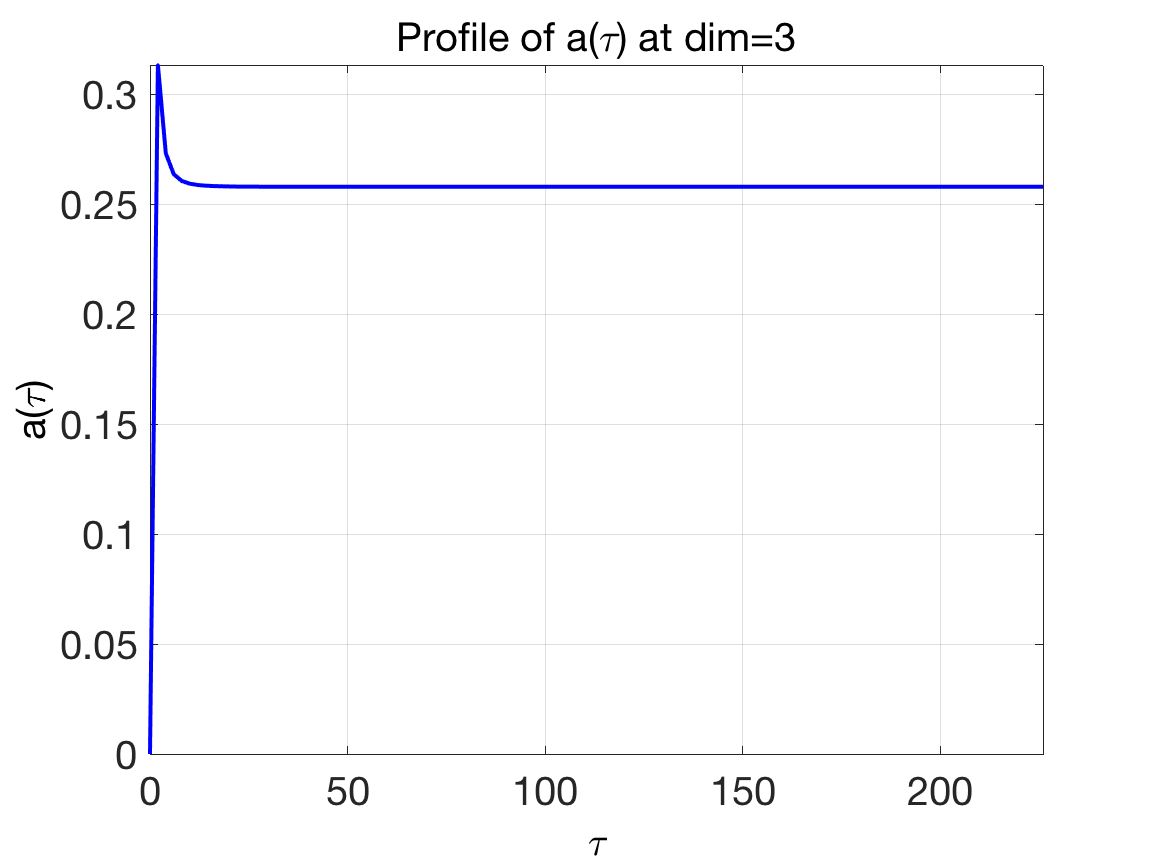}
\includegraphics[width=0.42\textwidth]{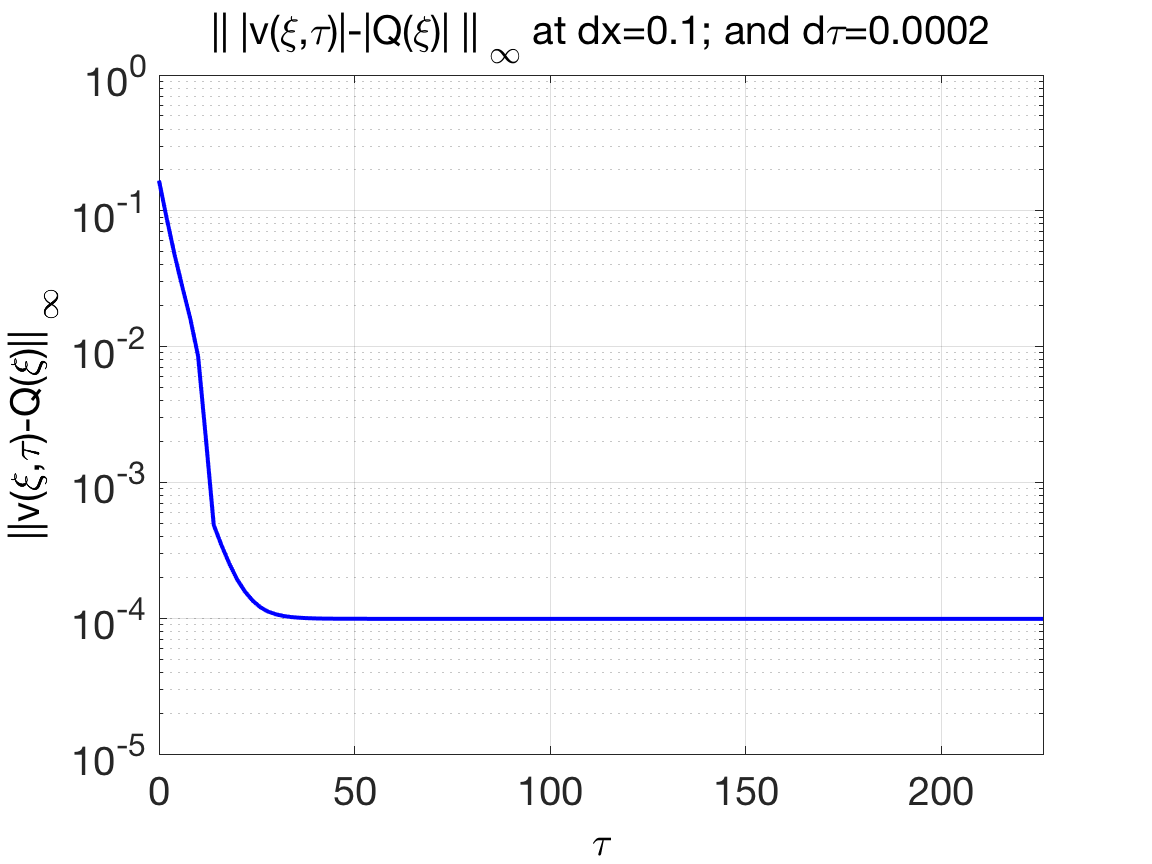}
\includegraphics[width=0.42\textwidth]{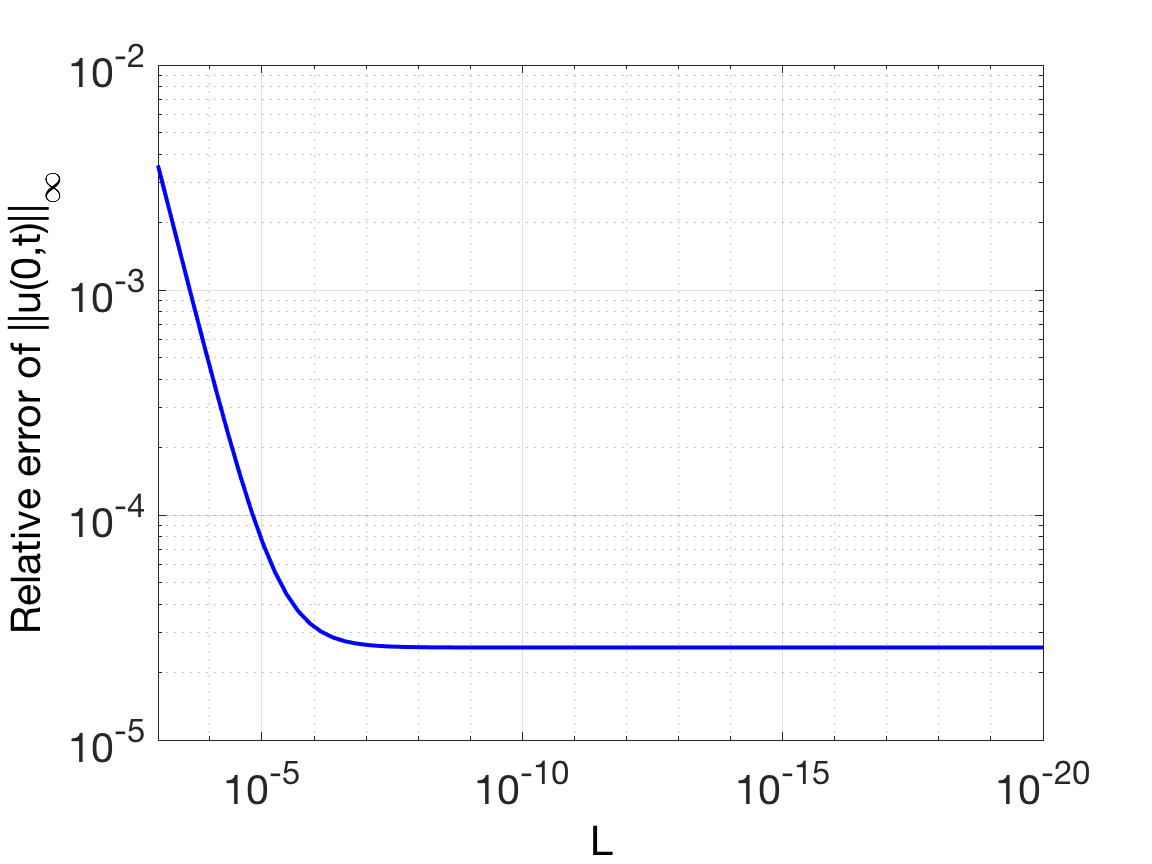}
\caption{ Blow-up data for the 3d cubic case: $\ln(T-t)$ vs. $\ln(L)$ (upper left), the quantity $a(\tau)$ (upper right), the distance between $Q$ and $v$ on time $\tau$ ($\| |v(\tau)| -|Q| \|_{L^{\infty}_{\xi}}$) (lower left), the relative error with respect to the predicted blow-up rate (lower right).  }
\label{3d3p data}
\end{center}
\end{figure}

\begin{figure}
\begin{center}
\includegraphics[width=0.42\textwidth]{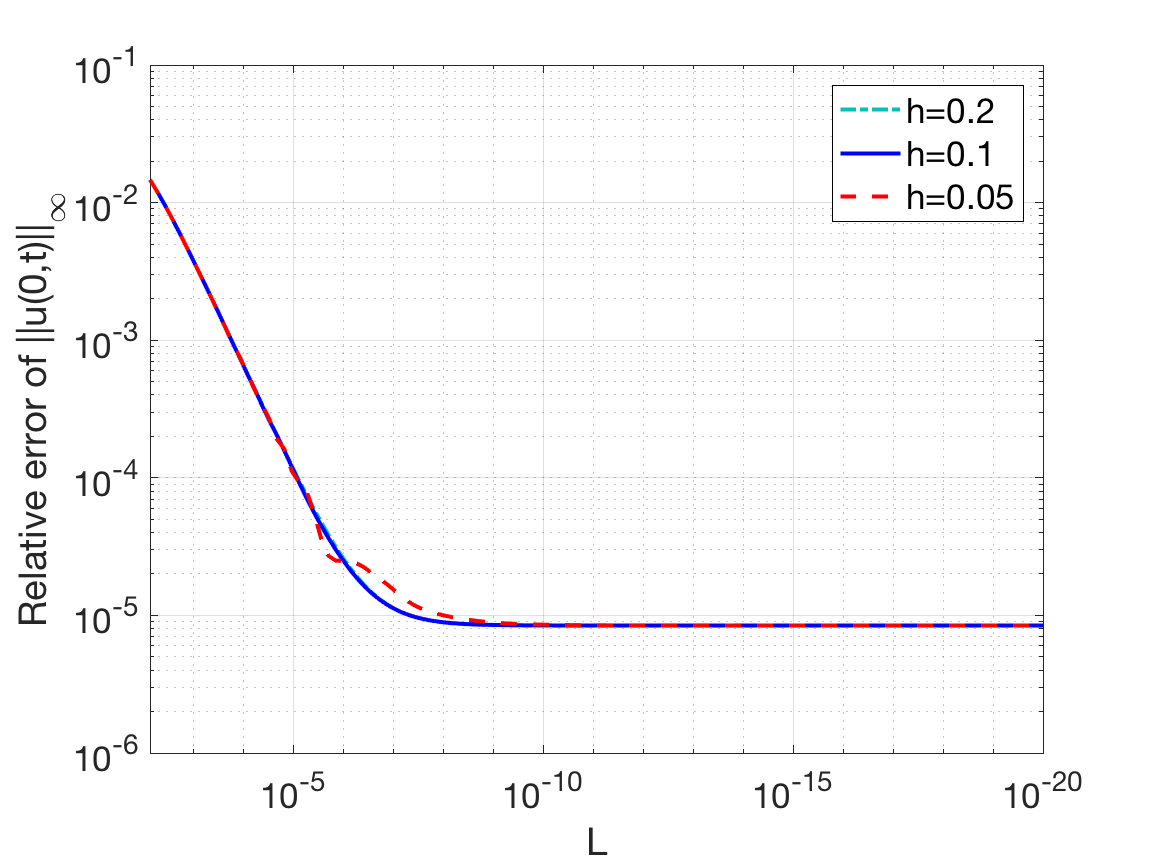}
\includegraphics[width=0.42\textwidth]{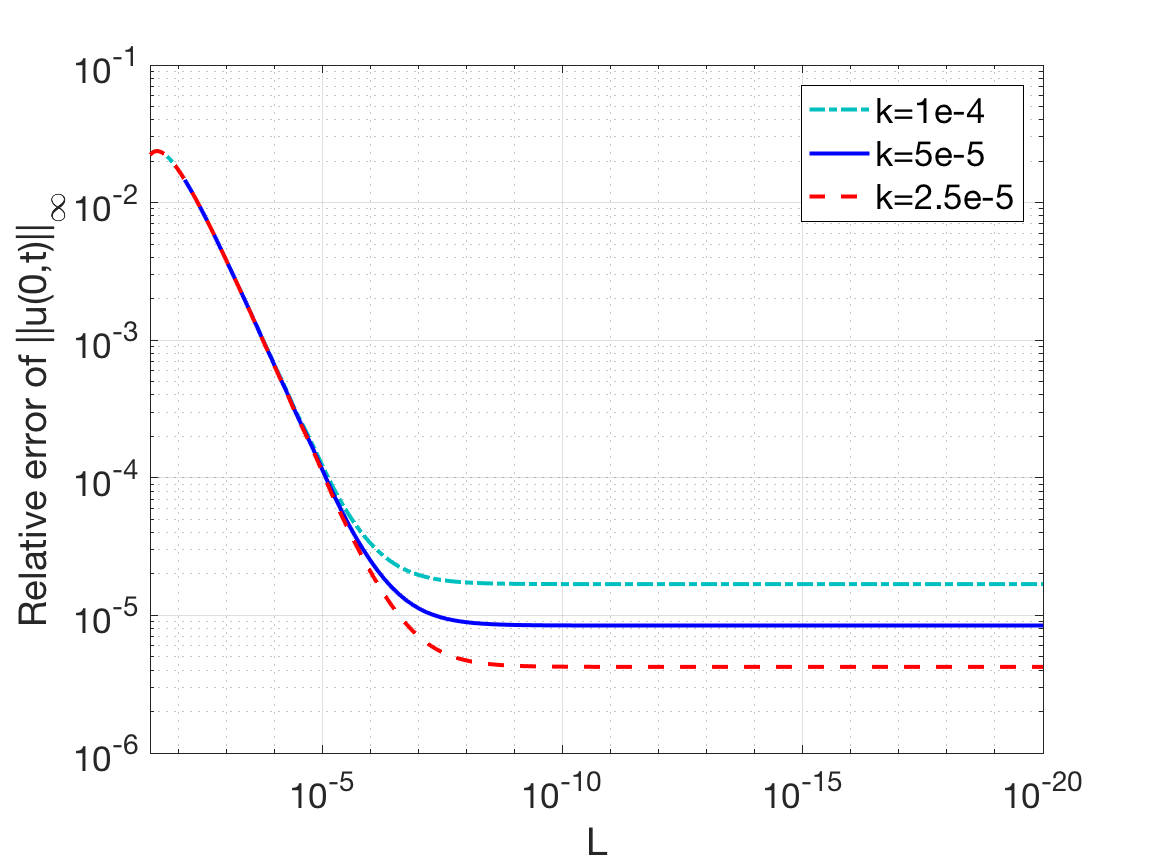}
\caption{ The relative error comparison when choosing different step sizes in space and time for the 3d quintic ($\sigma=2$) case.}
\label{NLS 3d5p error analysis}
\end{center}
\end{figure}

\subsection{Numerical results.} We now present our numerical results for different dimensions and nonlinearities.

We first present the {\it cubic} case in dimensions $d=3,4,5$. These cases include $s_c<1$, $s_c=1$ and $s_c>1$, respectively.
Figure \ref{3d3p profiles} shows the blow-up profiles at different times $\tau$ (or $t$). One can see that the solution converges to the predicted blow-up profile $Q$, up to scaling, as $\tau \rightarrow \infty$ (or $t \rightarrow T$). The two profiles become nearly indistinguishable after $\tau = 10$. The top left subplot in Figure \ref{3d3p data} shows that the slope of $L$ on log scale is still $\frac{1}{2}$. The top right subplot in Figure \ref{3d3p data} shows that the parameter $a(\tau)$ converges to a constant right away (before $\tau=10$). This indicates that before $\tau=10$, the solution enters the self-similar blow-up regime and explains why the profile $v(\xi,\tau)$ and the rescaled $Q$ become indistinguishable right after $\tau=10$. Note that it is not the case for the mass-critical case, where the ``log-log" blow-up regime can only be reached at extremely high focusing levels (see e.g. \cite{SS1999}, \cite{F2015}, \cite{RYZ2017}).
The bottom left subplot in Figure \ref{3d3p data} shows the distance between $|v(\xi,\tau)|$ and the rescaled $|Q(\xi)|$. The quantity stabilizes before $\tau=10$, which agrees with the time when the quantity $a(\tau)$ stabilizes in the top right subplot. The bottom right subplot in Figure \ref{3d3p data} shows the relative error $\mathcal{E}_{rel}$. This quantity $\mathcal{E}_{rel}$ stabilizes at the focusing level $L \sim 10^{-6}$ with a satisfactory order (around $10^{-5}$). This shows that there is no such ``adiabatic" regime occurring before the self-similar regime as the case for the $L^2$-critical NLS (see e.g., \cite{FP1998}, \cite{FP2000}, \cite{RYZ2017}).

Figures \ref{4d3p profiles} and \ref{4d3p data} show the results for the 4d cubic case. Both the blow-up profiles and quantities we track are similar to the 3d cubic case. Compared with the 3d cubic case in Figures 
\ref{3d3p profiles} and \ref{3d3p data}, the blow-up profiles become indistinguishable around 
$\tau=50$ (see Figure \ref{4d3p profiles}), and the relative error reaches the stabilized regime around the focusing level $L \approx 10^{-7}$ (see bottom right subplot in Figure \ref{4d3p data}). For the 5d cubic case, besides the similar results to the previous cases (Figures \ref{5d3p profiles} and \ref{5d3p data}), from the bottom right subplot in Figure \ref{5d3p data}, we see that the relative error reaches a stabilized regime around the focusing level $L \approx 10^{-8}$. 

We next present the results for the {\it quintic} case for dimensions $d=2,3,4$, see Figures \ref{2d5p profiles} to \ref{4d5p data}. We obtain the similar results for these cases as for the cubic cases above. However, by comparing the bottom right subplots in Figures \ref{2d5p data}, \ref{3d5p data} and \ref{4d5p data}, we can see that the relative error reaches the stabilized regime at the focusing level $10^{-6}$, $10^{-7}$ and $10^{-8}$, respectively. This shows that for a fixed nonlinearity, the higher dimension leads to the higher focusing level for the solution to reach the self-similar regime. This is also true for the cubic cases (see Figures \ref{3d3p data}, \ref{4d3p data} and \ref{5d3p data}).

Finally, we list the 3d {\it septic} ($s_c>1$) case in Figures \ref{3d7p profiles} and \ref{3d7p data}. While the profiles and results are still similar to the previous ones, we compare the results for different nonlinearities. Again, the solution reaches the self-similar regime around the focusing level $L=10^{-6}$ for the 3d cubic case, around $L=10^{-7}$ for the 3d quintic case and around $L=10^{-10}$ for the 3d septic case (see the bottom right subplot in Figure \ref{3d7p data}). This shows that for a fixed dimension, the higher nonlinearity leads to the higher focusing level for the solution to reach the self-similar regime.

\section{Conclusion}

We generalized the existence and local uniqueness theory of $Q$ for dimensions $d \geq 2$ and nonlinearities other than cubic, in the mass-supercritical NLS equations, $s_c>0$. 
From our numerical simulations, we found that the energy-subcritical, critical and energy-supercritical cases enjoy very similar stable blow-up dynamics.
We obtained the self-similar profiles to which stable blow-up solutions converge. Our numerical results show that $E[Q]=$ constant if $s_c=1$ and $E[Q]=-\infty$ if $s_c>1$. All these facts show that it maybe challenging to analyze theoretically the stable blow-up dynamics in the mass-supercritical setting. On the other hand, unlike the mass-critical case, certain features of the stable blow-up are easier to exhibit, for example, blow-up solutions converge to the predicted blow-up rate very fast, which can be simply observed numerically, and thus, we do not need to use extended and involved asymptotic analysis on solutions to get information about the rates such as an ``adiabatic" regime or the ``log-log" regime, which appear in the mass-critical case.

\begin{figure}
\begin{center}
\includegraphics[width=0.42\textwidth]{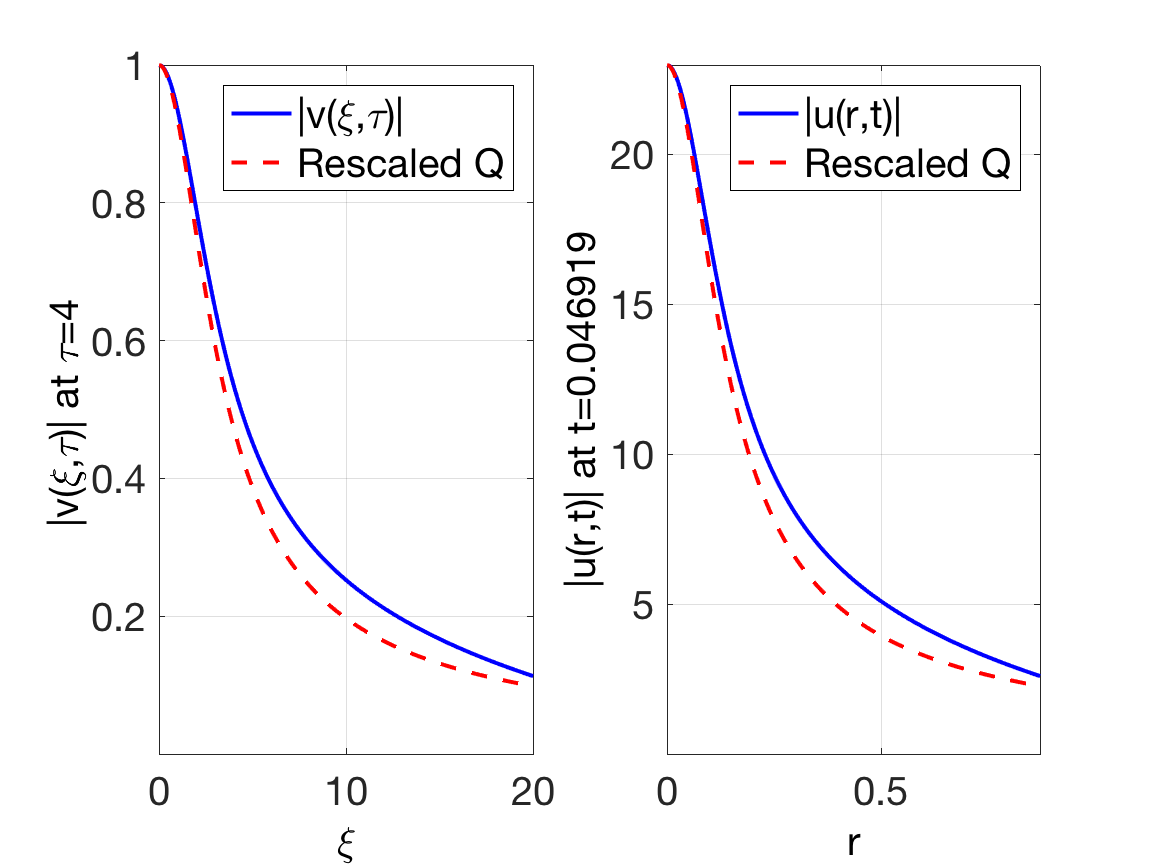}
\includegraphics[width=0.42\textwidth]{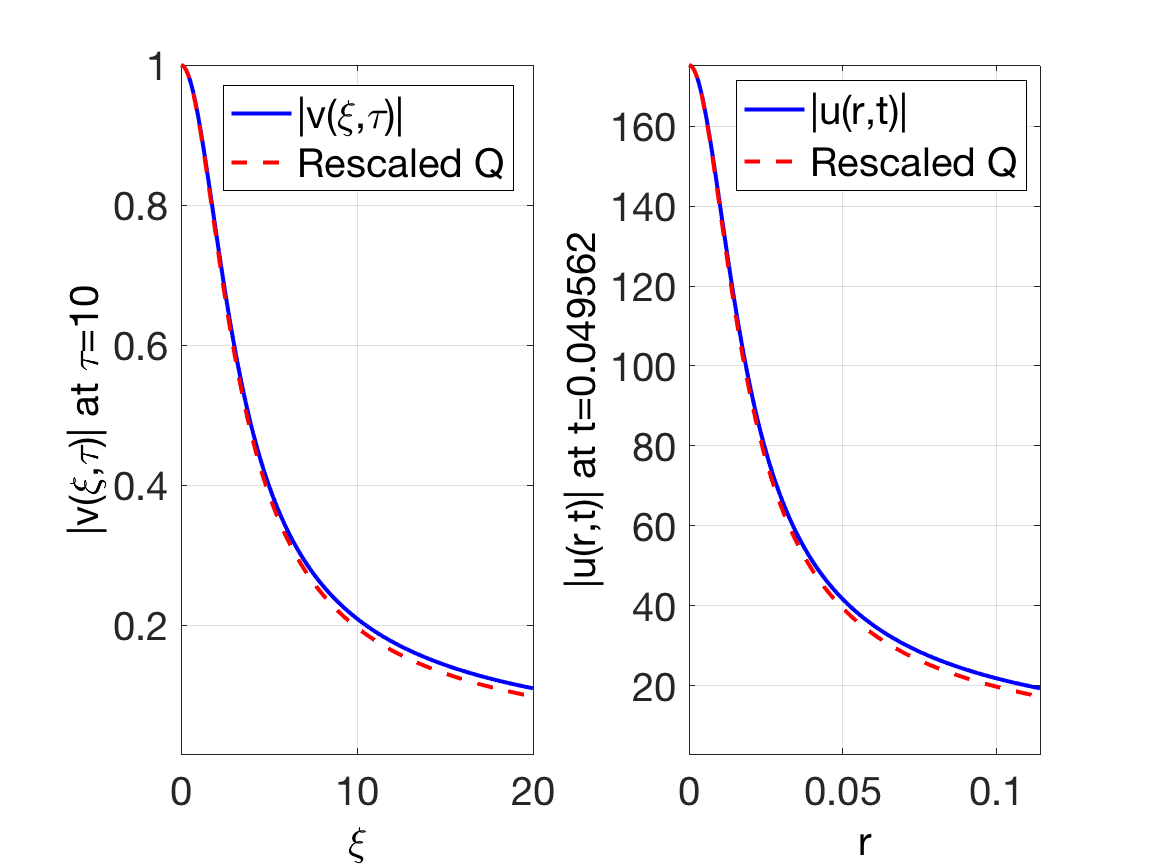}
\includegraphics[width=0.42\textwidth]{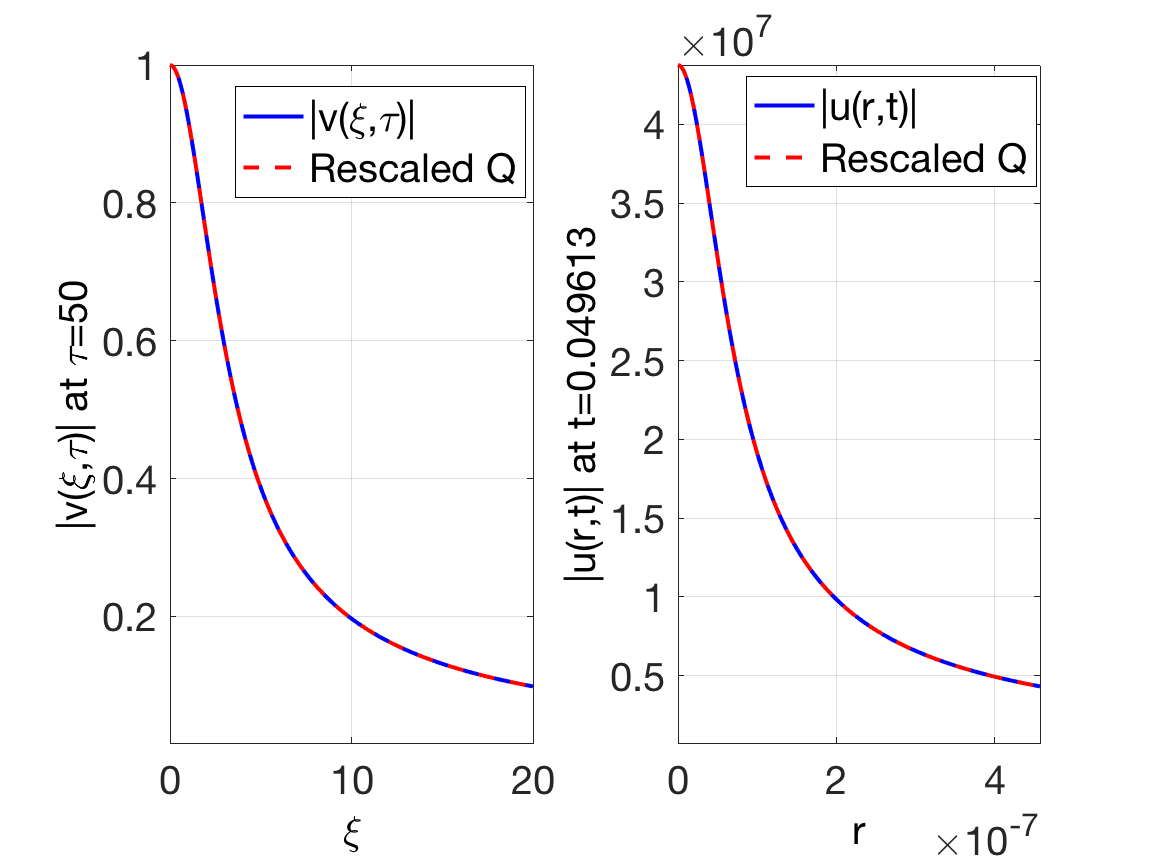}
\includegraphics[width=0.42\textwidth]{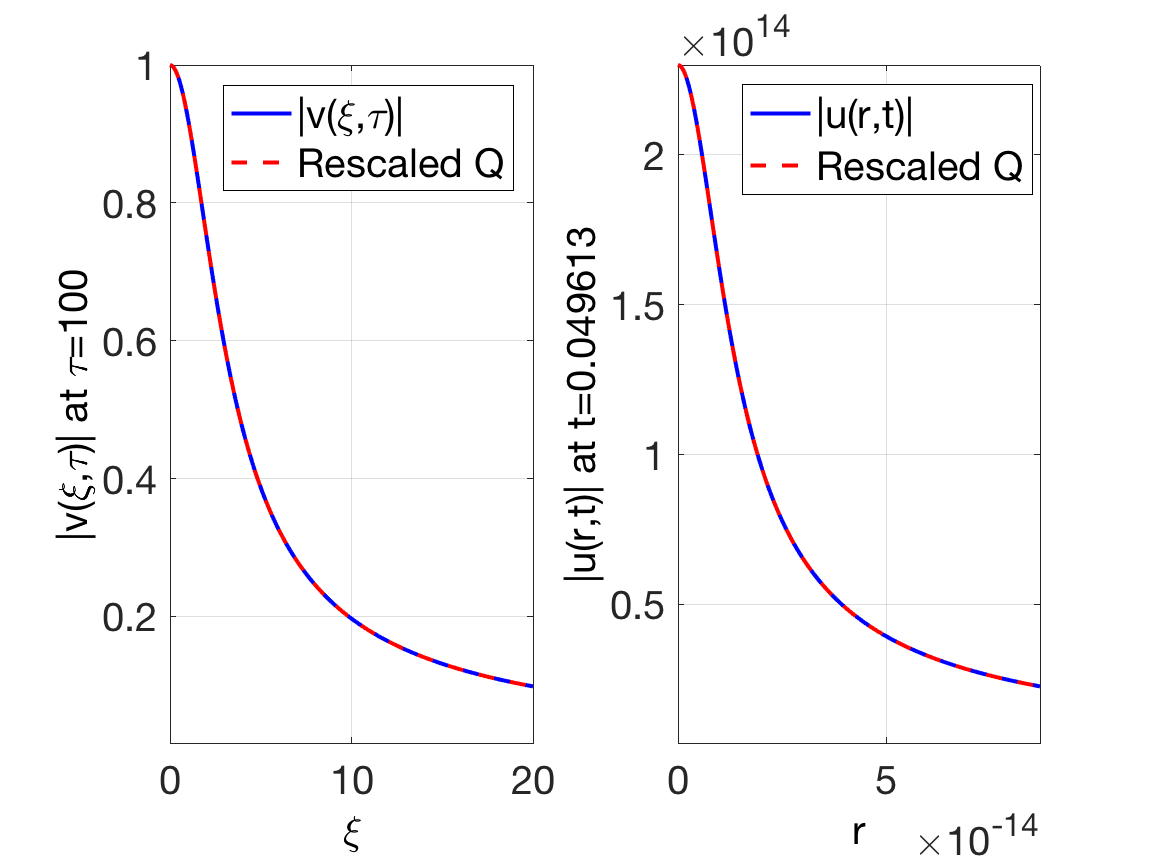}
\caption{ Blow-up profiles for the 4d cubic case at different time $\tau$ and $t$.}
\label{4d3p profiles}
%
\includegraphics[width=0.42\textwidth]{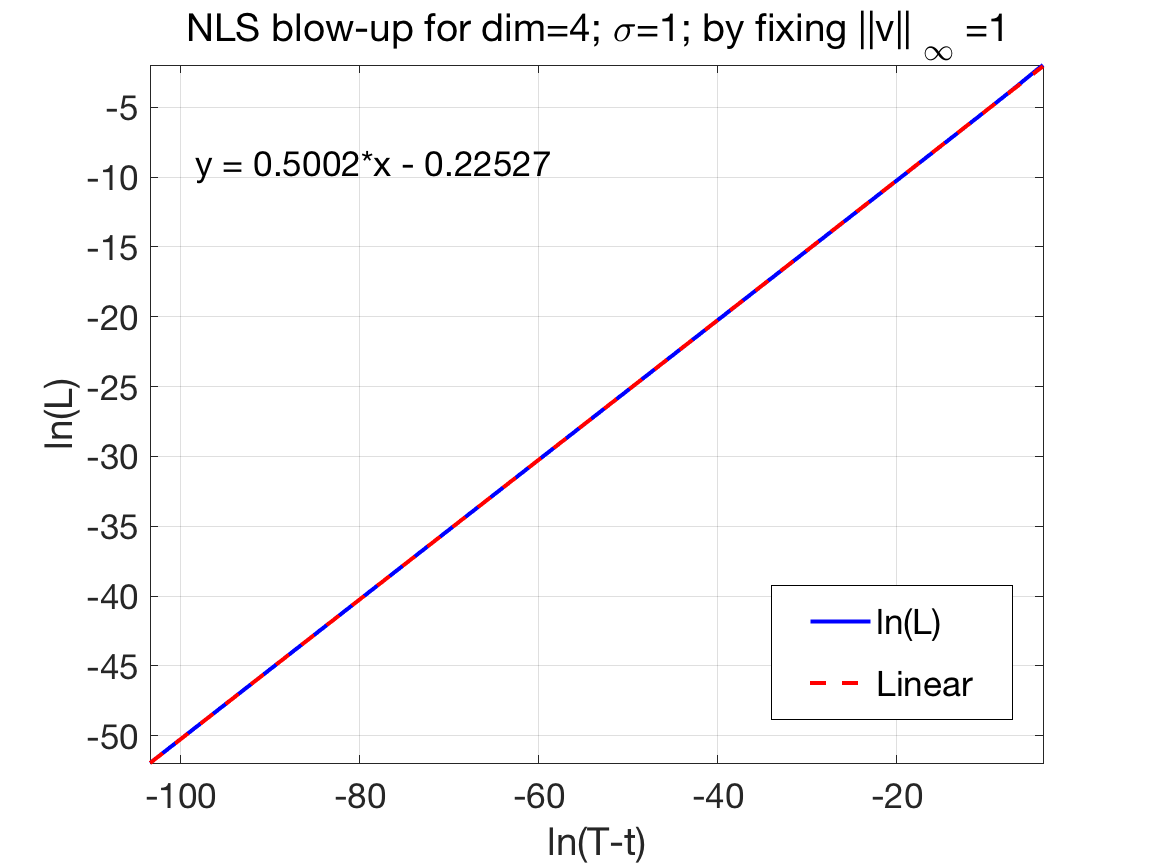}
\includegraphics[width=0.42\textwidth]{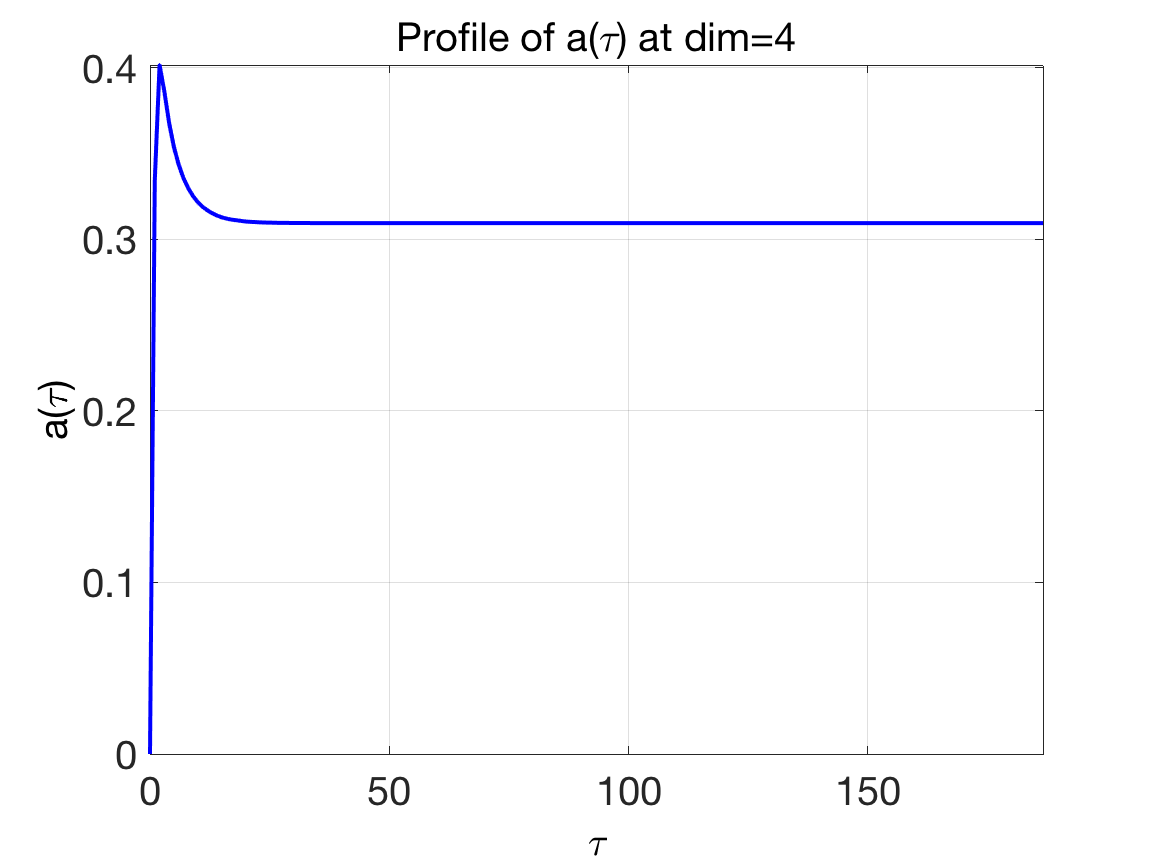}
\includegraphics[width=0.42\textwidth]{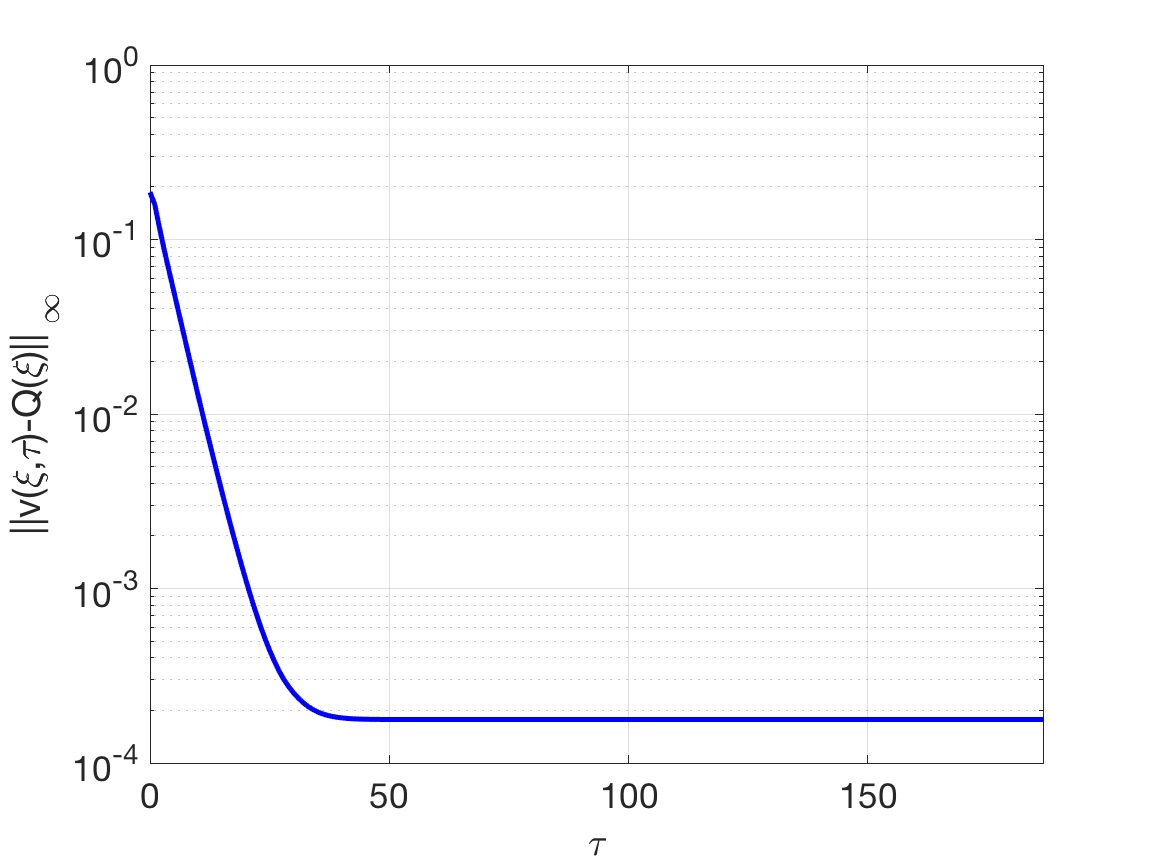}
\includegraphics[width=0.42\textwidth]{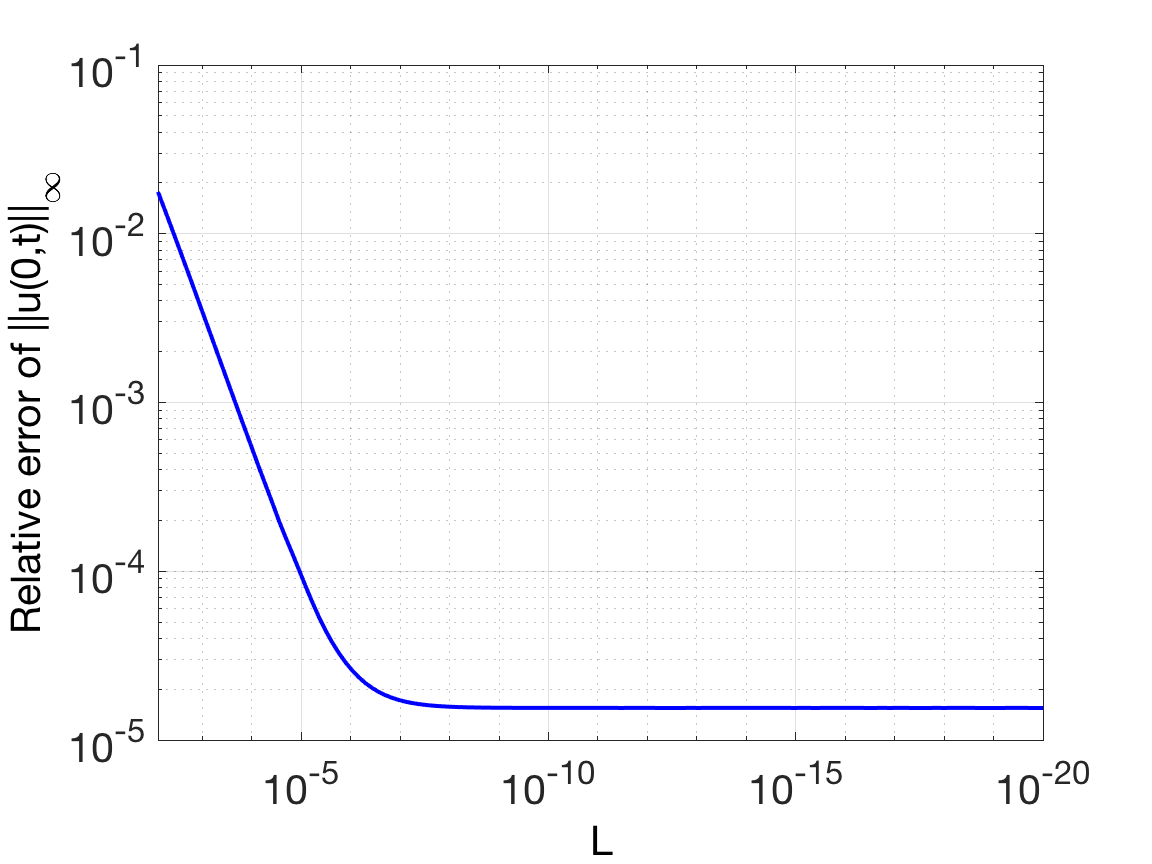}
\caption{ Blow-up data for the 4d cubic case: $\ln(T-t)$ vs. $\ln(L)$ (upper left), the quantity $a(\tau)$ (upper right), the distance between $Q$ and $v$ on time $\tau$ ($\| |v(\tau)| -|Q| \|_{L^{\infty}_{\xi}}$) (lower left), the relative error with respect to the predicted blow-up rate (lower right).}
\label{4d3p data}
\end{center}
\end{figure}

\begin{figure}
\begin{center}
\includegraphics[width=0.42\textwidth]{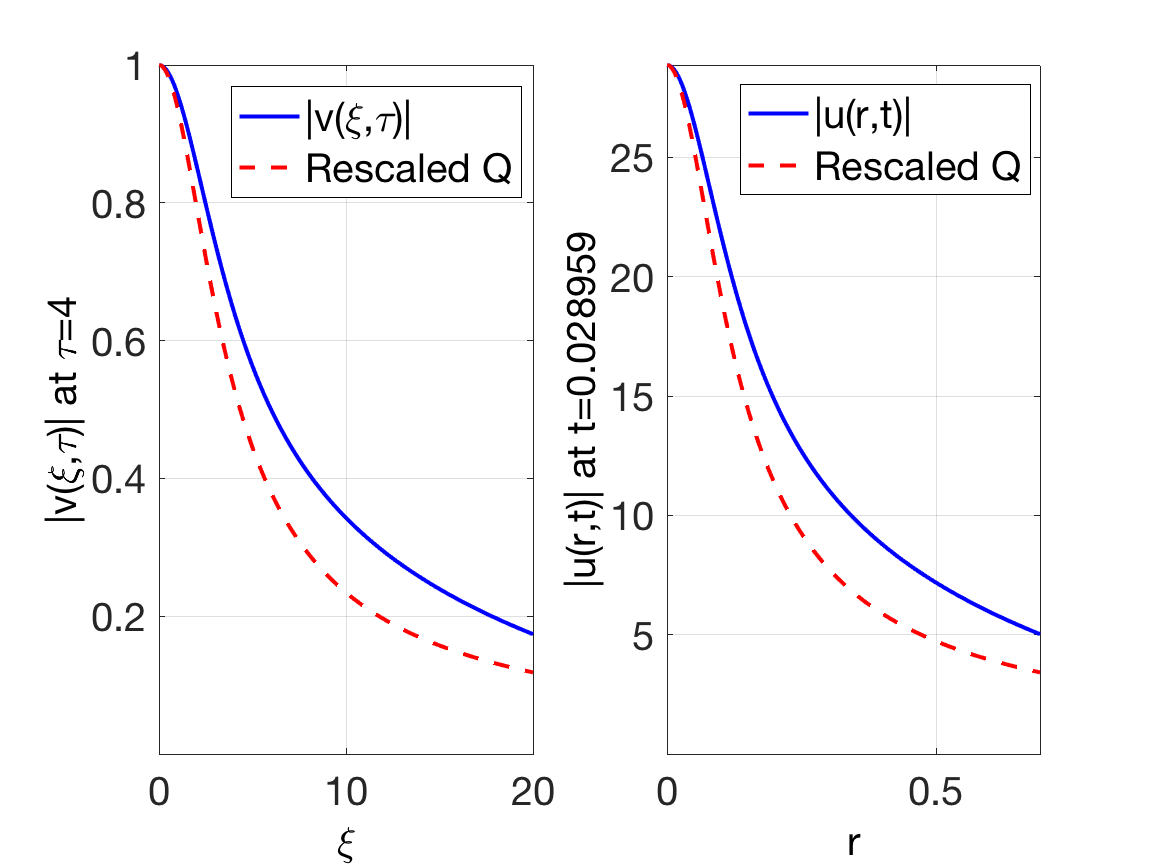}
\includegraphics[width=0.42\textwidth]{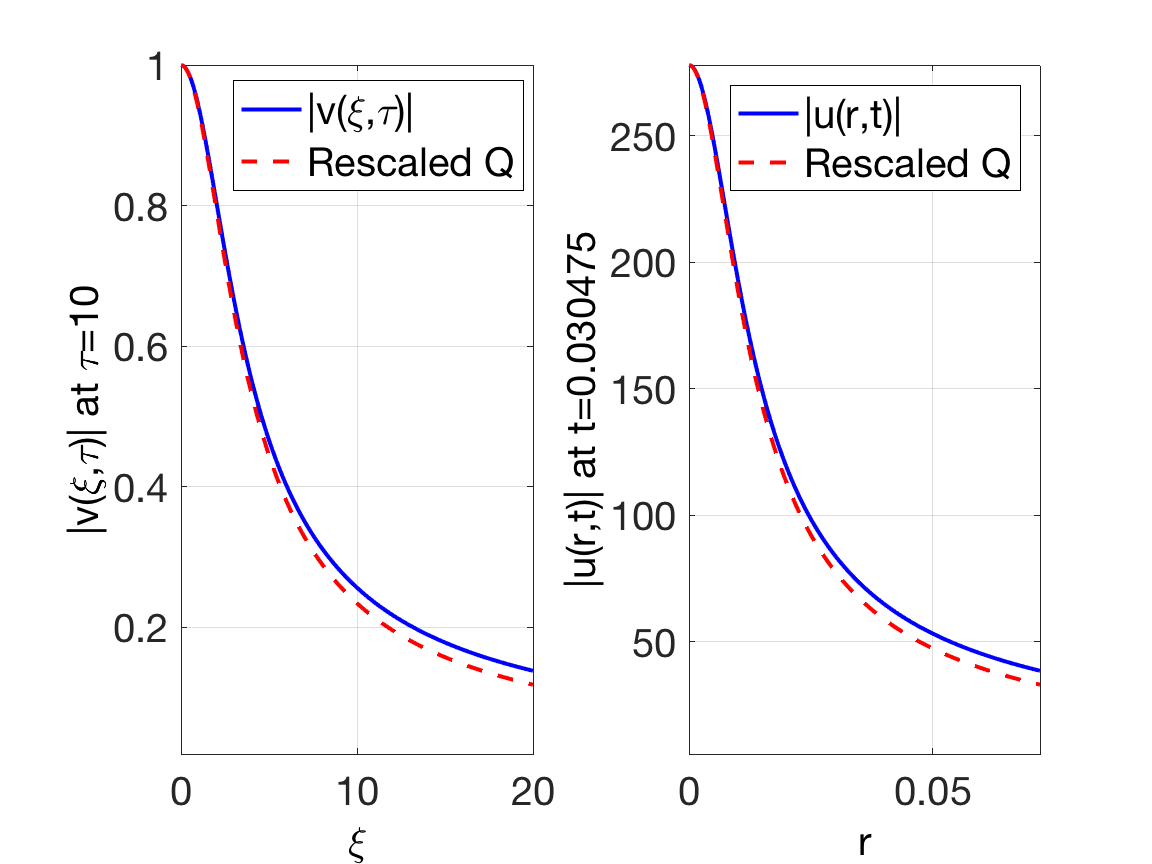}
\includegraphics[width=0.42\textwidth]{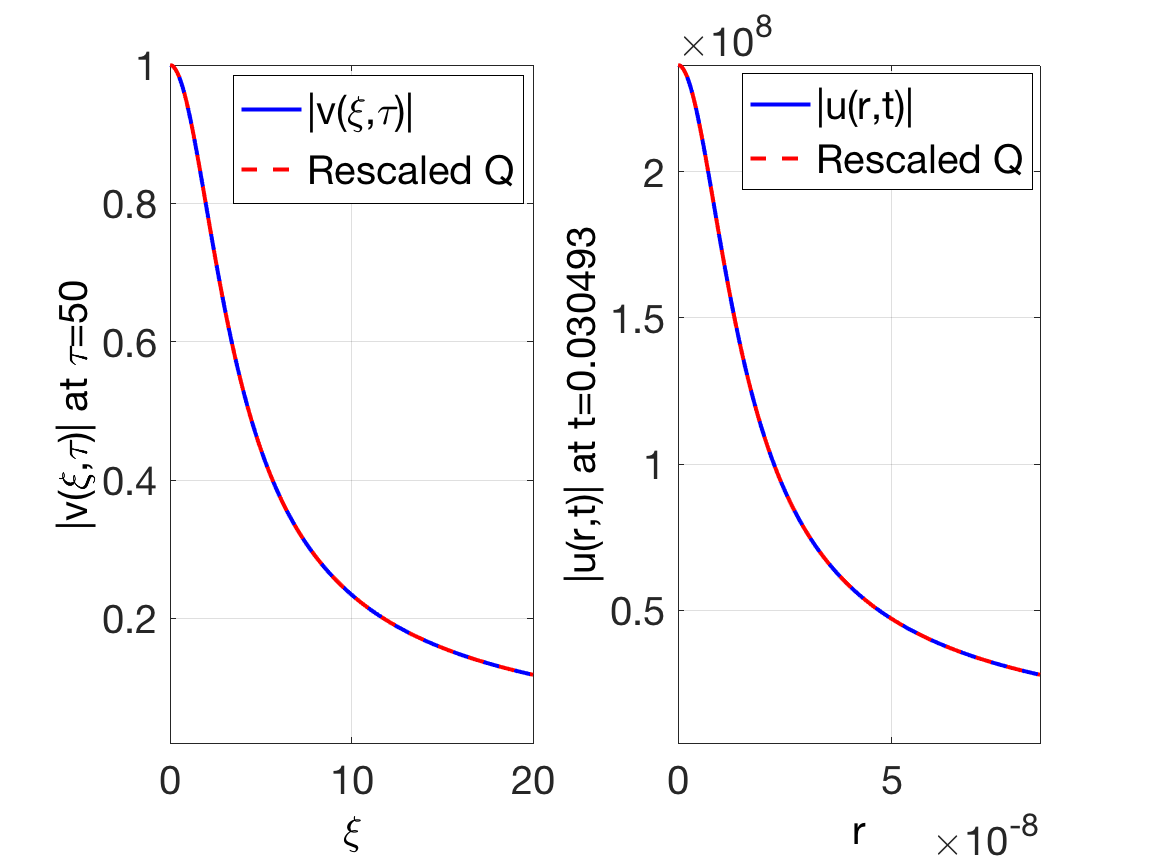}
\includegraphics[width=0.42\textwidth]{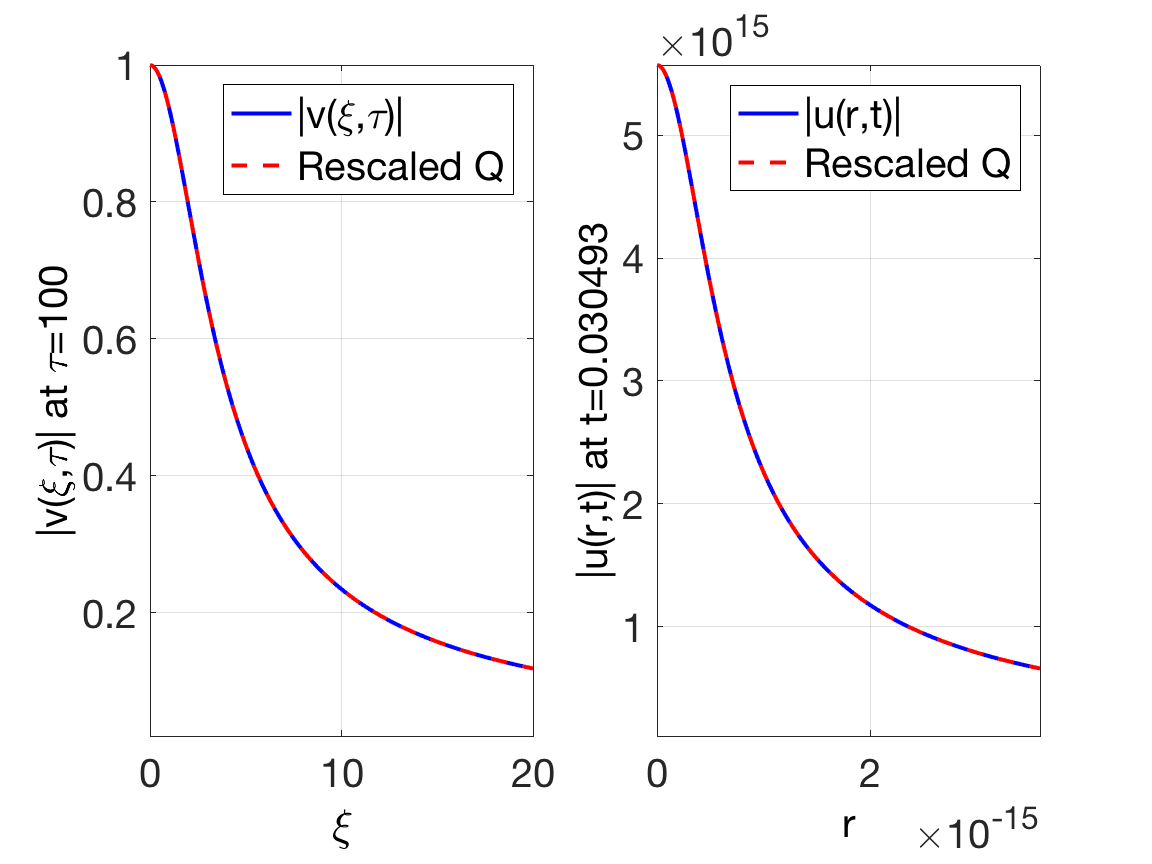}
\caption{ Blow-up profiles for the 5d cubic case at different time $\tau$ and $t$.}
\label{5d3p profiles}
%
\includegraphics[width=0.42\textwidth]{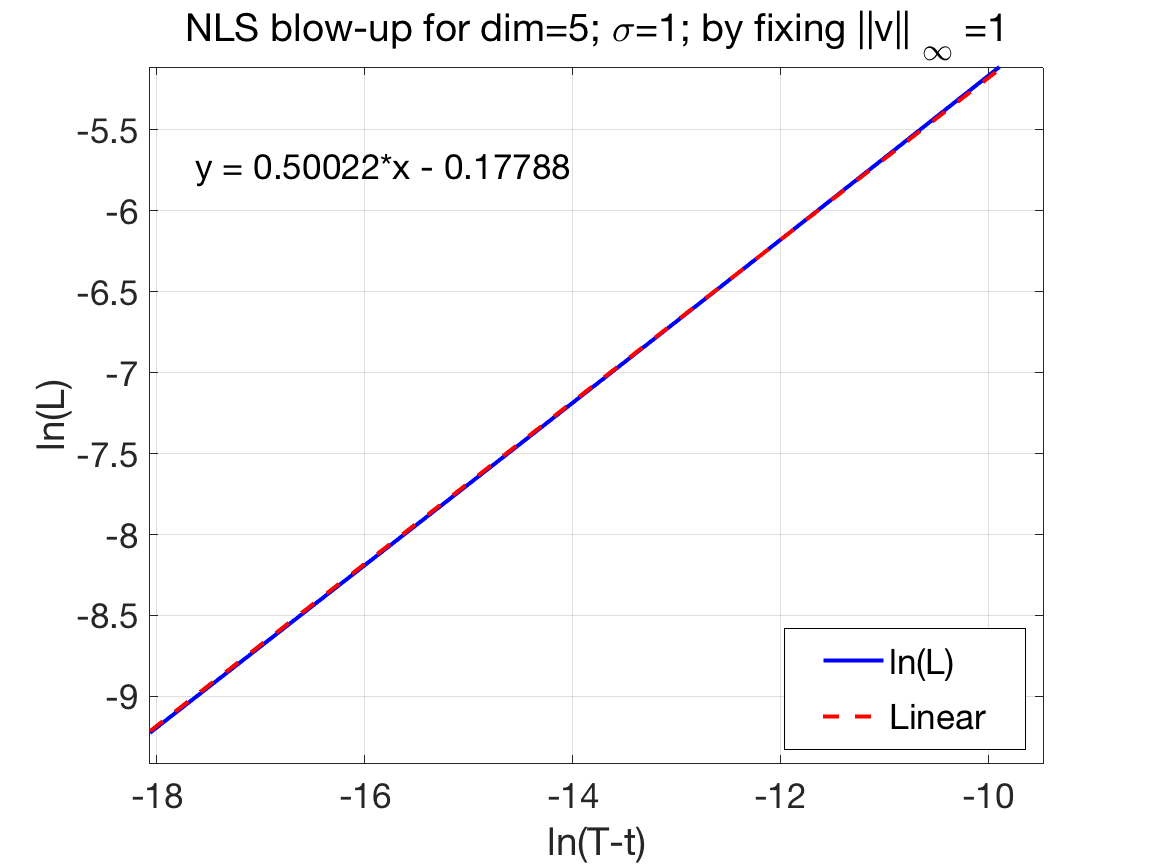}
\includegraphics[width=0.42\textwidth]{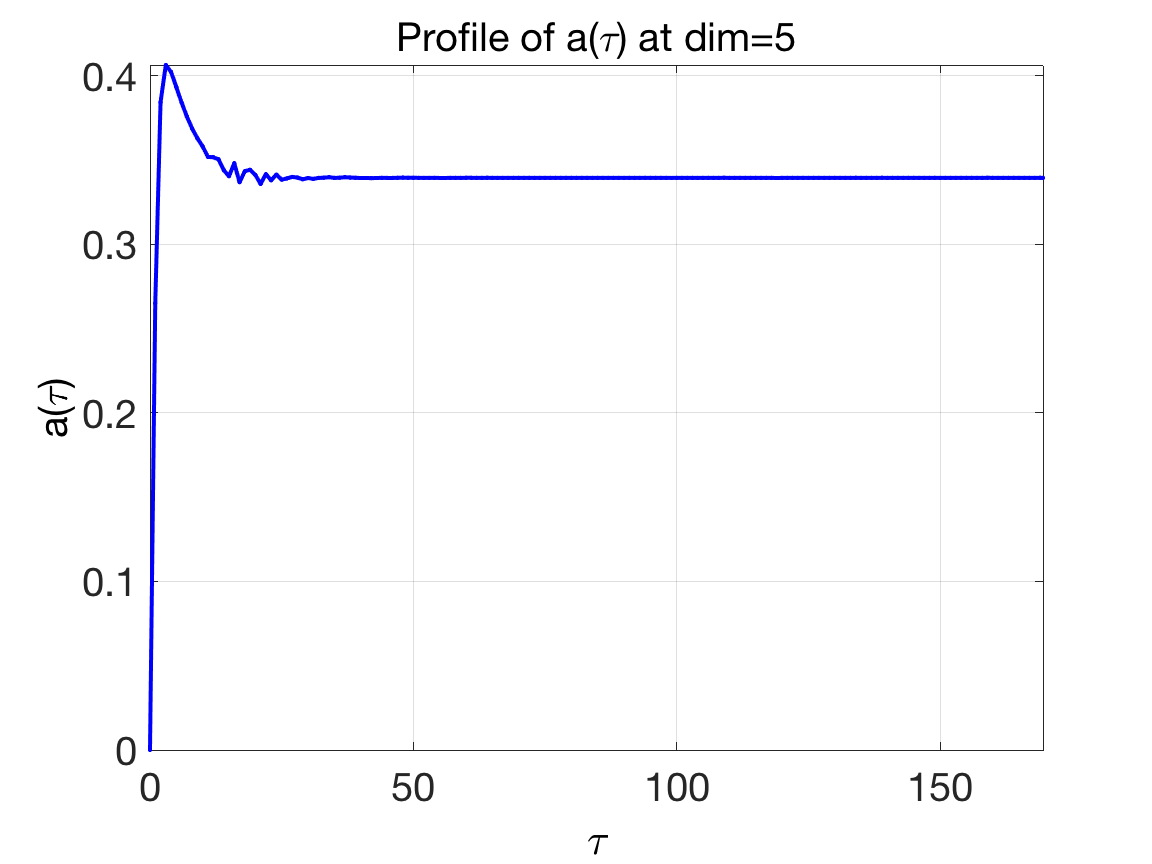}
\includegraphics[width=0.42\textwidth]{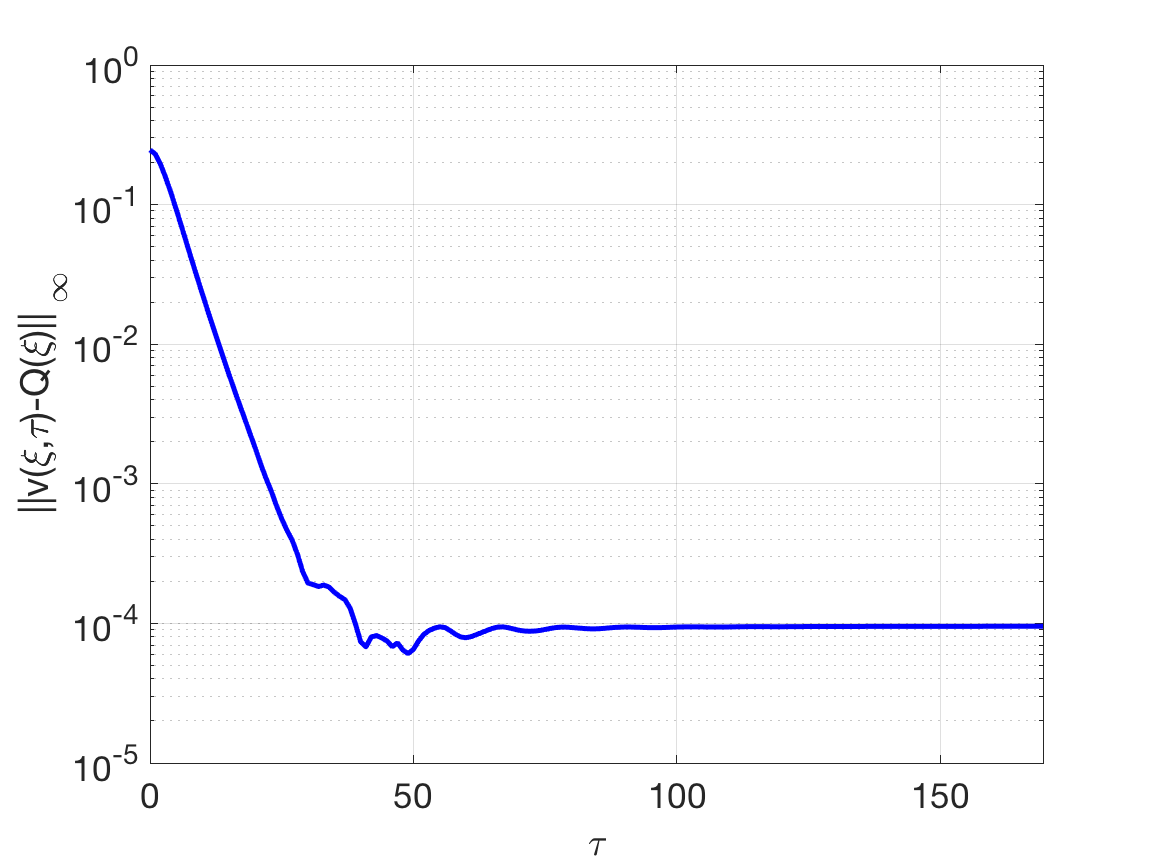}
\includegraphics[width=0.42\textwidth]{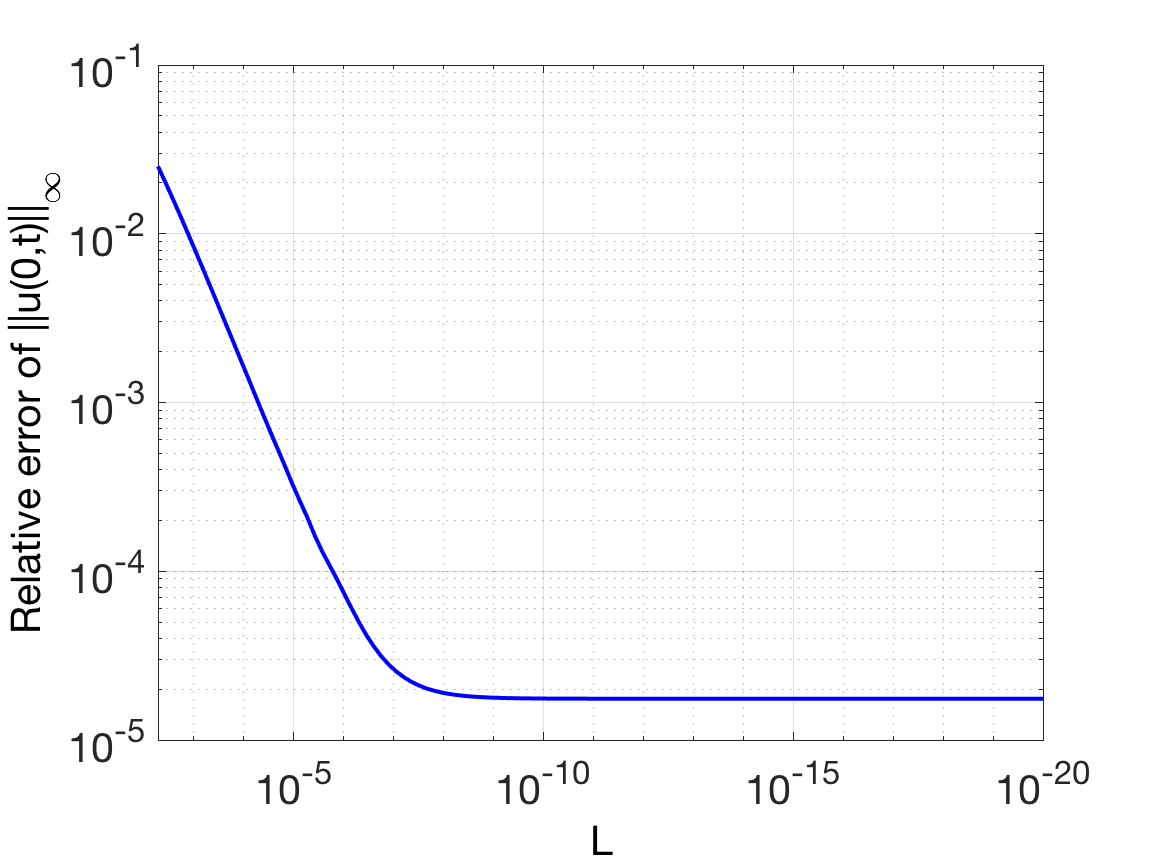}
\caption{ Blow-up data for the 5d cubic case: $\ln(T-t)$ vs. $\ln(L)$ (upper left), the quantity $a(\tau)$ (upper right), the distance between $Q$ and $v$ on time $\tau$ ($\| |v(\tau)| -|Q| \|_{L^{\infty}_{\xi}}$) (lower left), the relative error with respect to the predicted blow-up rate (lower right).  }
\label{5d3p data}
\end{center}
\end{figure}

\begin{figure}
\begin{center}
\includegraphics[width=0.42\textwidth]{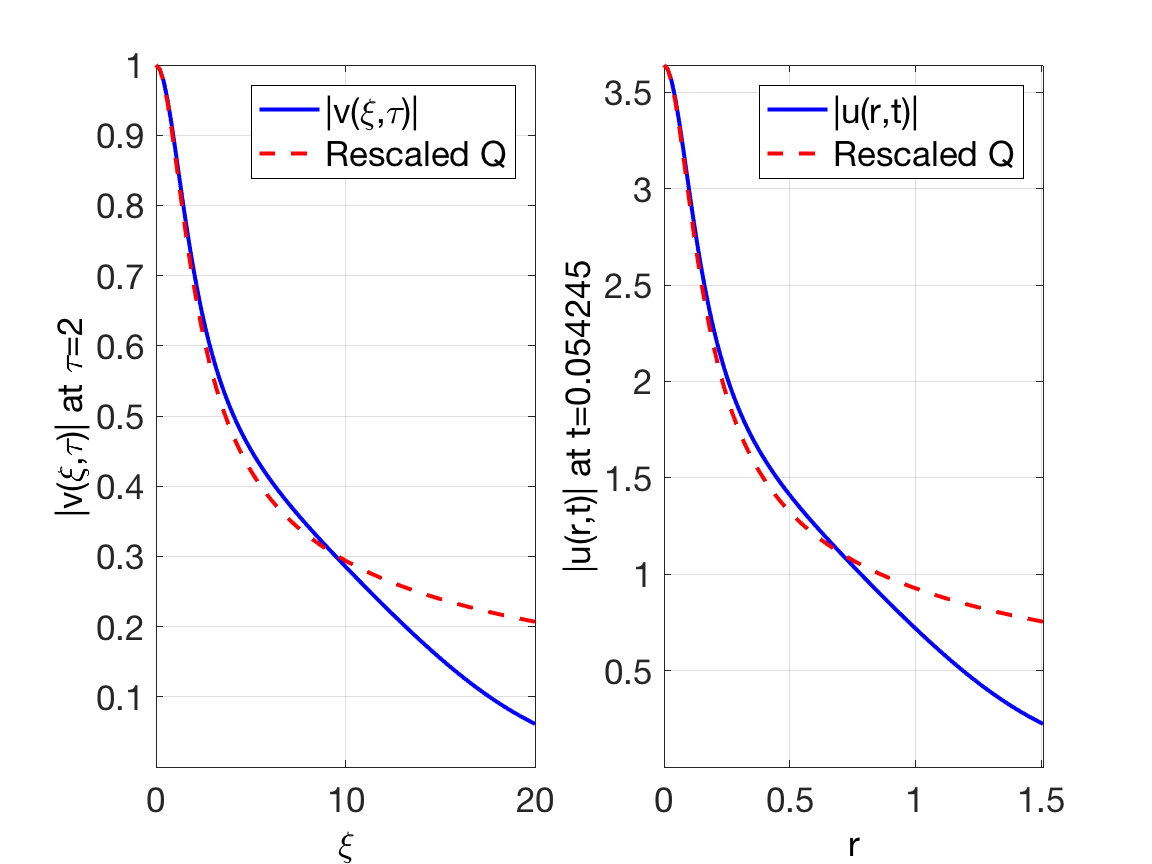}
\includegraphics[width=0.42\textwidth]{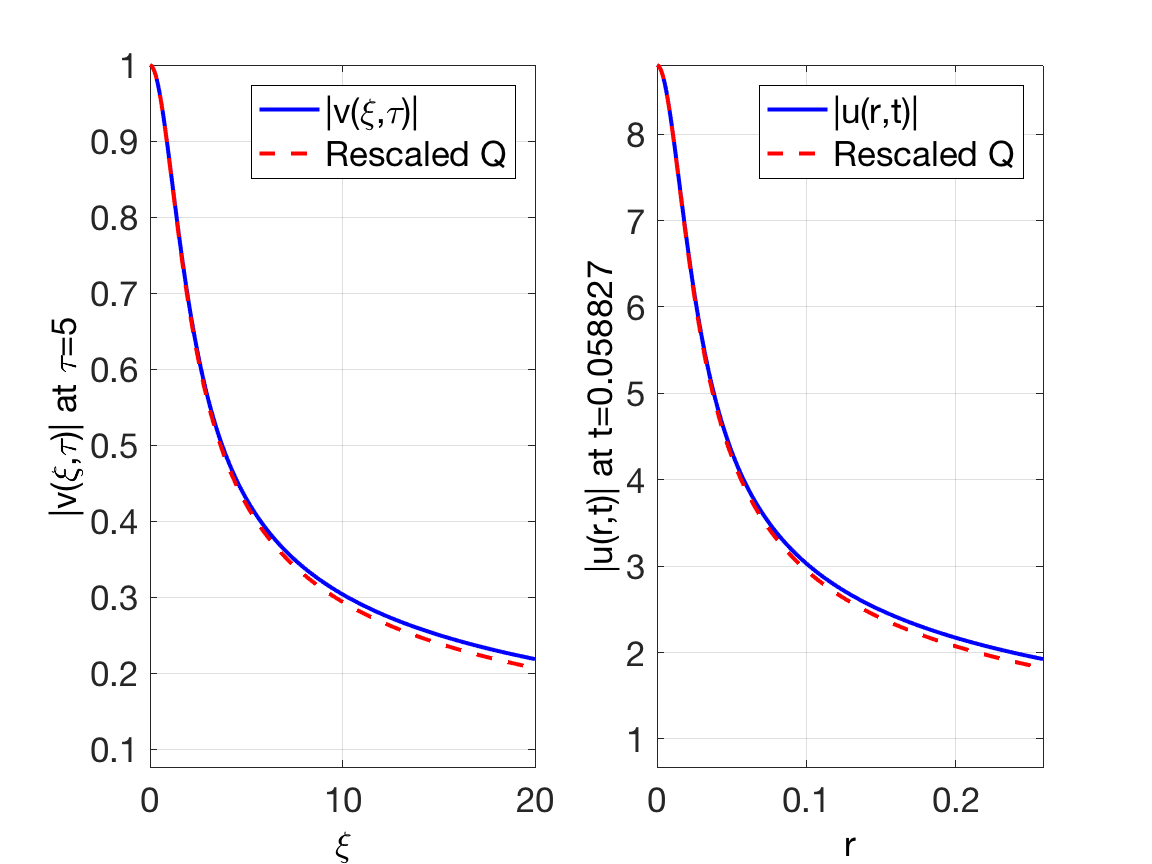}
\includegraphics[width=0.42\textwidth]{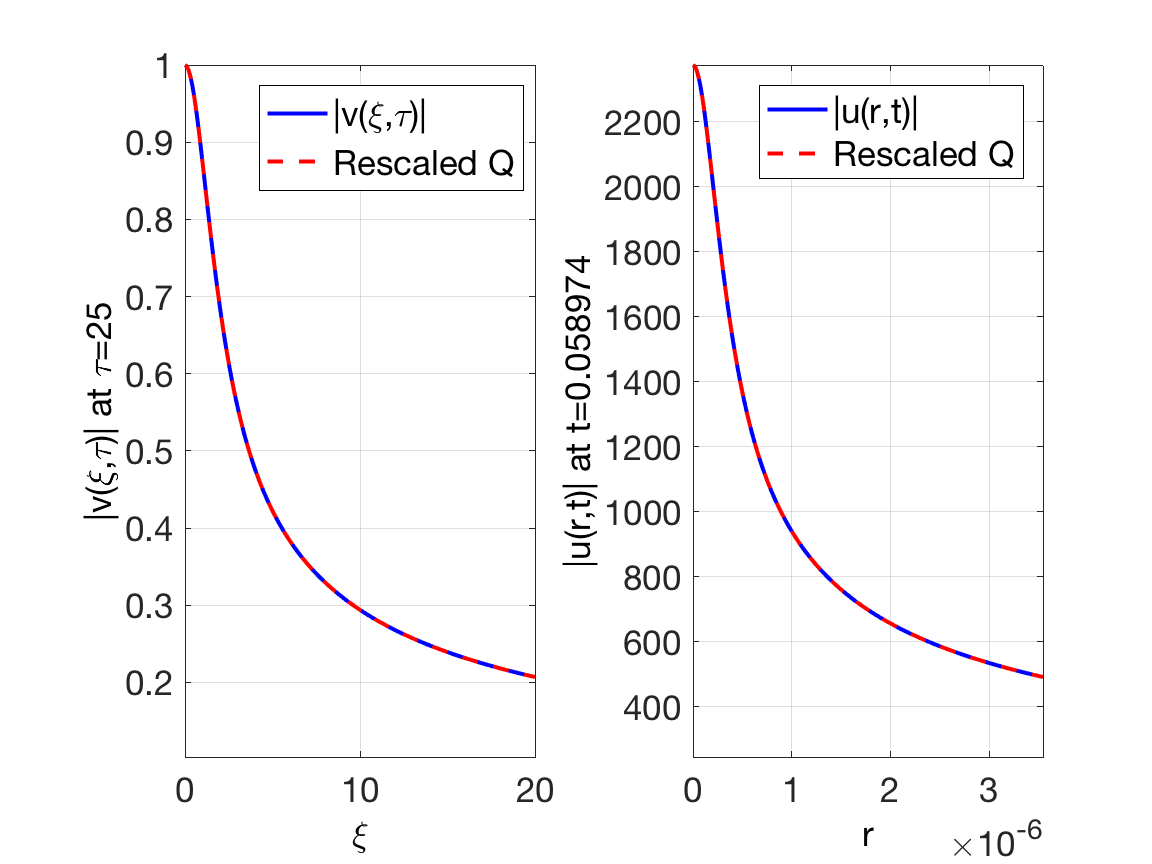}
\includegraphics[width=0.42\textwidth]{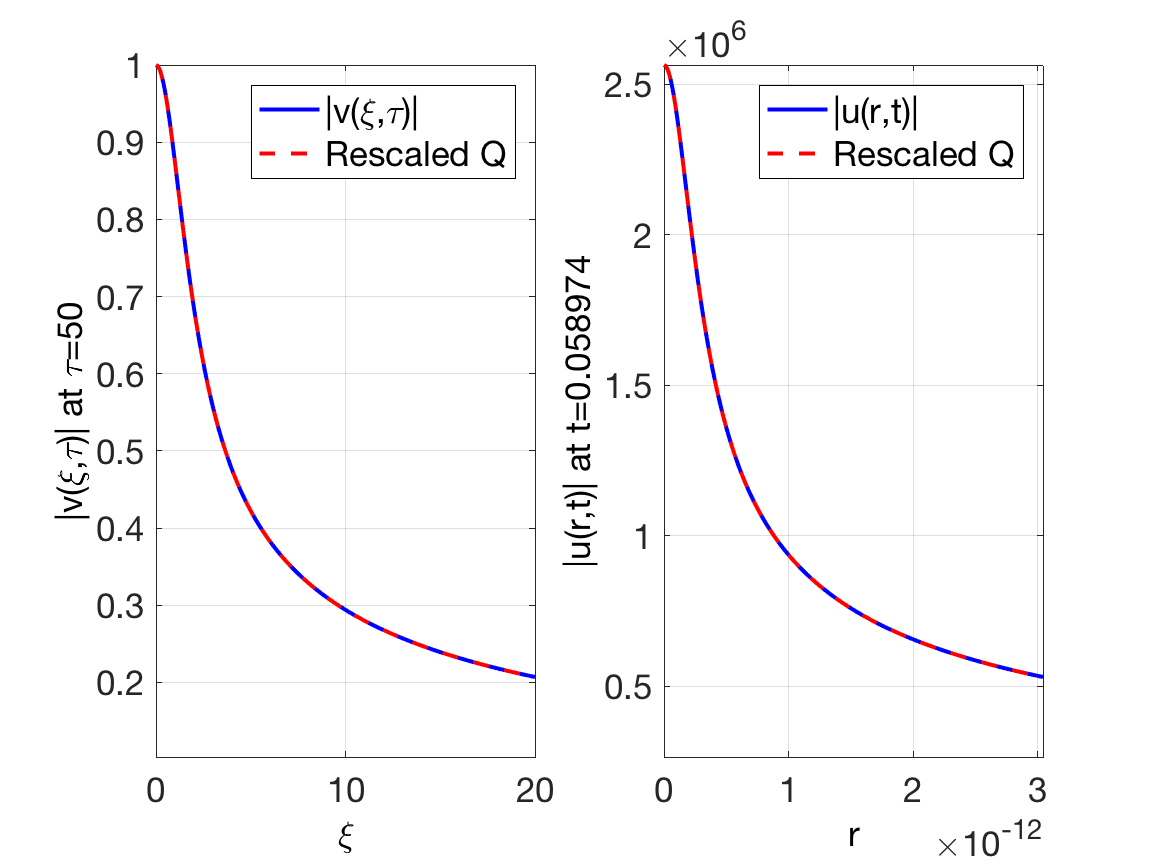}
\caption{ Blow-up profiles for the 2d quintic case at different time $\tau$ and $t$.}
\label{2d5p profiles}
%
\includegraphics[width=0.42\textwidth]{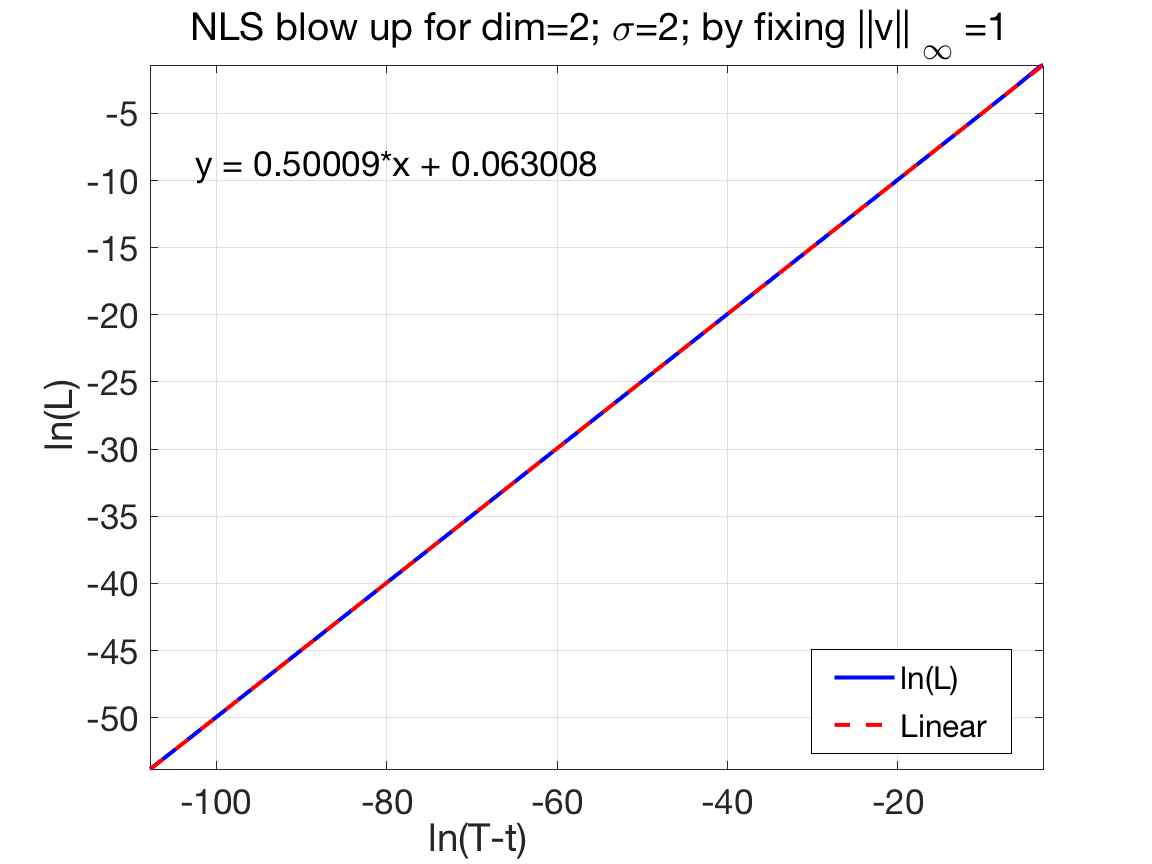}
\includegraphics[width=0.42\textwidth]{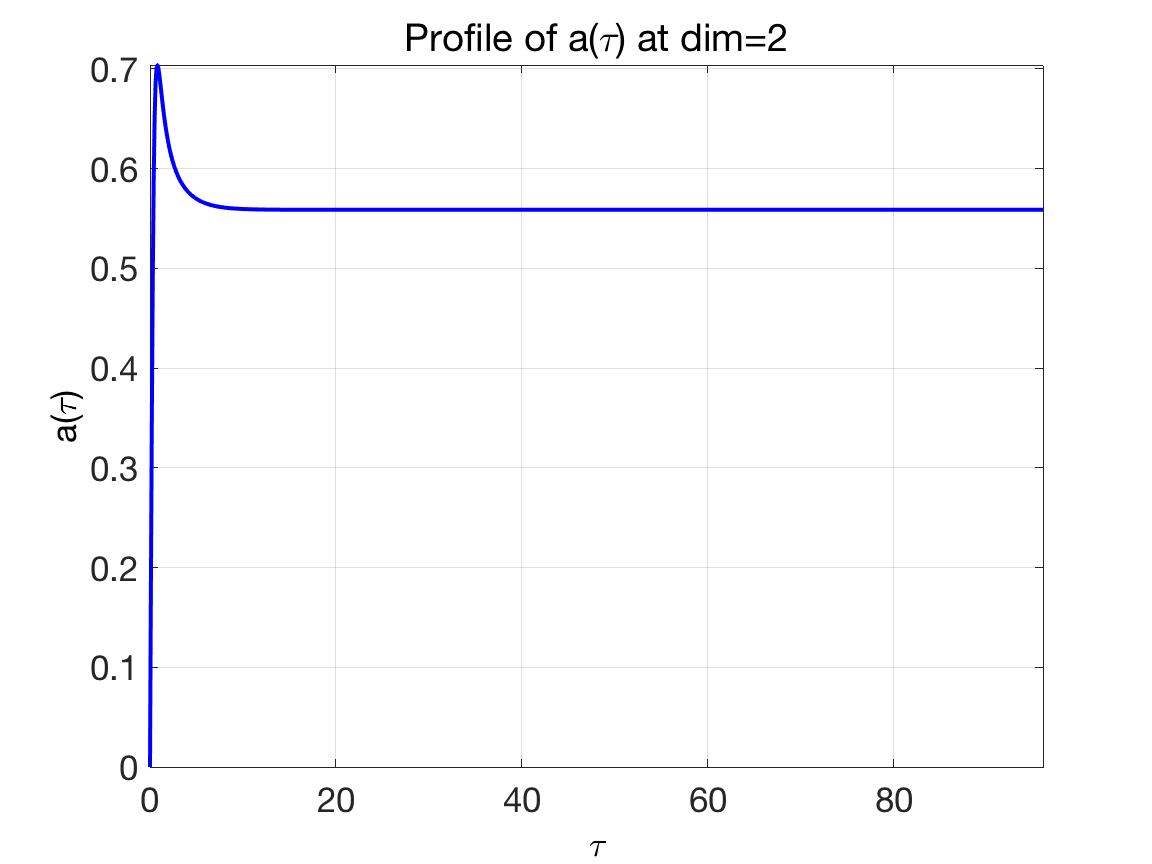}
\includegraphics[width=0.42\textwidth]{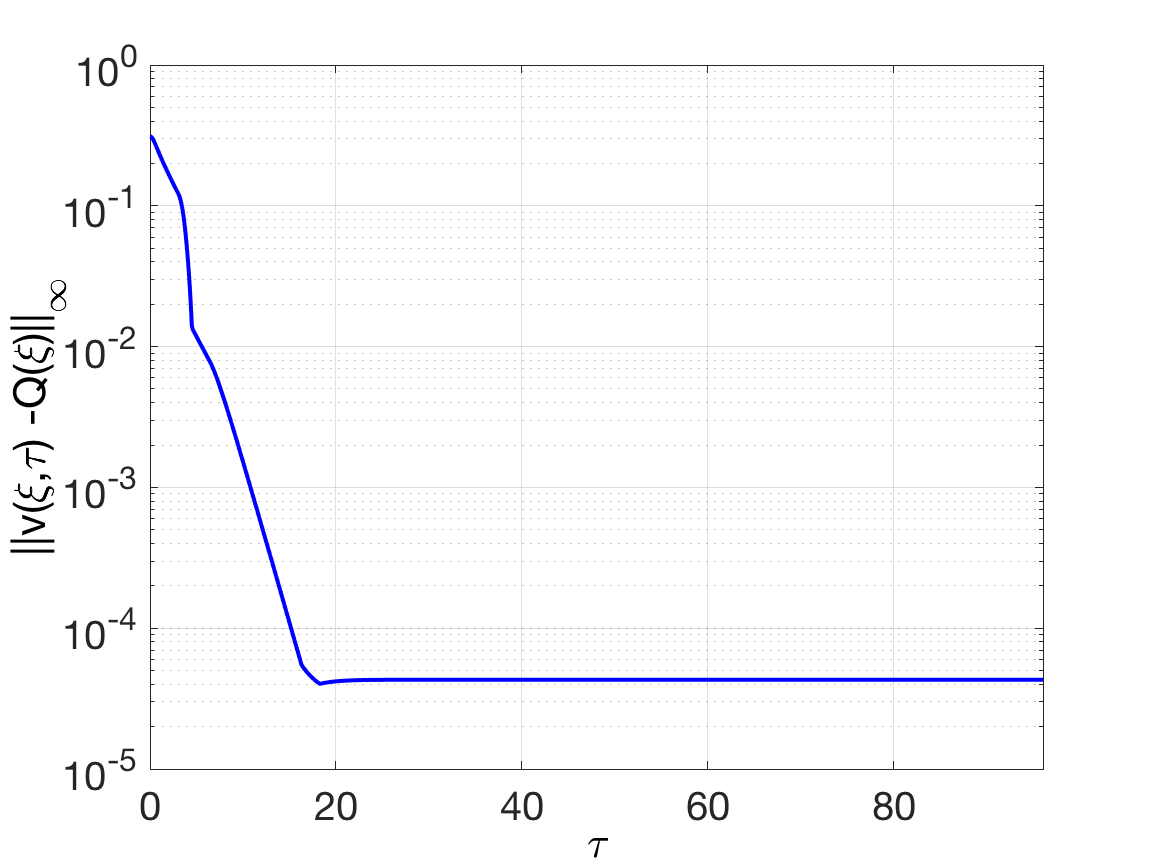}
\includegraphics[width=0.42\textwidth]{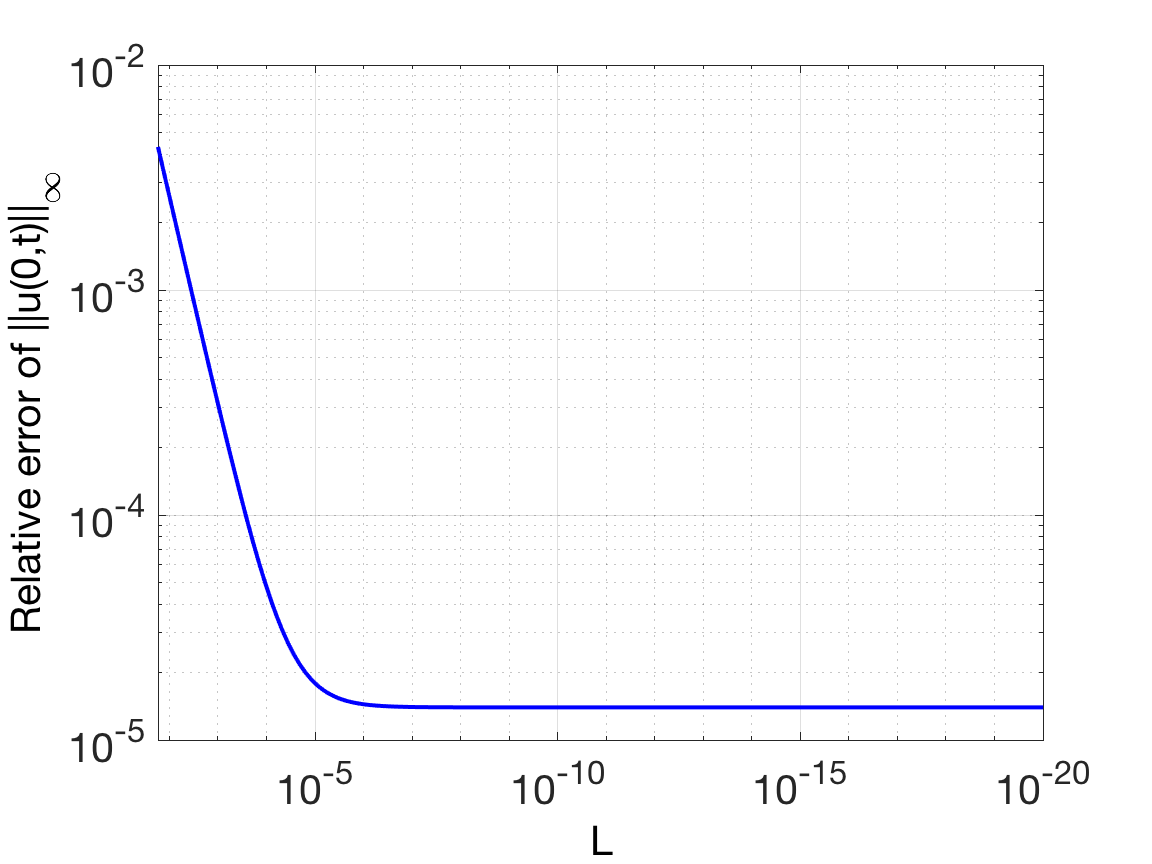}
\caption{Blow-up data for the 2d quintic case: $\ln(T-t)$ vs. $\ln(L)$ (upper left), the quantity $a(\tau)$ (upper right), the distance between $Q$ and $v$ on time $\tau$ ($\| |v(\tau)| -|Q| \|_{L^{\infty}_{\xi}}$) (lower left), the relative error with respect to the predicted blow-up rate (lower right).}
\label{2d5p data} 
\end{center}
\end{figure}

\begin{figure}
\begin{center}
\includegraphics[width=0.42\textwidth]{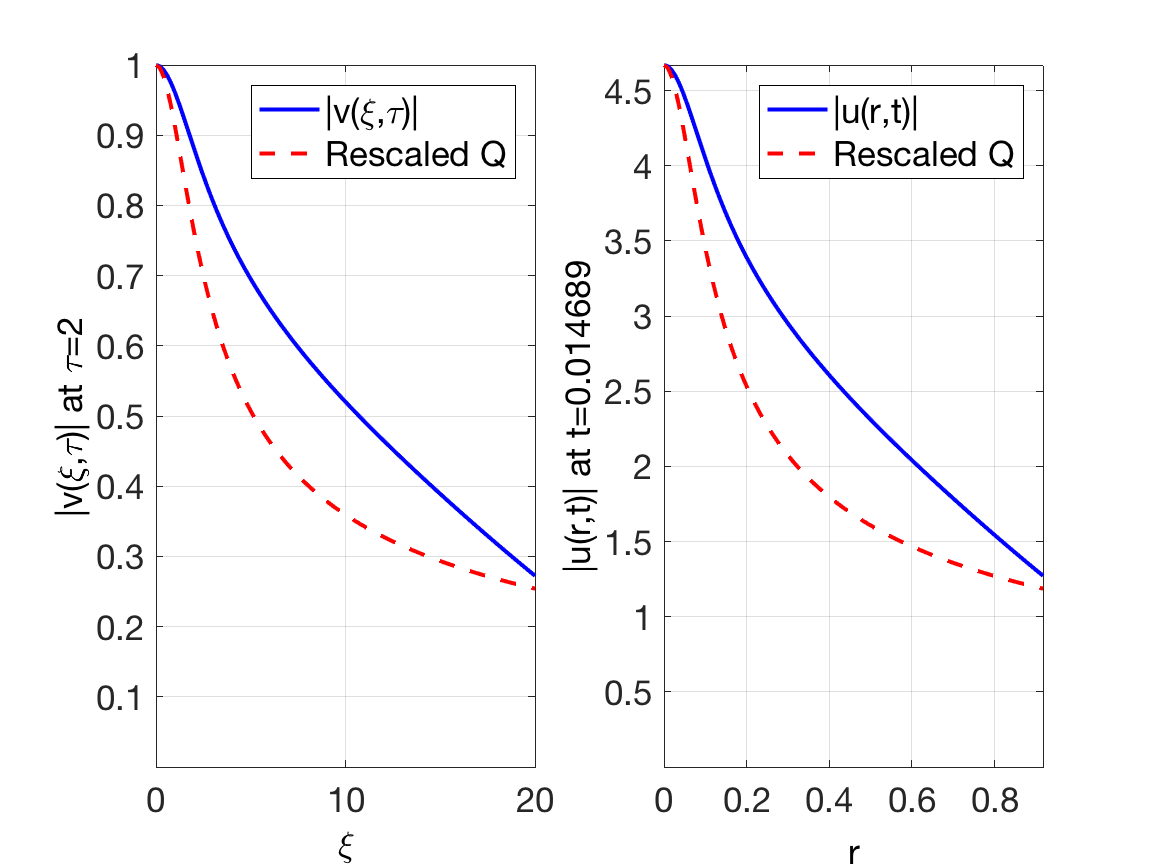}
\includegraphics[width=0.42\textwidth]{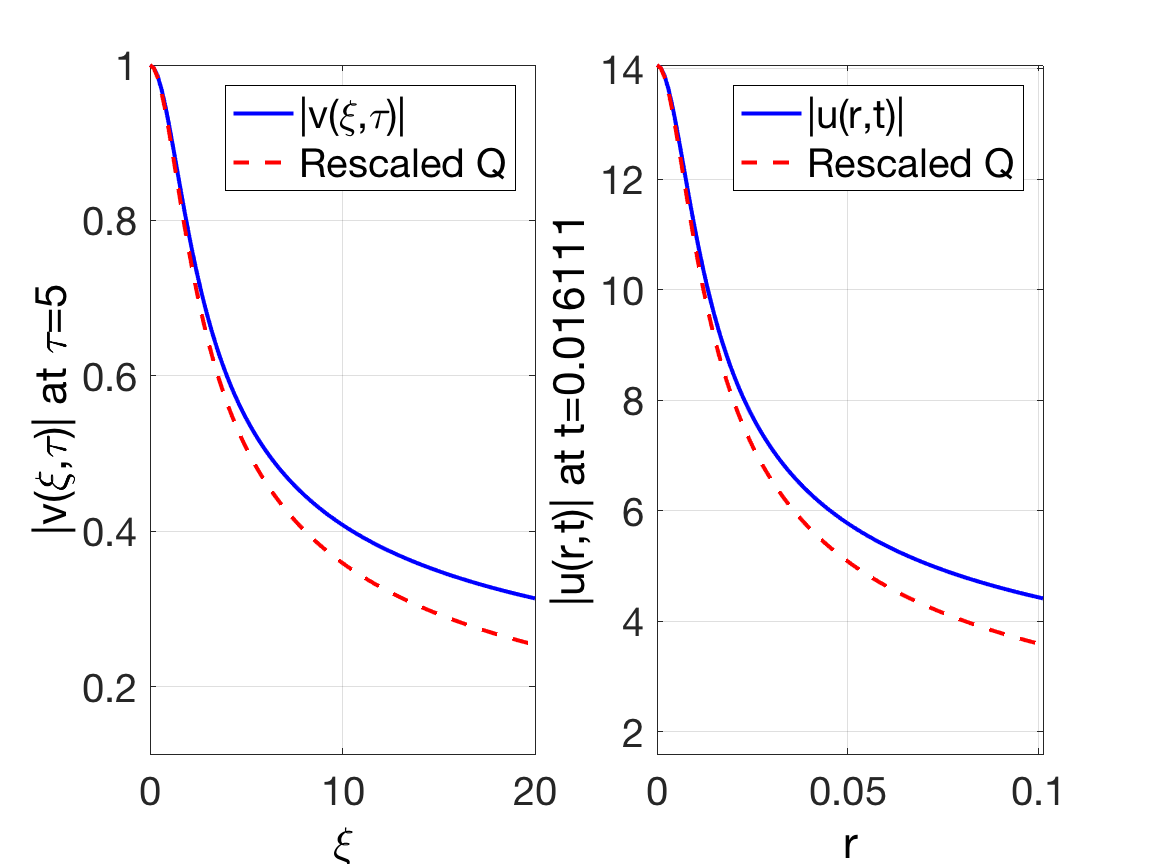}
\includegraphics[width=0.42\textwidth]{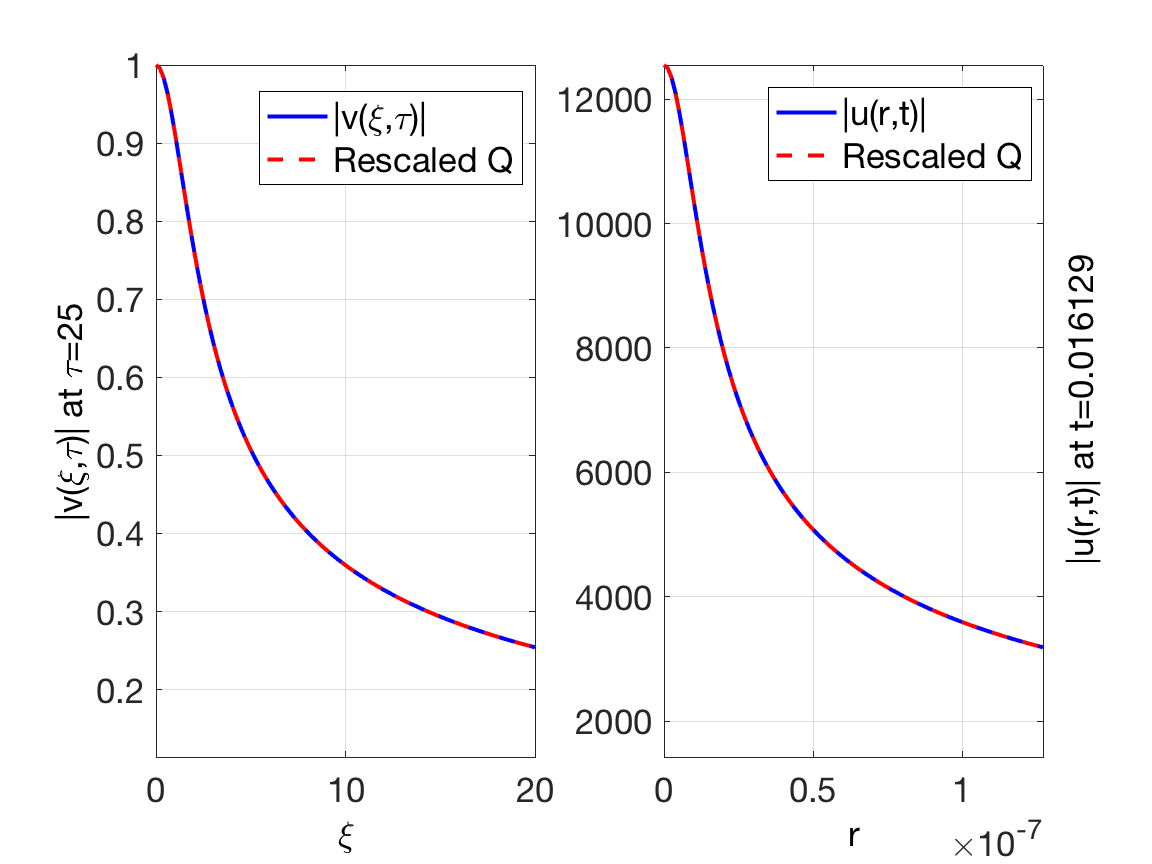}
\includegraphics[width=0.42\textwidth]{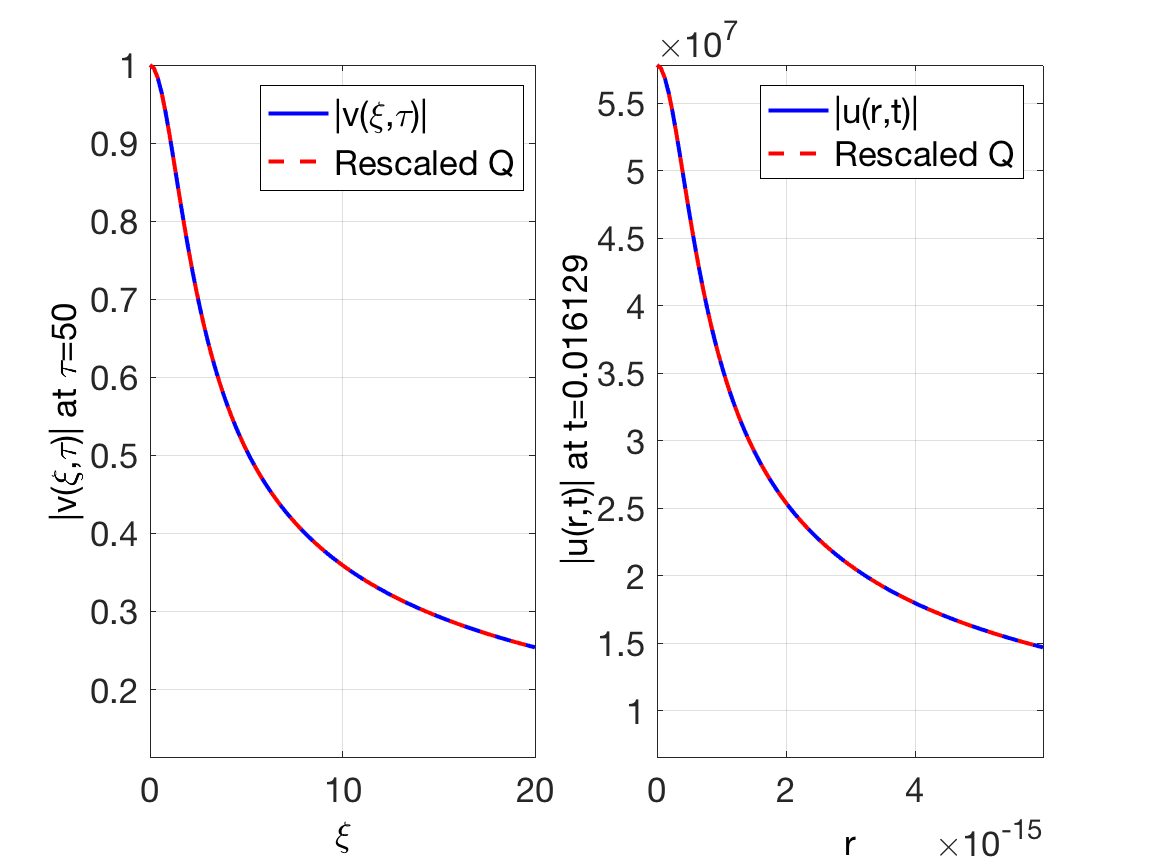}
\caption{ Blow-up profiles for the 3d quintic case at different time $\tau$ and $t$.}
\label{3d5p profiles}
%
\includegraphics[width=0.42\textwidth]{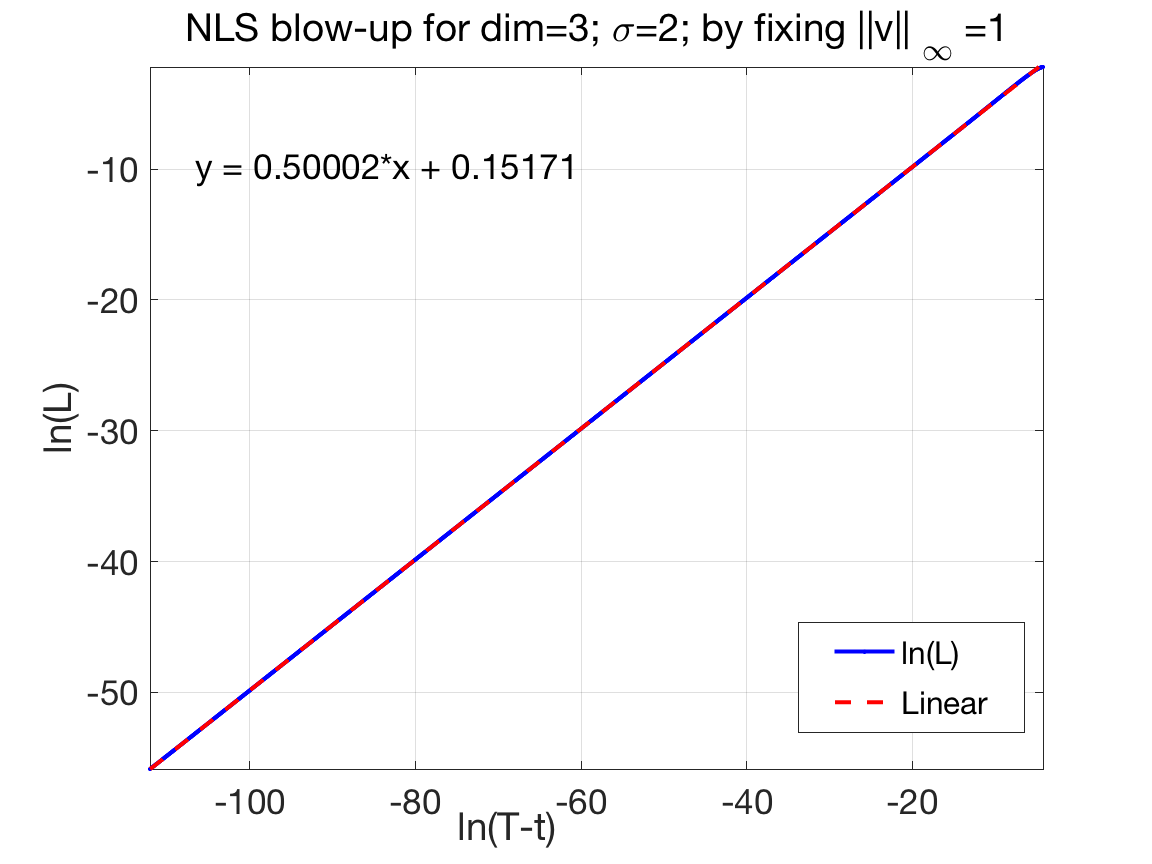}
\includegraphics[width=0.42\textwidth]{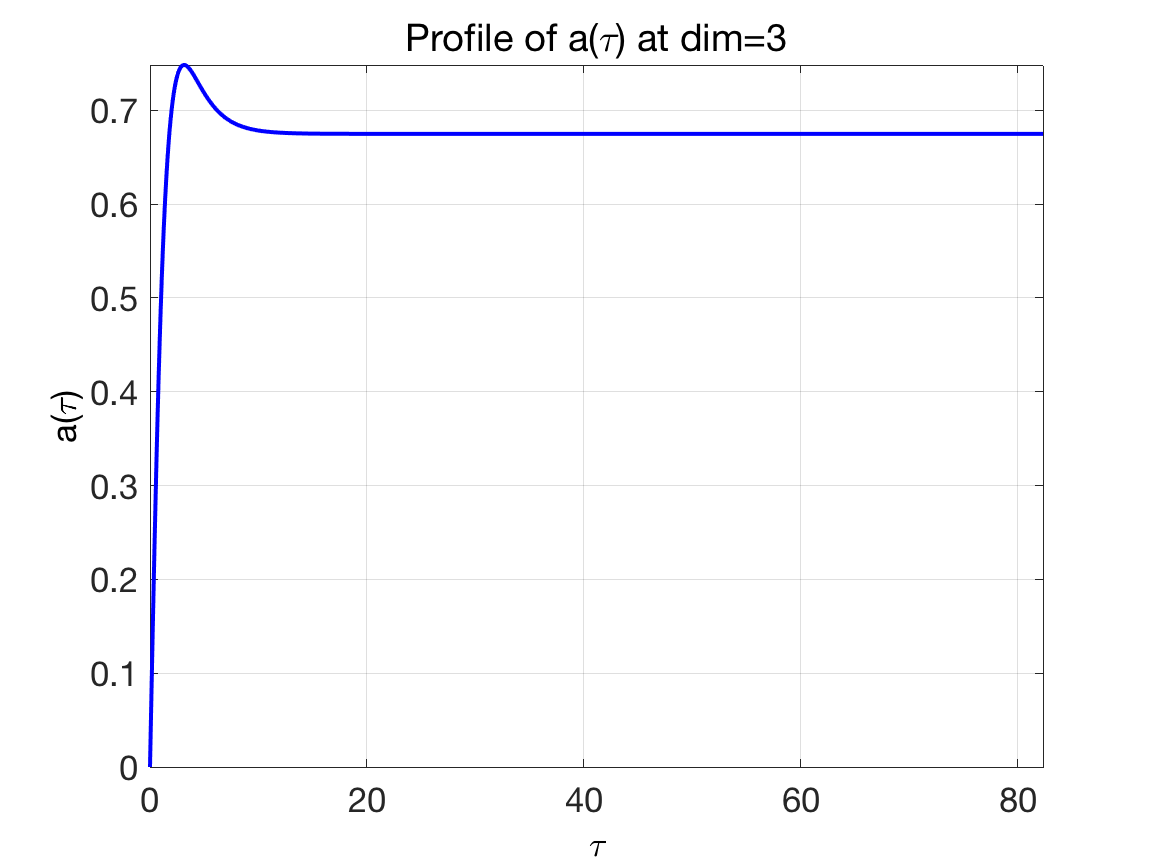}
\includegraphics[width=0.42\textwidth]{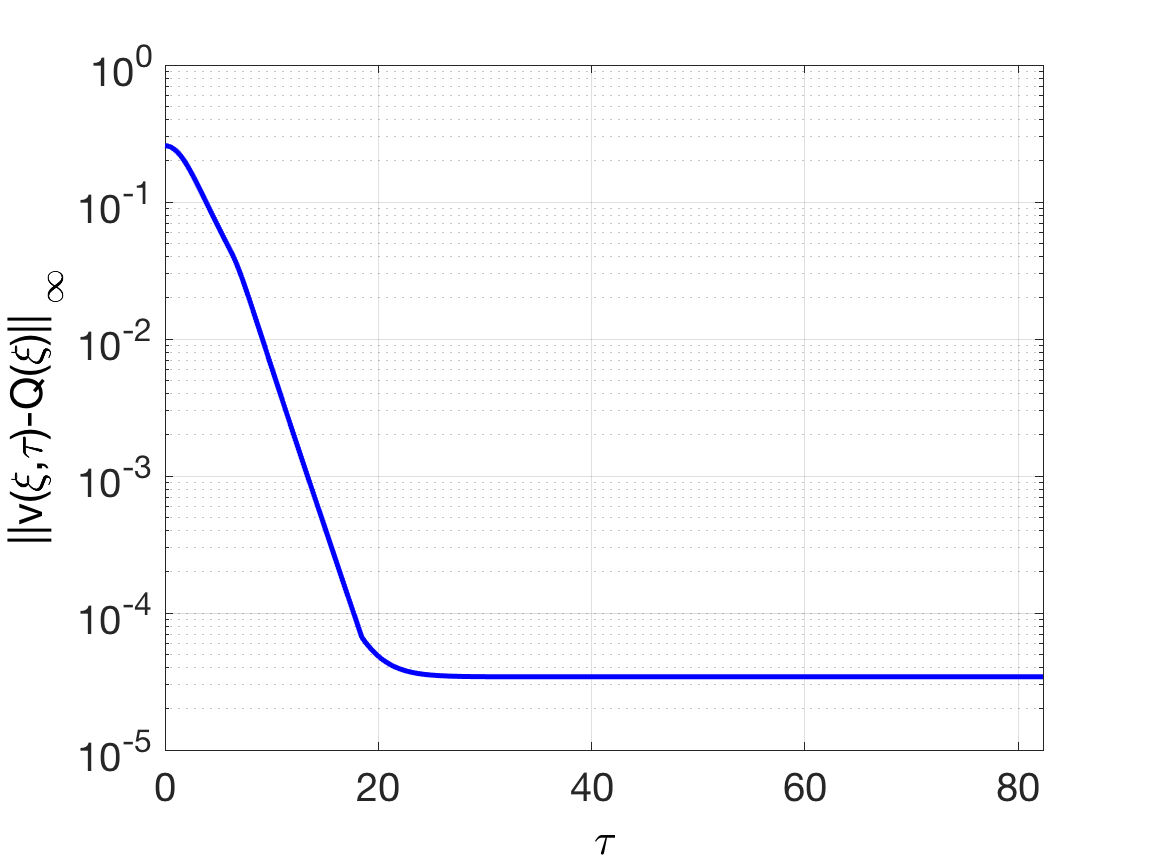}
\includegraphics[width=0.42\textwidth]{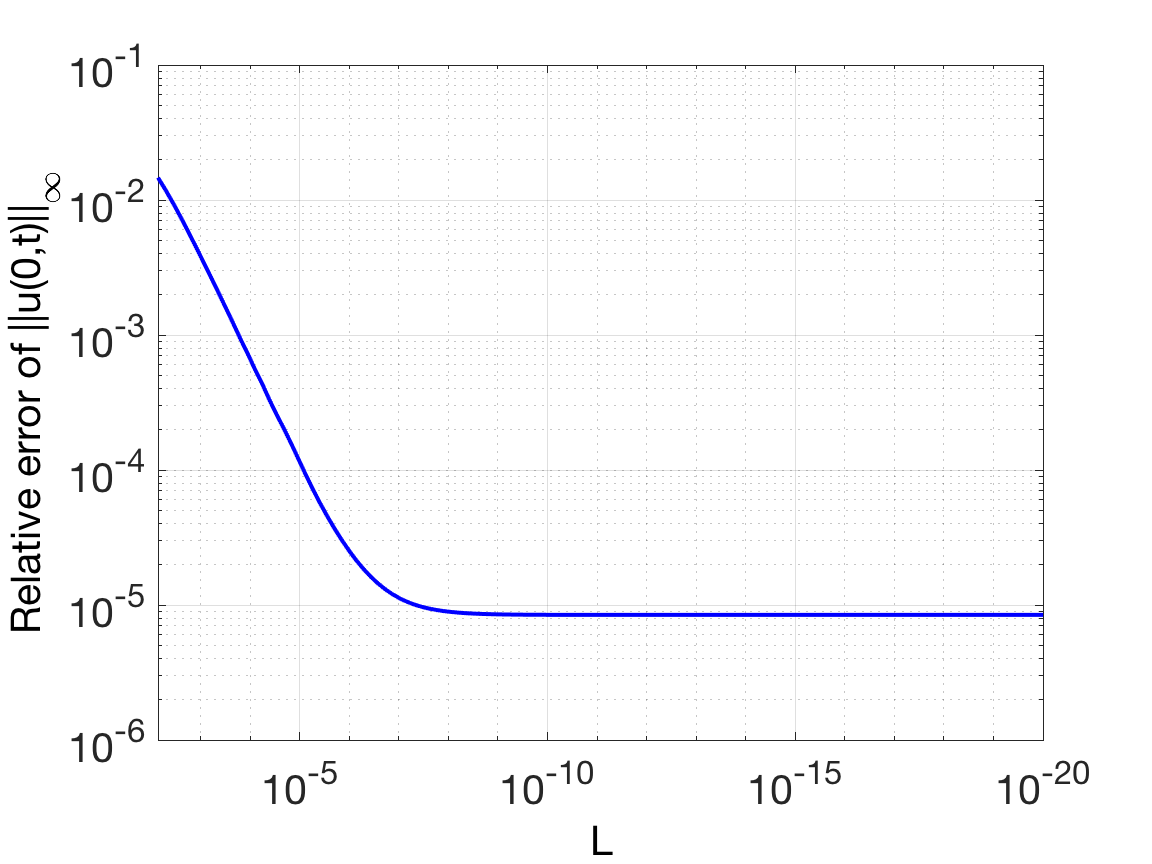}
\caption{ Blow-up data for the 3d quintic case: $\ln(T-t)$ vs. $\ln(L)$ (upper left), the quantity $a(\tau)$ (upper right), the distance between $Q$ and $v$ on time $\tau$ ($\| |v(\tau)| -|Q| \|_{L^{\infty}_{\xi}}$) (lower left), the relative error with respect to the predicted blow-up rate (lower right).  }
\label{3d5p data}
\end{center}
\end{figure}

\begin{figure}
\begin{center}
\includegraphics[width=0.42\textwidth]{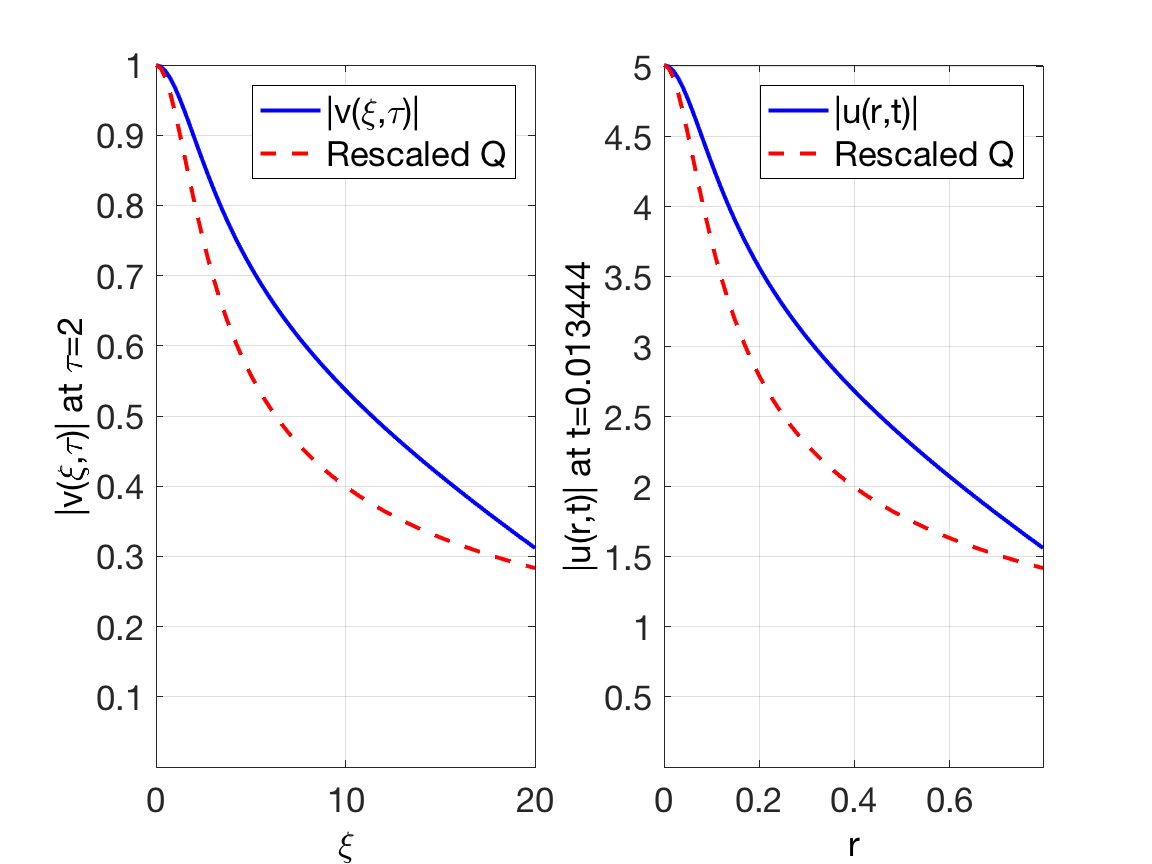}
\includegraphics[width=0.42\textwidth]{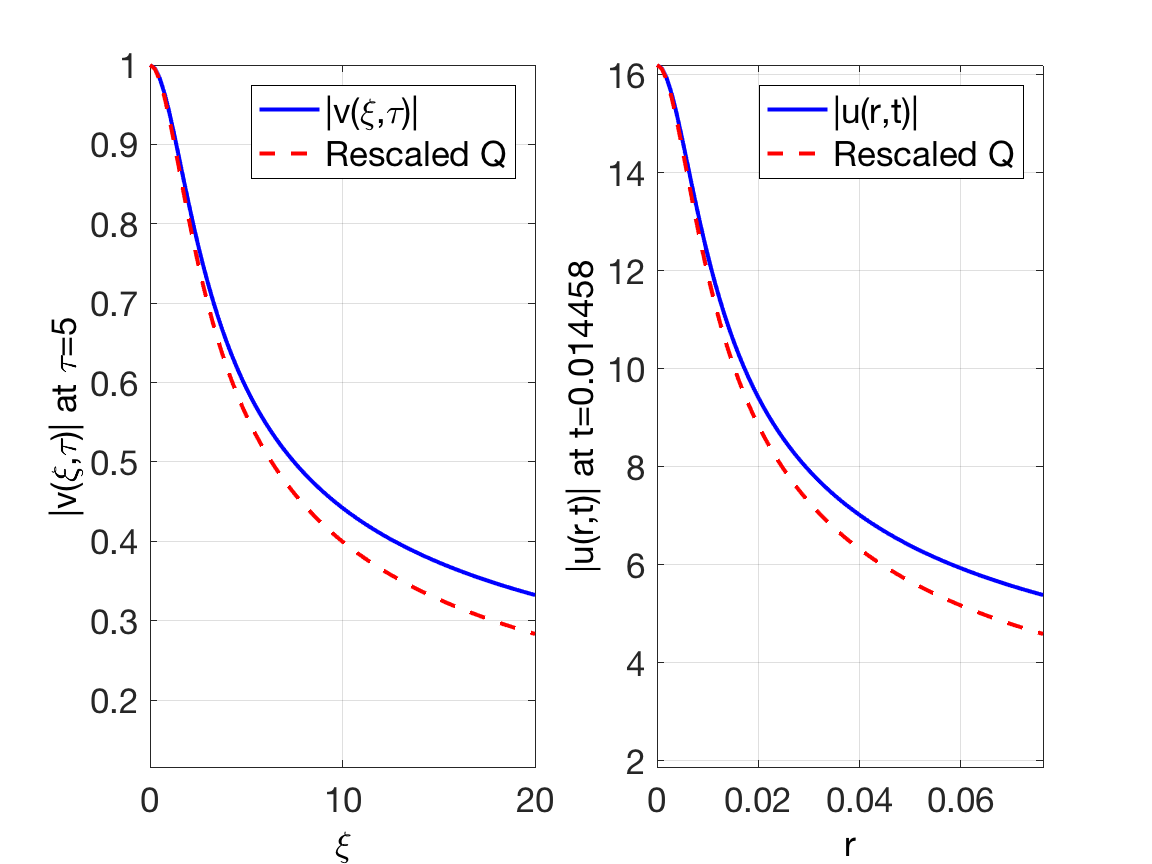}
\includegraphics[width=0.42\textwidth]{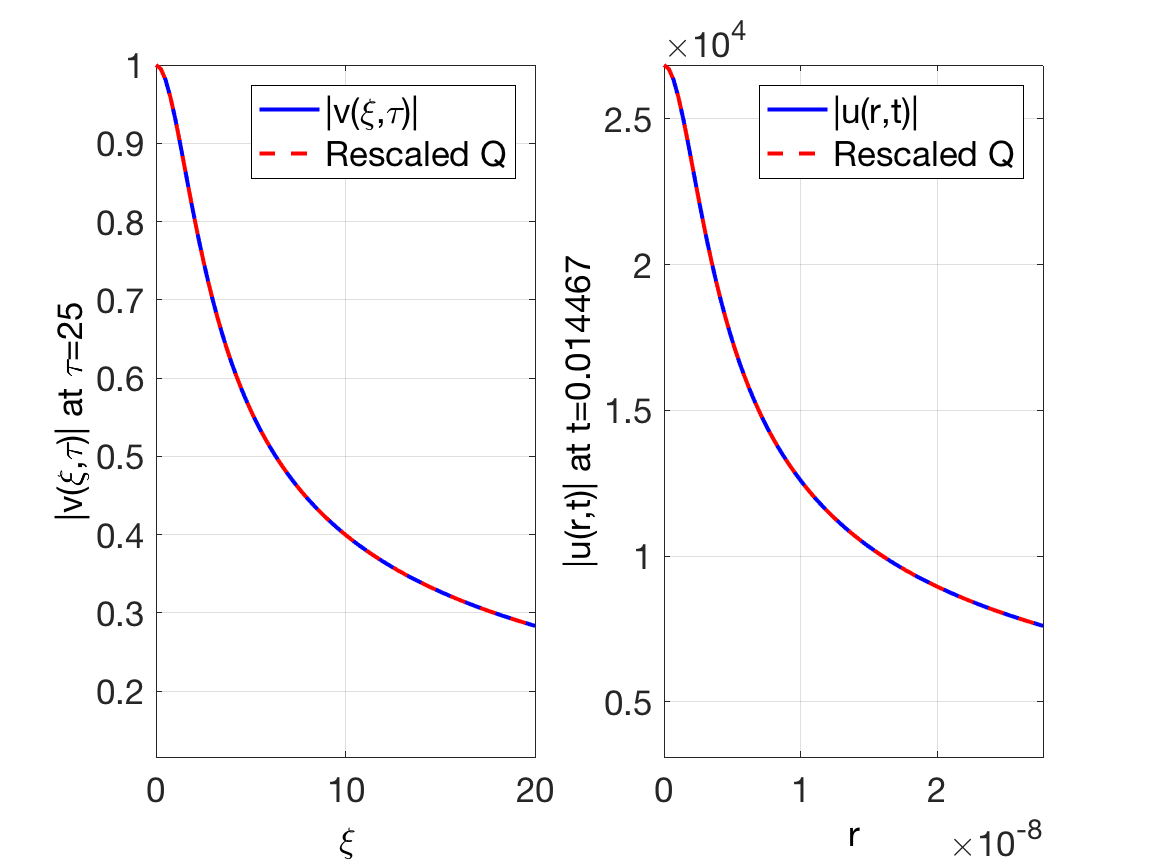}
\includegraphics[width=0.42\textwidth]{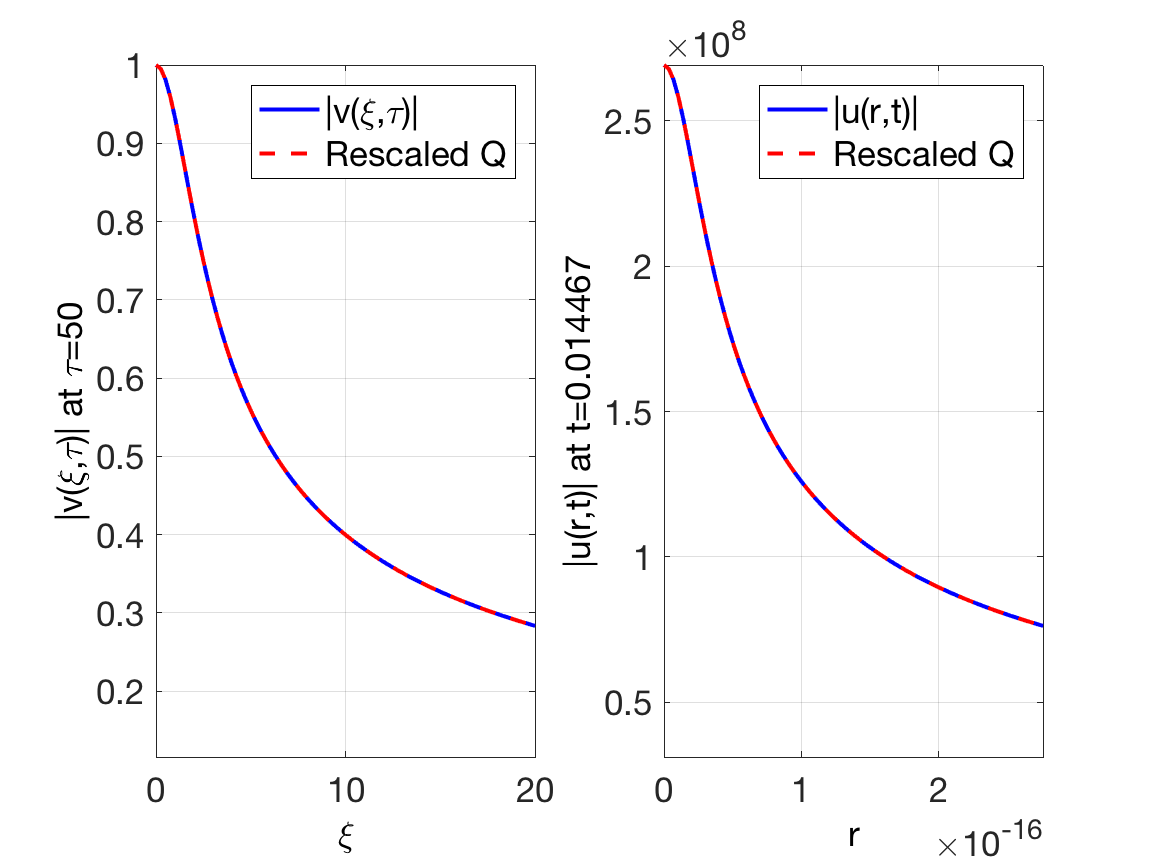}
\caption{ Blow-up profiles for the 4d quintic case at different time $\tau$ and $t$.}
\label{4d5p profiles}
%
\includegraphics[width=0.42\textwidth]{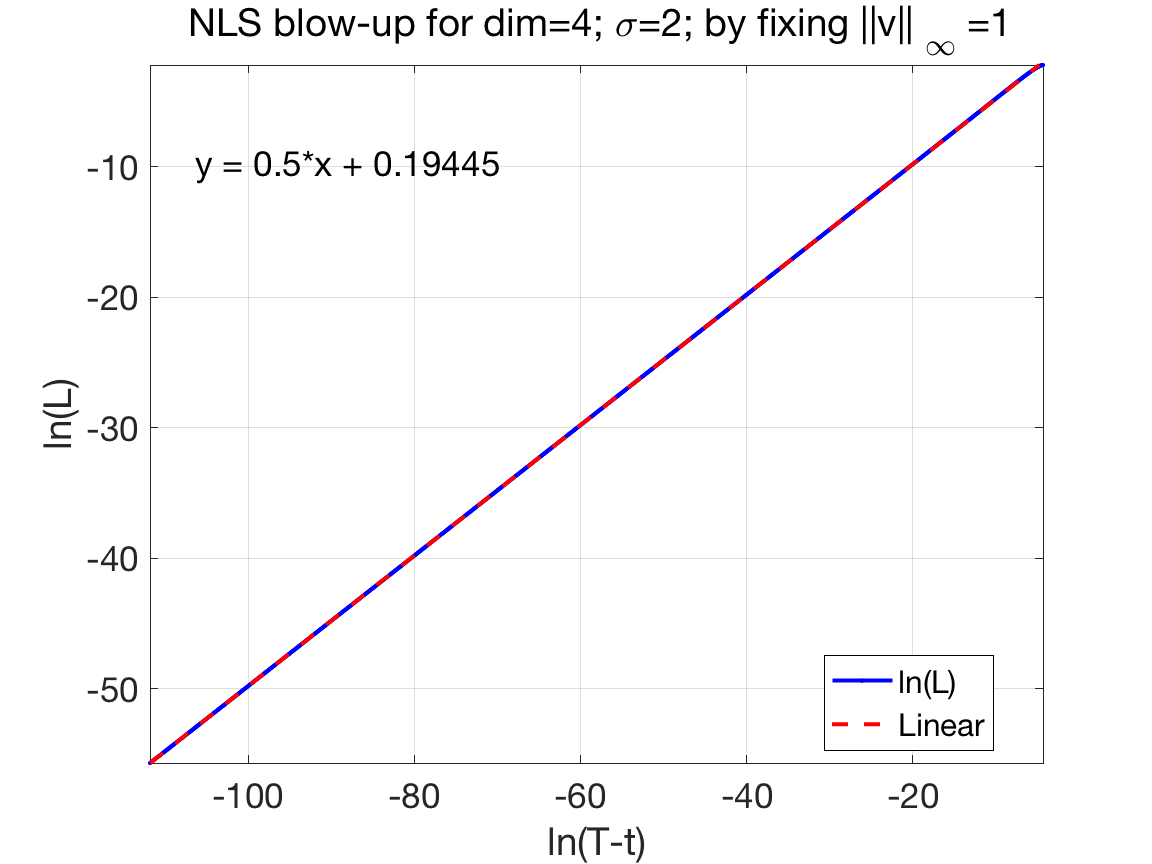}
\includegraphics[width=0.42\textwidth]{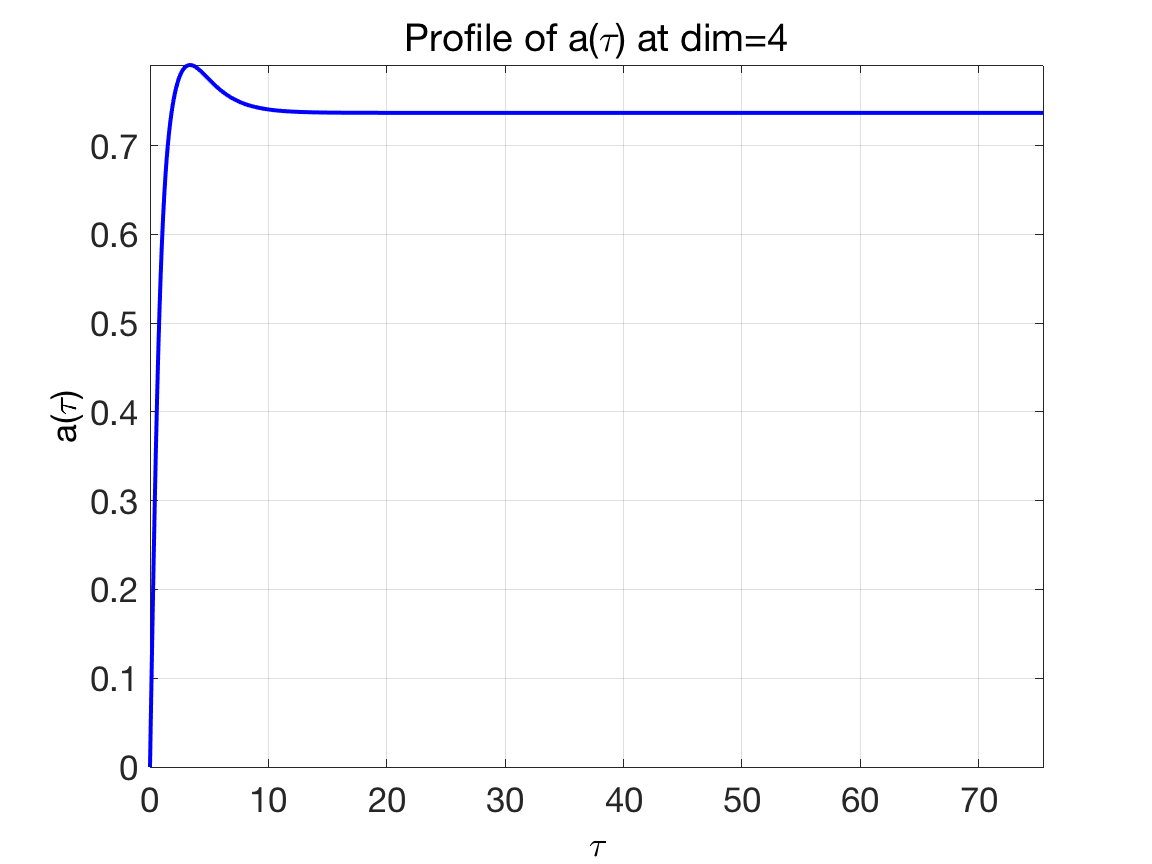}
\includegraphics[width=0.42\textwidth]{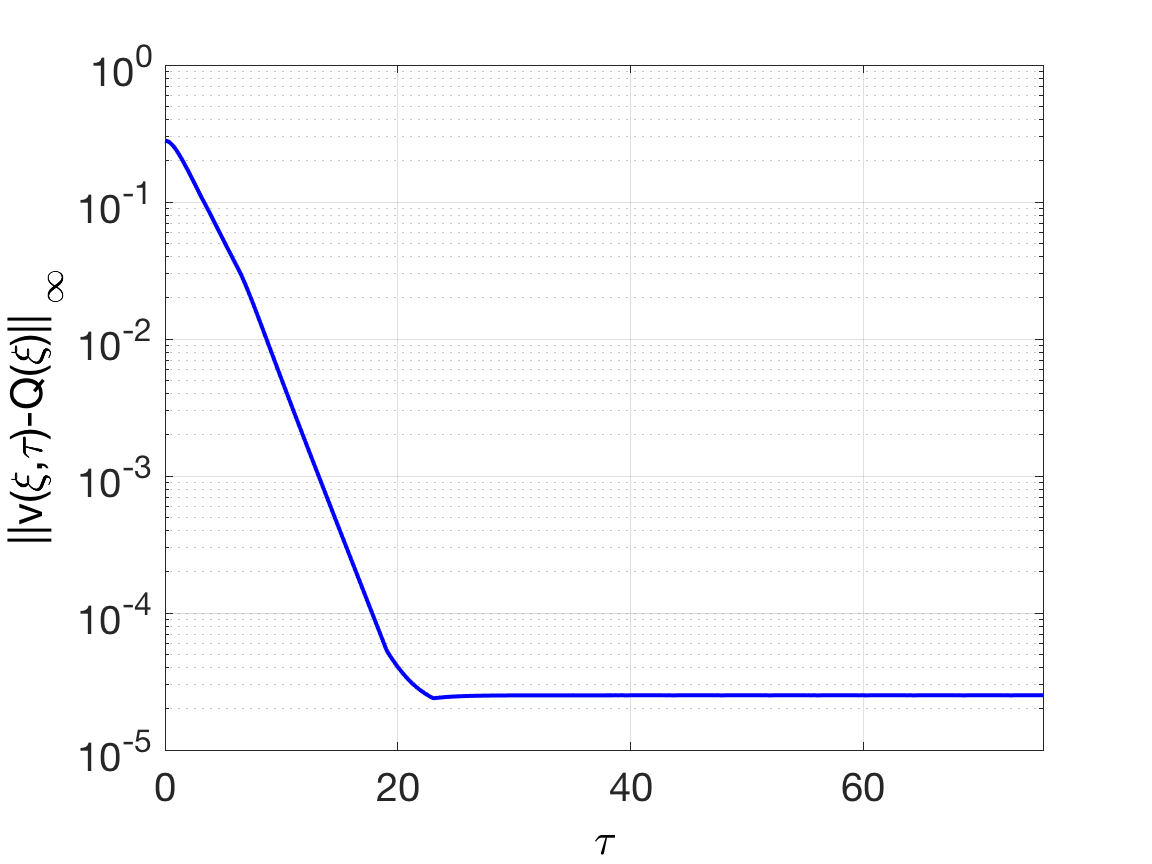}
\includegraphics[width=0.42\textwidth]{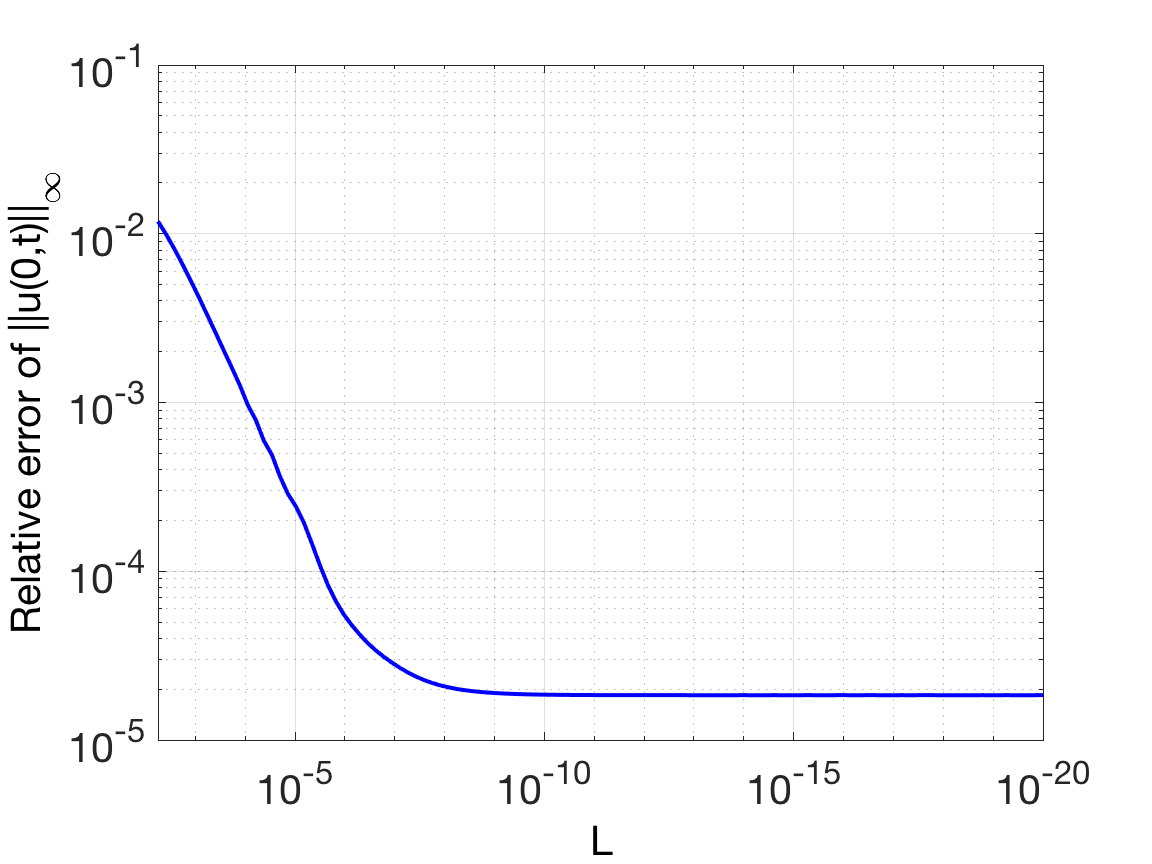}
\caption{ Blow-up data for the 4d quintic case: $\ln(T-t)$ vs. $\ln(L)$ (upper left), the quantity $a(\tau)$ (upper right), the distance between $Q$ and $v$ on time $\tau$ ($\| |v(\tau)| -|Q| \|_{L^{\infty}_{\xi}}$) (lower left), the relative error with respect to the predicted blow-up rate (lower right).  }
\label{4d5p data}
\end{center}
\end{figure}

\begin{figure}
\begin{center}
\includegraphics[width=0.42\textwidth]{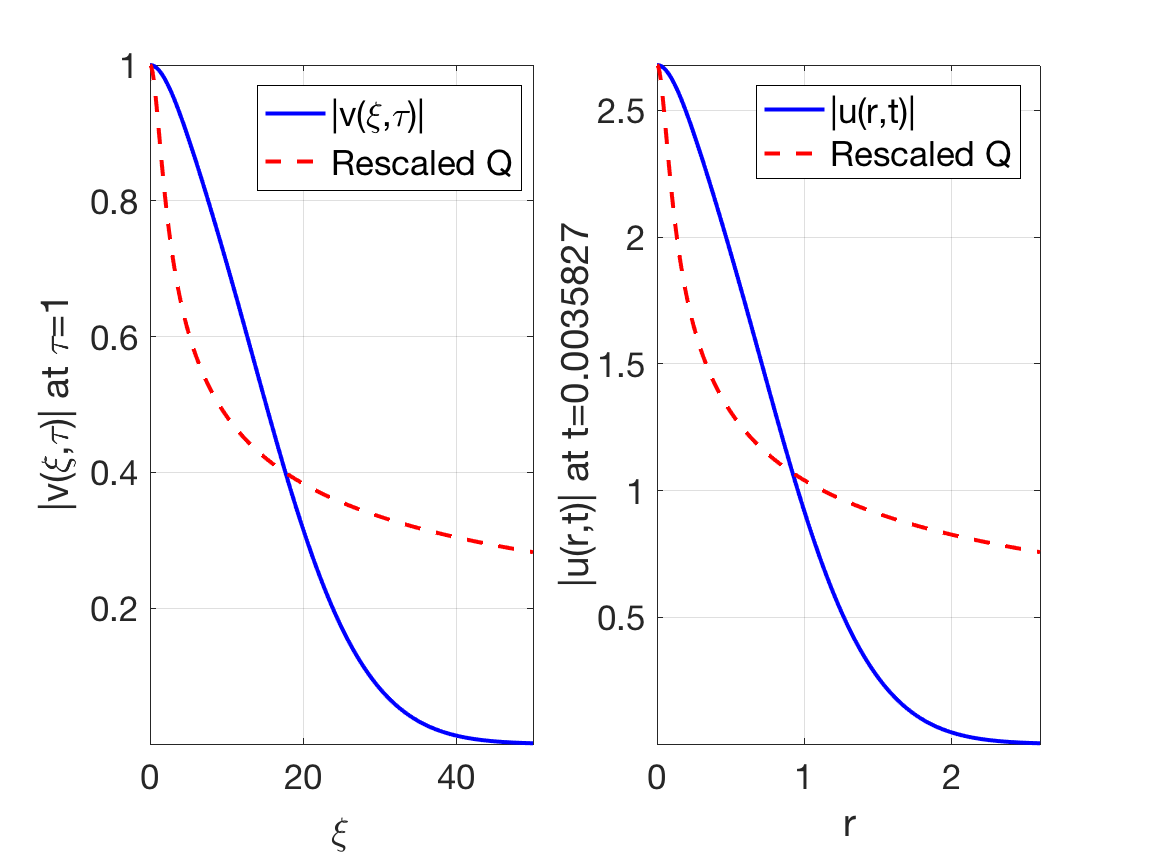}
\includegraphics[width=0.42\textwidth]{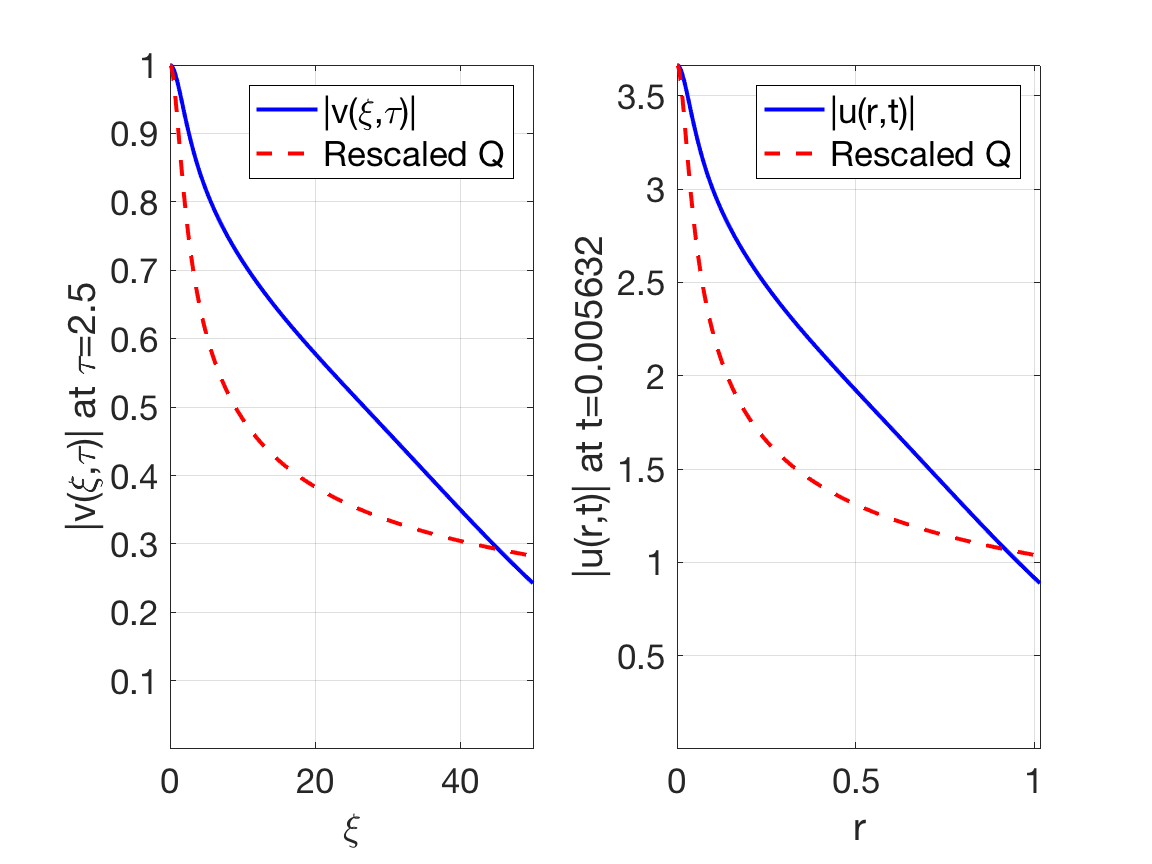}
\includegraphics[width=0.42\textwidth]{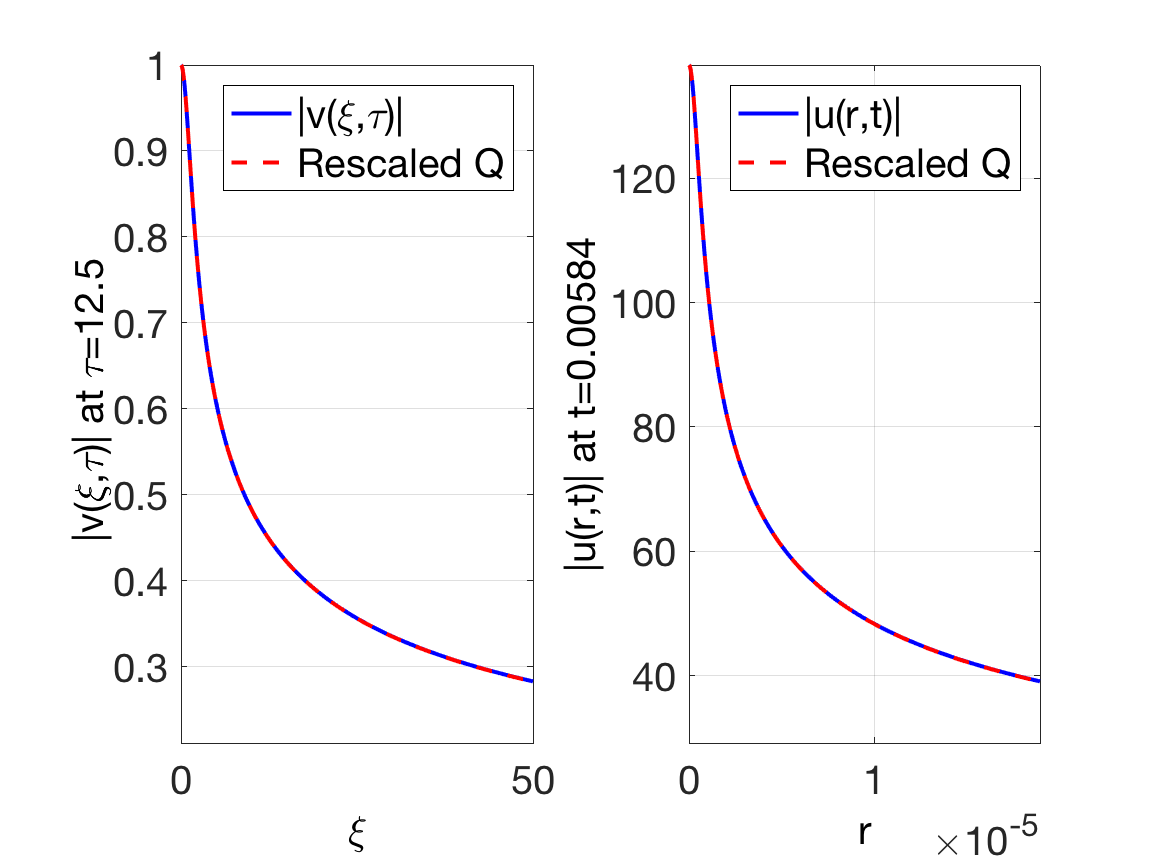}
\includegraphics[width=0.42\textwidth]{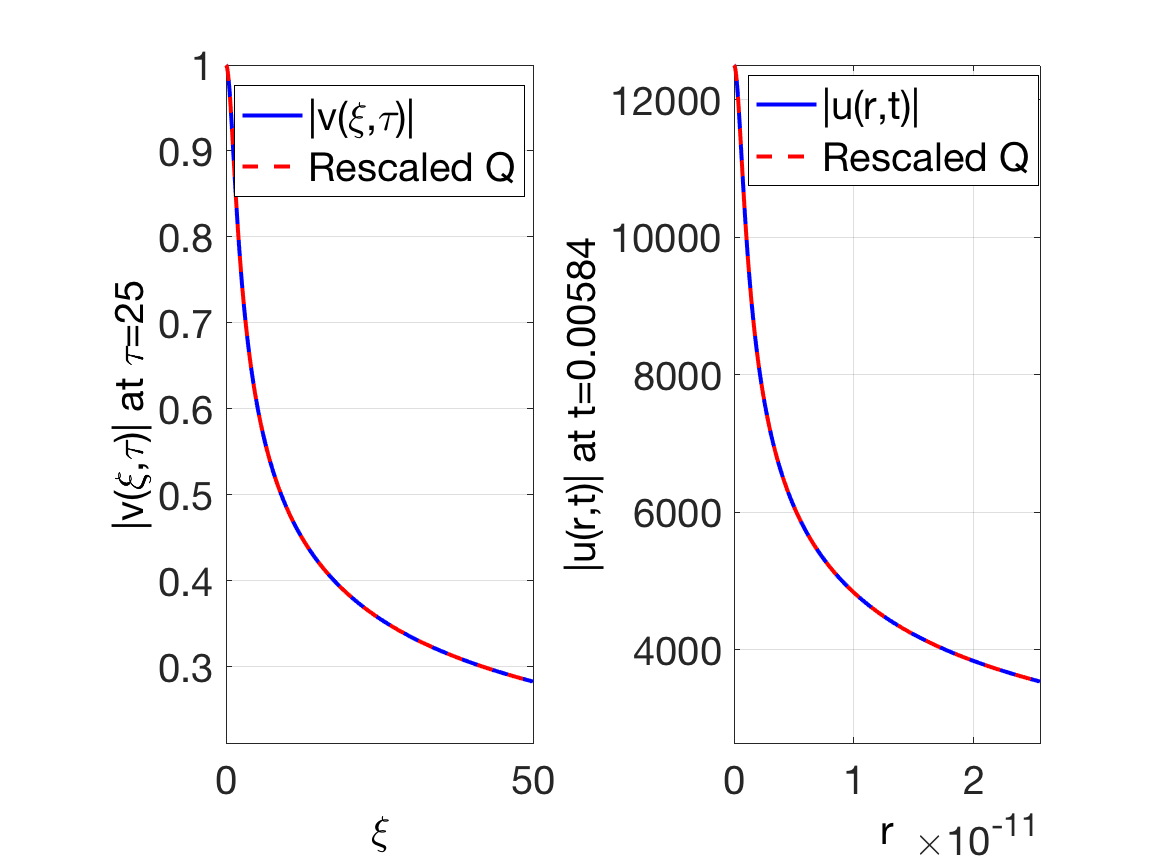}
\caption{ Blow-up profiles for the 3d septic case at different time $\tau$ and $t$.}
\label{3d7p profiles}
%
\includegraphics[width=0.42\textwidth]{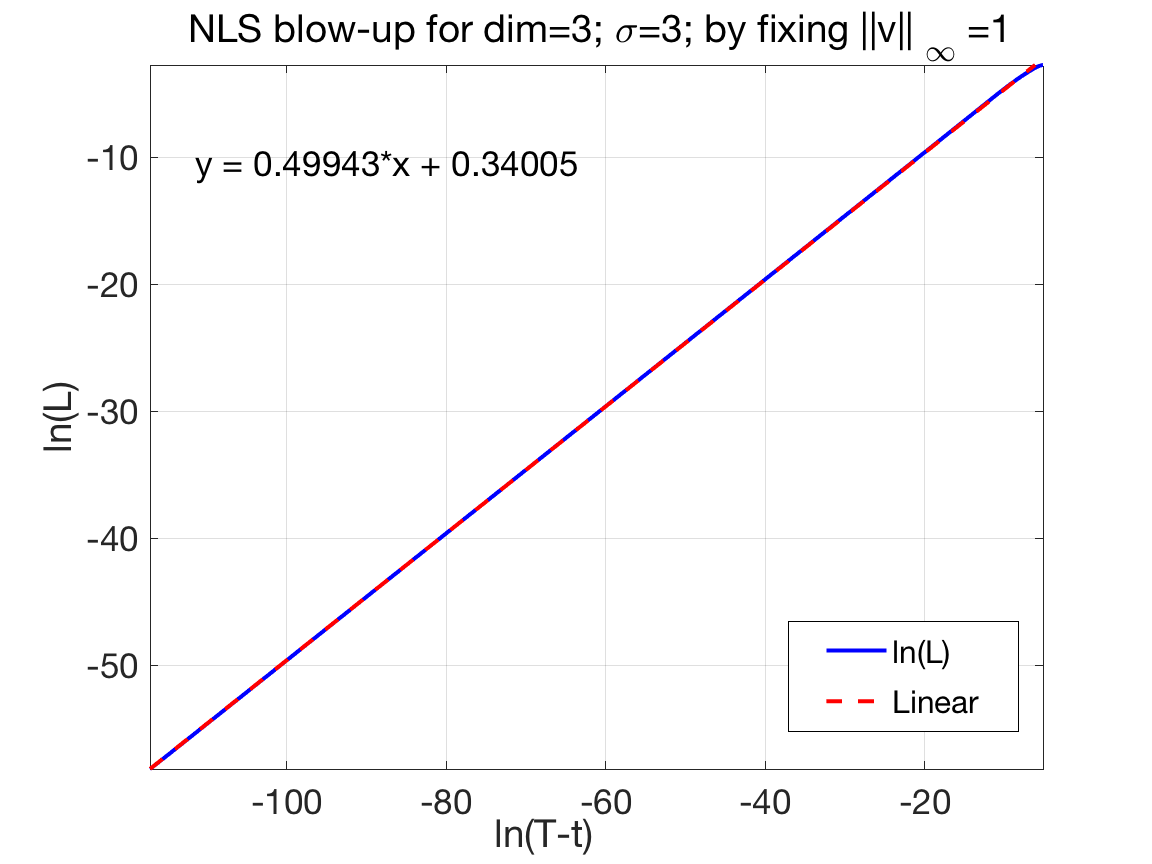}
\includegraphics[width=0.42\textwidth]{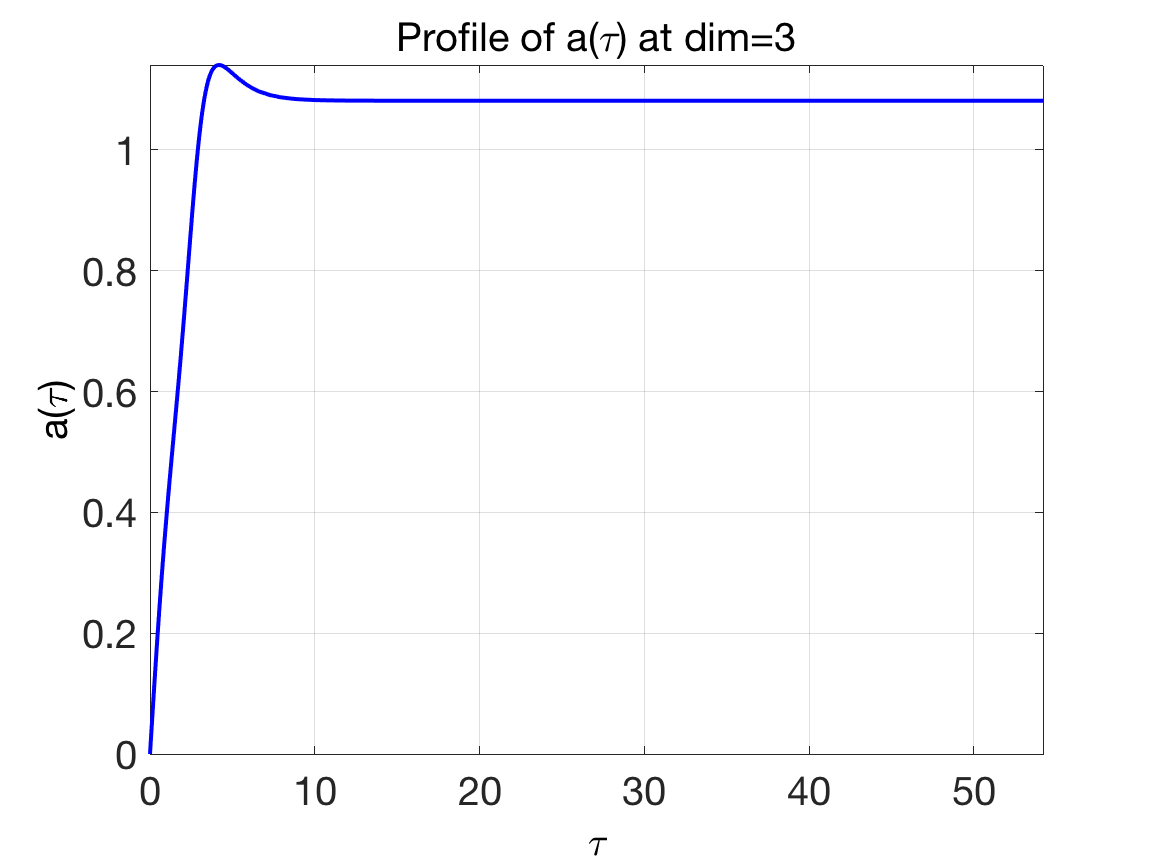}
\includegraphics[width=0.42\textwidth]{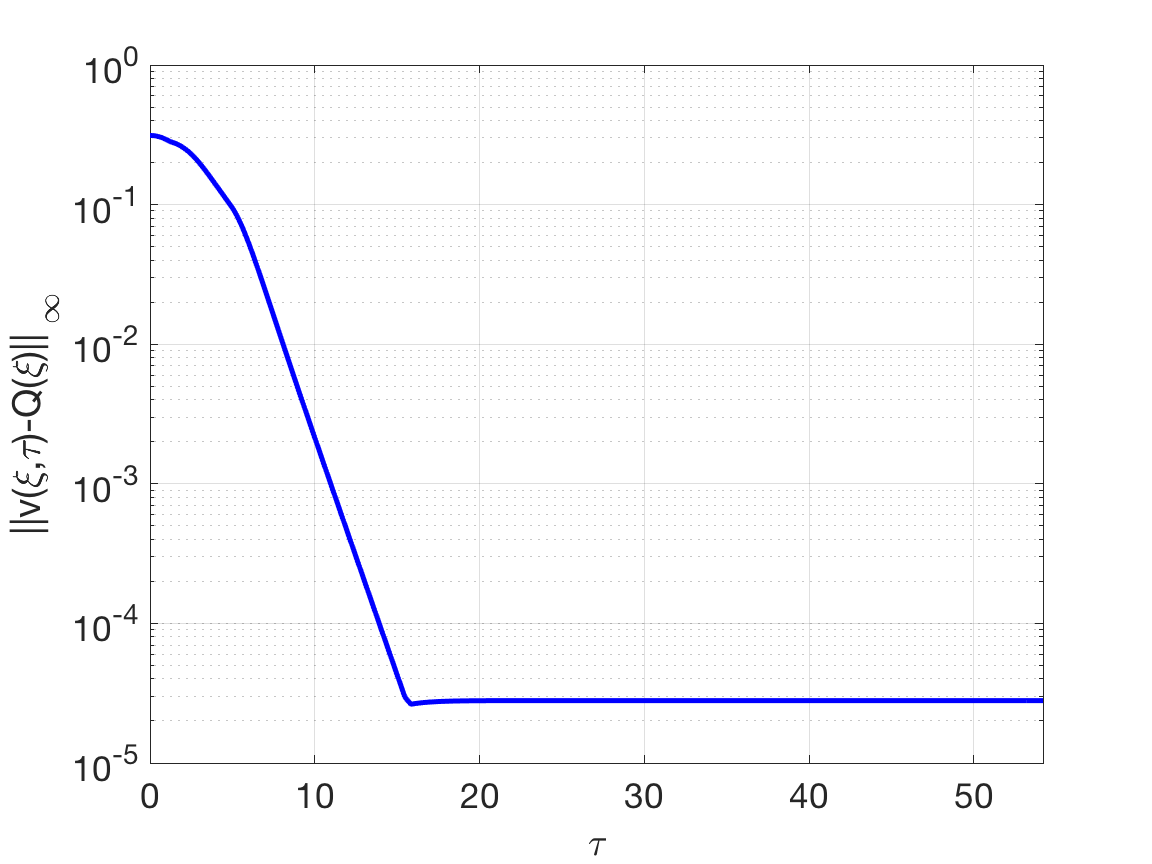}
\includegraphics[width=0.42\textwidth]{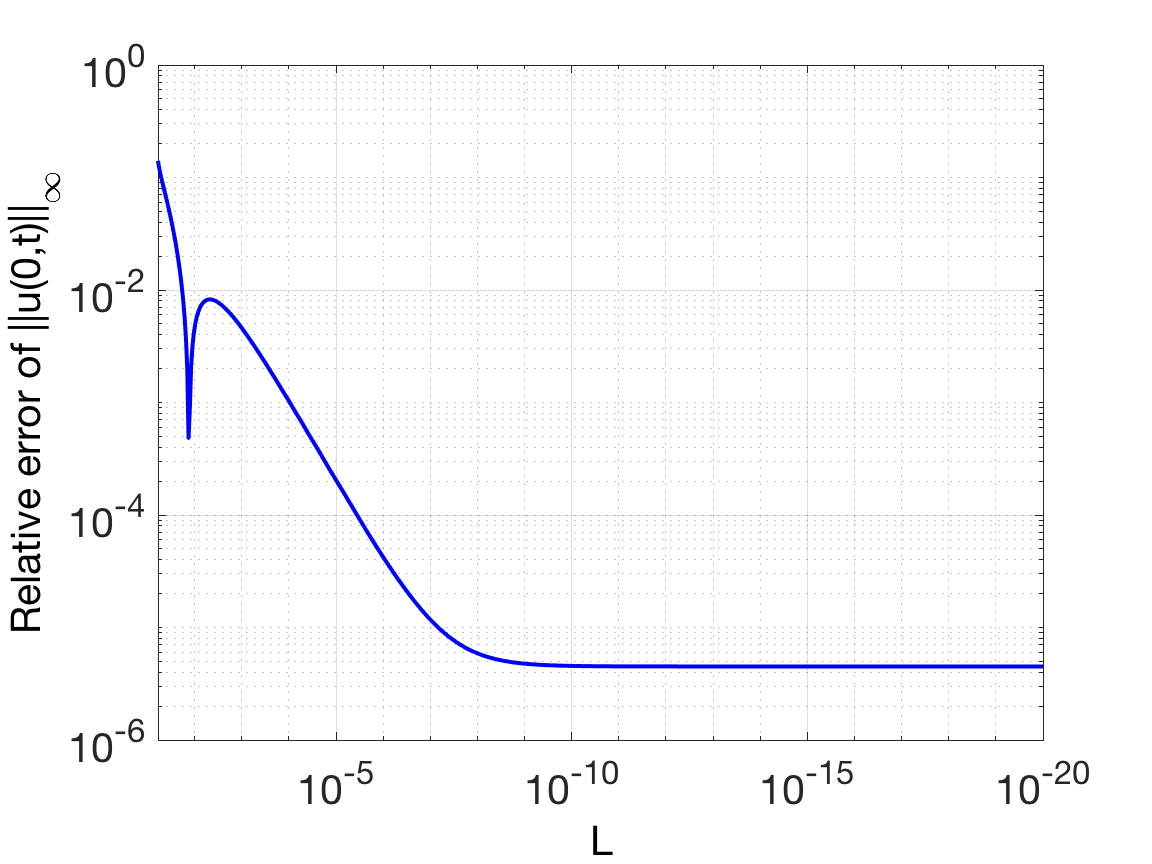}
\caption{ Blow-up data for the 3d septic case: $\ln(T-t)$ vs. $\ln(L)$ (upper left), the quantity $a(\tau)$ (upper right), the distance between $Q$ and $v$ on time $\tau$ ($\| |v(\tau)| -|Q| \|_{L^{\infty}_{\xi}}$) (lower left), the relative error with respect to the predicted blow-up rate (lower right).  }
\label{3d7p data}
\end{center}
\end{figure}




\clearpage

\bibliography{Kai_bib}
\bibliographystyle{abbrv}

\end{document}